\documentclass{amsart}
\usepackage[utf8x]{inputenc}
\usepackage{amsmath, amsthm, amssymb}
\usepackage[usenames,dvipsnames,svgnames,table]{xcolor}
\usepackage[margin=1in]{geometry}
\usepackage[colorlinks,linkcolor = blue,citecolor = ForestGreen]{hyperref}
    \makeatletter
        \renewcommand*{\eqref}[1 ]{\hyperref[#1]{\textup{\tagform@{\ref*{#1}}}}}
    \makeatother

\usepackage[noadjust]{cite}
\usepackage{xfrac}

\newtheorem{theorem}{Theorem}[section]
\newtheorem{proposition}[theorem]{Proposition}
\newtheorem{lemma}[theorem]{Lemma}

\usepackage{zref-clever}
	\newcommand{\Cref}[1]{\zcref[S]{#1}}

\AddToHook{env/theorem/begin}
    {\zcsetup{countertype={num=theorem}}}
\AddToHook{env/proposition/begin}
    {\zcsetup{countertype={theorem=proposition}}}
\AddToHook{env/lemma/begin}
	{\zcsetup{countertype={theorem=lemma}}}
\AddToHook{env/definition/begin}
	{\zcsetup{countertype={theorem=definition}}}
\AddToHook{env/remark/begin}
	{\zcsetup{countertype={theorem=remark}}}
\AddToHook{env/corollary/begin}
	{\zcsetup{countertype={theorem=corollary}}}

\newcommand{\R}{\mathbb R}

\newcommand{\cP}{\mathcal P}
\newcommand{\eps}{\varepsilon}

\newcommand{\dd}{\, \mathrm{d}}

\newcommand{\tr}{\mbox{tr}}

\DeclareMathOperator{\Id}{Id}

\newcommand{\RR}{\langle R \rangle}

\renewcommand{\phi}{\varphi}

\usepackage{mathtools}
	\mathtoolsset{showonlyrefs}

\newcommand{\QRL}{{\mathcal Q}_{\rm RL}}

\newcommand{\pp}{\langle p\rangle}
\newcommand{\ppo}{\langle p_0\rangle}

\newcommand{\qq}{\langle q \rangle}

\newcommand{\cG}{\mathcal{G}}

\newcommand{\be}{\begin{equation}}
\newcommand{\ee}{\end{equation}}

\numberwithin{equation}{section}
\numberwithin{theorem}{section}
\title[Relativistic Landau equation]
{Global regularity and decay estimates for the relativistic Landau equation}

\author{Christopher Henderson}
\address{Department of Mathematics, University of Maryland, College Park, MD 20742}
\email{ckhend@umd.edu}

\author{Stanley Snelson}
\address{Department of Mathematics and Systems Engineering, Florida Institute of Technology, Melbourne, FL 32901}
\email{ssnelson@fit.edu}

\author{Andrei Tarfulea}
\address{Department of Mathematics, Louisiana State University, Baton Rouge, LA 70803}
\email{tarfulea@lsu.edu}

\author{Maja Taskovi\'c}
\address{Department of Mathematics, Emory University, Atlanta, GA 30322}
\email{maja.taskovic@emory.edu}

\thanks{CH was supported by NSF grants DMS-2204615, 2337666, and 2617615. SS was supported by NSF grant DMS-213407. AT was supported by NSF grants DMS-2012333 and DMS-2408163. MT was supported by NSF grant DMS-2206187.}

\begin{document}

\maketitle

%\vspace{-24pt}

\begin{abstract}
We consider the relativistic Landau equation in the spatially inhomogeneous, far-from-equilibrium regime. We establish regularity estimates of all orders, implying that solutions remain smooth for as long as some zeroth-order conditional bounds hold. We also prove that polynomial and exponential decay in  the momentum variable is propagated forward in time. 

As part of our proof, we establish a Schauder estimate for linear relativistic kinetic equations, that may be of independent interest. 
\end{abstract}

%\tableofcontents

\section{Introduction}

We are concerned with the relativistic Landau equation, which is a nonlinear kinetic model from plasma physics. It was introduced by Budker and Beliaev in 1956 \cite{belyaev-budker1} as a modification of the  widely-studied (classical) Landau equation, which ignores relativistic effects. These relativistic effects become important when particle speeds approach the speed of light, which happens frequently in hot plasmas. 

With units chosen so that the speed of light equals $1$, the relativistic Landau equation reads
\begin{equation}\label{e.main}
    \left(\partial_t  + \frac p{\pp} \cdot \nabla_x\right) f = \QRL(f,f),
\end{equation}
where $\pp := \sqrt{1+|p|^2}$ and the unknown function $f(t,x,p)\geq 0$ represents the density of particles at time $t\geq 0$, location $x\in \R^3$, and with momentum $p\in \R^3$. The collision operator $\QRL(f,g)$ is an integral operator, acting only in the $p$ variable, defined for $f, g:\R^3\to \R$ by
\begin{equation}\label{e.collision}
\QRL(f,g) = \nabla_p \cdot \left( \int_{\R^3} \Phi(p,q) \left( f(q) \nabla_p g(p) - g(p) \nabla_q f(q)\right) \dd q\right),
\end{equation}
where the integral kernel $\Phi(p,q)$ is a $3\times 3$ matrix described by: % that we describe in detail below. 
\begin{equation}\label{e.Phi-def}
\begin{split}
\Phi^{ij}(p,q) &= \Lambda(p,q) S^{ij}(p,q),\\
\Lambda(p,q) &= \frac{(\tau-1)^2} {\pp\qq}[\tau(\tau-2)]^{-\sfrac{3}{2}},\\
S^{ij}(p,q) &= \tau(\tau-2) \delta_{ij} - (p_i-q_i)(p_j-q_j) +(\tau-2)(p_iq_j + p_j q_i),\\
\tau &= \pp\qq - p\cdot q +1.
\end{split}
\end{equation}

Mathematically, equation~\eqref{e.main} features the interaction of relativistic transport with nonlinear, nonlocal diffusion in the $p$ variable. As with other collisional kinetic equations, controlling the collision term quantitatively is a key difficulty. 
Compared to the kernel of the classical Landau collision operator, $\Phi(p,q)$ is much more algebraically complex.  
In particular, $\Phi$ is not a function of $p-q$ only, which means the integral in~\eqref{e.collision} is not a convolution.

Prior mathematical work on~\eqref{e.main} has focused on either the spatially homogeneous case, where $f$ is independent of the $x$ variable \cite{strain2019relativistic, strain2020uniqueness}, or the close-to-equilibrium case, where $f$ is sufficiently close to a relativistic Maxwellian (also called J\"uttner solution) $\mu(t,x,p) = c_1 e^{-c_2 \pp}$, $c_1,c_2>0$ \cite{strain2004stability, yang2010hypo, lyu2022finite}. In this article, we initiate the study of~\eqref{e.main} in the {\it spatially inhomogeneous, large-data} regime. By large-data, we mean that the solution and initial data are not assumed to be close to equilibrium.

A reasonable starting point for the large-data theory is to establish {\it a priori} regularity and decay estimates. These two properties are needed in any quantitative bound on the collision operator because $\Phi(p,q)$ has a singularity at $p=q$ and decays slowly as $|q|\to \infty$. In this article, we find that solutions are $C^\infty$ in all three variables and satisfy rapid decay estimates as $|p|\to \infty$ as soon as bounds on some zeroth-order quantities are satisfied (see \Cref{t:main}). The smoothing part of our theorem provides a relativistic counterpart of the main result of \cite{henderson2020smoothing} by the first two authors. (We also refer to \cite{liu2014regularization, golse2019, morimoto2020smoothing, HendersonWang, cao2023smoothing} and the references therein for other regularity results on the classical, inhomogeneous Landau equation.) As in the classical case, the basic smoothing mechanism of~\eqref{e.main} comes from the diffusive effect of the collision term. Since this term acts only in $p$, regularization in the $x$ variable relies on the mixing generated by the interaction of $p$-diffusion and transport.

Apart from their own intrinsic interest, the regularity and decay estimates in this article are intended to lay the groundwork for a large-data existence theory of~\eqref{e.main} and will also be helpful for the study of further qualitative properties, by giving easy-to-check conditions that allow one to work with very smooth, rapidly-decaying solutions.

Recall that (in our choice of units) the energy of a particle with momentum $p$ is given by $\pp$. We can therefore define the macroscopic quantities
\[
\begin{split}
    M_f(t,x)
        &= \int_{\R^3} f(t,x,p) \dd p,
        \quad \text{(mass density)},\\
    E_f(t,x)
        &= \int_{\R^3} \pp f(t,x,p) \dd p,
        \quad \text{(energy density)}.
\end{split}
\]
Another relevant macroscopic quantity, that however does not appear in our main results, is the absolute entropy density:
\[
    \overline H_f(t,x)
        = \int_{\R^3} f(t,x,p)|\log f(t,x,p)| \dd p.
\]
Our main results state that solutions to~\eqref{e.main} are $C^\infty$ with rapid decay whenever three conditions on $f$ hold: the mass density is uniformly bounded below, the energy density is uniformly bounded above, and the solution $f$ is bounded.

\subsection{Main results}

First, we present our decay estimates for large $p$. We are able to prove both polynomial and exponential decay estimates: 

\begin{theorem}\label{t:decay}
Let $f$ be a classical solution of the relativistic Landau equation~\eqref{e.main} on $[0,T]\times \R^3_x\times\R^3_p$ for some $T>0$. Assume that there are $E_0, F_0 > 0$ such that
\begin{equation}\label{e.E_0L_0}
    E_f(t,x) \leq E_0
    \quad\text{ and }
    \quad \|f(t,x,\cdot)\|_{L^\infty(\R^3)} \leq F_0,
\end{equation}
for all $(t,x) \in [0,T]\times\R^3$. Further, suppose that there are $k > 0$ and $M_k > 0$ such that the initial data $f_{\rm in}: \R^3_x\times\R^3_p\to [0,\infty)$ satisfies 
\[
\pp^k f_{\rm in}(x,p) \leq M_k
    \quad\text{ for all } (x,p) \in \R^3\times \R^3.
\]
Then $f$ satisfies the decay estimate
\[
\pp^k f(t,x,p) \leq e^{\beta t}M_k
    \quad\text{ for all } (t,x,p) \in [0,T]\times\R^3\times \R^3,
\]
for some $\beta>0$ depending only on $k$, $E_0$, and $F_0$.

If, in addition, there are $\sigma, N>0$ such that
\[
e^{\sigma\pp} f_{\rm in}(x,p) \leq N
    \quad\text{ for all } (x,p) \in \R^3 \times \R^3,
\]
then $f$ satisfies the exponential decay estimate
\[
    e^{\sigma \pp} f(t,x,p)
        \leq N\exp(CN e^{Ct} t)
        \quad\text{ for all } (t,x,p) \in [0,T]\times\R^3\times \R^3,
\]
for some $C>0$ depending only on $\sigma$, $E_0$, and $F_0$.
\end{theorem}
Recall that the equilibrium solutions of~\eqref{e.main} are the  J\"uttner solutions $\mu(t,x,p) = c_1 e^{-c_2 \pp}$, $c_1, c_2>0$. \Cref{t:decay} implies in particular that if $f$ starts below a J\"uttner solution, it stays bounded (up to a time-dependent constant) by the same J\"uttner solution for positive times, as long as the energy density and $L^\infty$-norm of $f$ remain bounded.

Our next main result is an {\em a priori} regularity estimate for solutions. Let us note that this regularity theorem is substantially more difficult to prove than the decay estimate of \Cref{t:decay}.

\begin{theorem}\label{t:main}
Let $f$ be a classical solution of the relativistic Landau equation~\eqref{e.main} on $[0,T]\times\R^3_x\times\R^3_p$ for some $T>0$. Assume that $f$ satisfies the bounds
\begin{equation}\label{e.hydro}
  M_f(t,x) \geq m_0, \quad E_f(t,x) \leq E_0
    \quad\text{ for all } (t,x) \in [0,T]\times \R^3,
\end{equation}
for some $m_0, E_0>0$, as well as
\begin{equation}
\|f\|_{L^\infty([0,T]\times\R^3\times\R^3)} \leq F_0,
\end{equation}
for some $F_0>0$. Furthermore, assume that the initial data satisfies the upper bounds
\begin{equation}\label{e.initial-decay}
\pp^k f_{\rm in}(x,p) \leq M_k, \quad (x,p) \in \R^3\times\R^3,
\end{equation}
for all $k>0$ and some constants $M_k>0$. 

Then $f \in C^\infty((0,T]\times\R^3\times\R^3)$, and for any partial derivative $\partial^\beta$ in $(t,x,p)$ variables, any $k>0$, and any $\tau \in (0,T]$, the estimate
\begin{equation}
\|\pp^k \partial^\beta f\|_{L^\infty([\tau,T]\times\R^3\times\R^3)} \leq K_{\beta,k}
\end{equation}
holds, for some $K_{\beta,k}$ depending only on $\beta$, $k$, $\tau$, $m_0$, $E_0$, $F_0$, and $M_{\ell}$ for some $\ell>k$.
\end{theorem}

%Let us remark that $C^\alpha_{\rm kin}$ is the H\"older norm with respect to the kinetic scaling, which is defined in \Cref{s:schauder}. %a Lorentz-invariant distance. This notation is explained in detail in \Cref{s:lorentz}. At least locally, the $C^\alpha_\cL$-norm controls the standard H\"older norm on $\R^7$, so \Cref{t:main} implies all partial derivatives of $f$ are H\"older continous in the usual sense.

We stress that, if $f$ decays at some finite, but sufficiently large polynomial rate as $|p|\to \infty$, our proof of \Cref{t:main} shows that regularity estimates up to a certain order will hold.  The higher the polynomial rate of decay, the smoother $f$ is.  See \Cref{t:local-reg} below for a precise statement of this fact.

\subsection{Comparison with classical Landau equation}

Both of our main results are in a similar spirit to prior results for classical Landau. Let us recall the form of the classical equation, in which the solution $f(t,x,v)$ satisfies
\begin{equation}\label{e.classical}
    (\partial_t + v\cdot \nabla_x) f = Q_{\rm CL}(f,f),
\end{equation}
where
\[
    Q_{\rm CL}(f,f)
        = \nabla_v\cdot \left(\int_{\R^3} |w|^{\gamma+2} (I - \hat w \otimes \hat w) [f(v-w)\nabla_v f(v) - f(v) \nabla_v f(v-w)] \dd w\right),
\]
and $\hat w = \sfrac{w}{|w|}$. The parameter $\gamma$ is generally taken in the range $[-3,1]$, and the case $\gamma = -3$ (Coulomb potentials) is considered the most important due to its relevance as a model in plasma physics. In comparing the results in the current article to the state of the art for~\eqref{e.classical}, we should note that our equation~\eqref{e.main} is the relativistic analogue of the Coulomb case specifically.

The smoothing estimate of \Cref{t:main} is the relativistic analogue of the main result of \cite{henderson2020smoothing}. The smoothing estimates in \cite{henderson2020smoothing} apply in the case $\gamma \in [-3,0)$ and depend quantitatively on upper bounds for the mass, energy, and entropy densities of $f$ and a lower bound for the mass density of $f$.  When $\gamma \leq -2$, they additionally depend on an upper bound for $f$ in $L^\infty_{t,x}L^p_v$ for some $p = p(\gamma)$ depending on $\gamma$, with $p(-3) =  \infty$. Therefore, the dependence on $E_0$ and $F_0$ in our \Cref{t:main} is in alignment with the condition in \cite{henderson2020smoothing} in the Coulomb case. (Note that the energy density and $L^\infty$-norm control the mass and entropy densities from above.) 
We should note, however, that the {\it proofs} of our regularity estimates are substantially different than in the classical case, as we explain below in \Cref{s:proofs}.

Decay estimates for large $v$ are available for~\eqref{e.classical} in, e.g., \cite{cameron2017landau, S2018hardpotentials, HST2019rough}. It is known that when $\gamma >0$, Gaussian decay estimates in $v$ appear spontaneously regardless of the initial data, and when $\gamma \in [-3,0]$, polynomial and/or Gaussian decay estimates that hold at time zero are propagated forward in time, but do not spontaneously appear. Therefore,  \Cref{t:decay} is analogous to what is known for the Coulomb case of~\eqref{e.classical}. (Note that the equilibrium solutions in the classical case are Gaussians in $v$, whereas in the relativistic case, equilibrium states take the form $e^{-\theta \pp}$ for any $\theta >0$.)

\subsection{Relativistic Schauder estimate}\label{s:schauder-intro}

To prove higher regularity of solutions, one of our tools is a Schauder estimate for linear relativistic kinetic equations with general nondivergence-form right-hand side:
\begin{equation}\label{e.ABC}
\partial_t u + \frac p {\pp} \cdot \nabla_x u   = \tr(A D_p^2 u) + B\cdot \nabla_p u + s,
\end{equation}
where $A$, $B$, $s$ are H\"older continuous functions, and $A$ is uniformly elliptic.  Our estimate says that if $u$ is locally bounded and $A, B, s$ are locally H\"older continuous, then $u$ lies in a second-order regularity space $C^{2,\alpha}_{\rm kin}$. The precise statement can be found in \Cref{t:first_schauder}.

This Schauder estimate we prove in \Cref{t:first_schauder} may be of independent interest. 
 There is a rich literature on Schauder estimates for kinetic equations, and more broadly hypoelliptic/ultraparabolic equations: see \cite{manfredini1997ultraparabolic, difrancesco2006schauder, bramanti2007schauder, imbert2021schauder, dong2022schauder, lucertini2023schauder,  loher2023schauder, HendersonWang,  biagi2024schauder}
 and the references therein. These works are applicable to classical kinetic equations involving the transport derivative $\partial_t + v\cdot \nabla_x$.  To the best of our knowledge, \Cref{t:first_schauder}  is the first Schauder estimate that applies to equations with the relativistic transport derivative $\partial_t + \frac p \pp \cdot \nabla_x$.

Apart from relativistic Landau,  \Cref{t:first_schauder} can also be used to establish higher regularity for the following {\it relativistic Fokker-Planck equation}:
\begin{equation}\label{e.relativistic-FP}
\partial_t u + \frac p {\pp} \cdot \nabla_x u = \nabla_p\cdot \left( \frac 1 \pp \left( I + p\otimes p\right) \nabla_p u + \beta p u\right),
\end{equation}
where $\beta$ is a friction coefficient, which is taken to be zero in some works. 
This equation has been studied by Alc\'antara and Calogero \cite{alcantara2011relativistic, alcantara2013relativistic}, who showed that it is invariant under Lorentz transformations (if $\beta=0$), studied the trend to equilibrium, and characterized the Newtonian limit. Later, Anceschi-Polidoro-Rebucci \cite{anceschi2022relativistic} proved a Harnack inequality and quantitative lower bounds. For a version of the equation with a confining potential, a hypoelliptic smoothing estimate in a weighted first order space was recently established by Arnold-Toshpulatov \cite{arnold2024relativistic}, along with a trend-to-equilibrium result. Using our \Cref{t:first_schauder} and differentiating the equation repeatedly, one can show that solutions to~\eqref{e.relativistic-FP} are in fact infinitely differentiable.

\subsection{Related work}\label{s:related}

Much of the mathematical work on equation~\eqref{e.main} has dealt with the near-equilibrium case. In some sense, this began with the work of Lemou \cite{lemou2000linearized}, who studied the linearization around the J\"uttner equilibrium solution and established a spectral gap property. The study of the linearized operator was later extended by Luo and Yu \cite{luo2016spectrum}. Yang and Yu \cite{yang2010hypo}
established hypocoercivity and optimal convergence rates to equilibrium in the whole space $\R^3_x$, and Lyu-Sun-Wu \cite{lyu2022finite} established finite speed of propagation in the near-equilibrium framework.

Still in the near-equilibrium case, there have been a number of works on systems that couple equation~\eqref{e.main} to an equation for the electro-magnetic field. Strain and Guo \cite{strain2004stability}  established global existence near equilibrium, and stability of equilibrium states, for the relativistic Landau-Maxwell system with periodic boundary conditions in space. Yang and Yu \cite{yang2012relativistic} later extended this result to solutions defined in the whole space $\R^3_x$, and Liu-Zhao \cite{liu2014optimal} established the optimal decay rate to equilibrium for solutions in the whole space. Li-Yu-Zhong \cite{li2017spectrum} studied the relativistic Vlasov-Poisson-Landau system, characterizing the spectral structure of the linearization and constructing global solutions with optimal decay rate toward equilibrium. Recently, Cao-He-Lei-Xiao \cite{cao2025newtonian} studied the Newtonian limit (i.e. the convergence as the speed of light goes to infinity) of the Vlasov-Maxwell-Landau system in the near-equilibrium regime.

 In \cite{ouyang2025hilbert}, Ouyang, Wu, and Xiao established the convergence of solutions to relativistic Landau toward solutions of relativistic Euler in the hydrodynamic scaling. The authors of \cite{ouyang2025hilbert} also studied the hydrodynamic limit of relativistic Vlasov-Maxwell-Landau. Unlike the works cited in the previous two paragraphs, the solutions in \cite{ouyang2025hilbert} are close to a {\it local} (i.e. $t$ and $x$ dependent) relativistic Maxwellian.

Equation~\eqref{e.main} has also been studied in the space-homogeneous case. The fourth named author of the present work, jointly with Strain \cite{strain2019relativistic}, studied entropy dissipation estimates, and applied these estimates to derive short-time existence of weak solutions and propagation of polynomial decay. Next, Strain-Wang \cite{strain2020uniqueness} studied uniqueness of the homogeneous equation, using a probabilistic representation of the solution.

We are not aware of any prior results on the relativistic Landau equation that deal with the spatially inhomogeneous, large-data setting. However, we should mention the result of Zhu \cite{zhu2021averaging}, who proved local boundedness and  $C^\alpha$ estimates for linear Fokker-Planck equations with general transport terms and rough coefficients, which are applicable to~\eqref{e.main} under suitable conditions on $f$.

There is a much broader literature on the classical Landau equation~\eqref{e.classical}, which is too large to summarize here. Focusing only on the large-data, spatially inhomogeneous regime, we refer to \cite{golse2019} for local $C^\alpha$ estimates, \cite{cameron2017landau} for conditional global upper bounds, \cite{henderson2020smoothing, S2018hardpotentials} for global $C^\infty$ estimates, \cite{HST2018landau, snelson2023landau} for continuation criteria, \cite{he2014boltzmannlandau, HST2018landau, HST2019rough, chaturvedi2023existence, snelson2024hardpotentials} for short-time existence of classical solutions, \cite{villani1996global} for global weak solutions, and \cite{alexandre2004landau} for a rigorous proof of the grazing collisions limit.

\subsection{Difficulties and proof ideas}\label{s:proofs}

\subsubsection{Decay} 

The proof of \Cref{t:decay} is based on the fact that inverse-polynomial functions $\phi(t,p) = e^{\beta t}\pp^{-k}$ and exponential functions $\psi(p) = e^{\beta t} e^{-\sigma \pp}$ are both supersolutions to the linear relativistic Landau equation
\[
\left(\partial_t  + \frac p {\pp}\cdot \nabla_x\right) g = \QRL(f,g),
\]
where $f$ is a fixed nonnegative function satisfying the conditional bounds $E_f(t,x) \leq E_0$ and $\|f\|_{L^\infty} \leq F_0$. This supersolution property follows quickly, once one has access to precise bounds for $\QRL(f,g)$, some of which we borrow from \cite{strain2019relativistic} and some we derive in \Cref{s:pointwise}. Then, if one chooses $f$ to be a solution to relativistic Landau, the functions $\phi - f$ and $\psi - f$ are supersolutions as well, which yields the desired upper bounds.

Similar barrier methods have been applied to prove pointwise decay for other kinetic equations: see e.g. \cite{imbert2018decay, cameron2017landau, S2018hardpotentials,  cameron2020decay, henderson2023decay}. We emphasize that the rather strong conclusion of \Cref{t:decay} follows quickly using this flexible method.

\subsubsection{Regularity}

By contrast, proving higher regularity of solutions is an involved process that requires taking into account certain geometric aspects of special relativity, which affect the argument in a severe way due to the nonlocality of $\QRL$ in the momentum variable.

First, let us note that equation~\eqref{e.main} can be written in the form of the linear equation~\eqref{e.ABC}, with $A = a^f$, $B = b^f$, and $s = c^f f$ for certain coefficients defined in terms of $f$. 
(For the precise formulas for $a^f$, $b^f$, and $c^f$, see \Cref{s:Q} below.) This suggests a strategy based on applying local regularity estimates, passing regularity from $f$ to the coefficients, differentiating the equation, and repeating.

Let us discuss the two main parts of this strategy in more detail, along with the difficulties that arise:

\begin{itemize}

\item {\it Relativistic Schauder estimates.} The Schauder-type estimate described in \Cref{s:schauder-intro} is our key tool in proving higher regularity. To establish this estimate, one might try to adapt some of the proof strategies in the literature on non-relativistic kinetic equations. However, most of these strategies would be complicated by the fact that~\eqref{e.ABC} has worse scaling properties than the corresponding Newtonian kinetic equation
\begin{equation}\label{e.Newton}
    (\partial_t  + v\cdot \nabla_x) f
    = \tr(A D_v^2 f) + B\cdot \nabla_v f + s.
\end{equation}
In particular, $f(r^2 t, r^3 x, rv)$ solves an equation of the same type as~\eqref{e.Newton}, with the same ellipticity constants, for any $r>0$. 
For the relativistic equation \eqref{e.ABC}, the scaling situation is more complicated because of the relativistic transport term: the rescaled function $f(r^2t,r^3x, rp)$ solves an equation with a different transport derivative (corresponding to a different speed of light). 
%There is no analogue of this fact for~\eqref{e.ABC}, due to the presence of the relativistic transport term, and 
This would likely be an obstacle in many approaches to proving Schauder estimates directly.

Instead, we prove our Schauder estimate (\Cref{t:first_schauder}) by applying a change of variables $p \mapsto v = p/\pp$ that sends momentum to velocity, which transforms~\eqref{e.ABC} into a non-relativistic kinetic equation  (which we refer to as the ``velocity equation''), with modified coefficients (see \Cref{l:transform} for the exact form of this new equation). We then apply known Schauder estimates to the velocity equation, and transform back. A new difficulty arises here: the change of variables produces an equation whose ellipticity constants degenerate as $|p|\to \infty$ (equivalently, $|v|\to 1$) in an anisotropic way.  The degenerating ellipticity constants make it difficult to track the dependence of the constants in our local Schauder estimates as the center of the local cylinder approaches $\infty$ in momentum. Tracking this dependence is essential because the regularity estimates must be fed into the nonlocal coefficients, i.e. the regularity of $a^f$, $b^f$, and $c^f$ at some $(t_0,x_0,p_0)$ depends on the regularity of $f$ in the entire $p$-domain.

To get around this difficulty, we introduce a second change of variables that ``straightens out'' the ellipticity estimates of the velocity equation. The most difficult aspect of this strategy is proving that the transformed version of the velocity equation satisfies uniform ellipticity estimates that do not degenerate for large $v$ (see \Cref{l:P}). This issue of degenerating ellipticity is similar to what is seen for the classical Landau equation, and the change of variables we use is reminiscent of one introduced in \cite{cameron2017landau} by the third named author, jointly with Cameron and Silvestre. However, our proof of uniform ellipticity is somewhat different and more elementary than the proof in \cite{cameron2017landau}.

\item {\it Bootstrapping.} 
After each application of a regularity estimate, one must pass regularity of $f$ to regularity of the coefficients $a^f$, $b^f$, and $c^f$. This requires us to prove delicate functional inequalities in H\"older spaces for the coefficients and their derivatives. (See \Cref{s:coeffs}.)

Unlike in the classical Landau equation, the integral defining the collision operator $\QRL(f,g)$ in~\eqref{e.collision} is not a convolution (i.e. $\Phi(p,q)$ is not a function of the difference $p-q$ only). This causes another difficulty: when differentiating the equation in order to bootstrap regularity estimates, it is unavoidable that $p$ derivatives fall on the kernel $\Phi(p,q)$. In particular, $\partial_{p_i} a^f \neq a^{\partial_{p_i} f}$. Since higher derivatives falling on $\Phi(p,q)$ can produce non-integrable singularities, we must use an integration-by-parts technique involving a mixed differential operator that is suited to the relativistic structure of the kernel $\Phi(p,q)$. This technique is borrowed from the work of Strain-Guo \cite{strain2004stability}.

\end{itemize}

\subsection{Open problems}

There is a lot of room for further study of~\eqref{e.main} in the large-data regime. Here, we focus on the (in our opinion) most important directions for the near future.

\subsubsection{Filling of vacuum}

The classical Landau equation is known to fill vacuum regions (i.e. regions in the spatial domain where the initial data is identically zero) instantaneously \cite{HST2018landau}. One should not expect this property to hold in the relativistic case, since particles cannot travel faster than the speed of light.  The finite speed of propagation gives rise to a light cone phenomenon, which has already been observed in the near-equilibrium regime \cite{lyu2022finite}. In the large-data setting, one would like to prove that solutions are strictly positive inside the forward light cone of the positivity set of the initial data $f_0$, with continuity estimates across the boundary of the light cone. We plan to explore this question in a forthcoming article.

\subsubsection{Weakening the conditional assumptions in \Cref{t:main}}

As mentioned above, the conditional regularity result of \cite{henderson2020smoothing} established higher regularity estimates for classical Landau, in the Coulomb case, depending on bounds for the mass and energy densities, as well as an upper bound on the $L^\infty$-norm of $f$. This $L^\infty$-bound was later improved to a bound in $L^\infty_{t,x} L^{\delta + \sfrac32}_v$ for any small $\delta>0$ in \cite{snelson2023landau}. One might expect that a similar result would be available in the relativistic case, involving a bound on $f$ in a norm that is weaker than $L^\infty$ in the $p$ variable.

\subsubsection{Short-time existence}

There are no existence theorems for classical solutions of~\eqref{e.main} in the spatially-inhomogeneous, large-data regime. Therefore, it is natural to extend local existence results for classical Landau such as \cite{he2014boltzmannlandau, HST2018landau, HST2019rough} to the relativistic case.

We should note that global-in-time existence for~\eqref{e.main} with large initial data is most likely out of reach with current techniques, and that the obstructions appear to be similar to those seen in the case of classical Landau.

\subsection{Notation}
Throughout the paper, for any $q \in [1,\infty]$, we use $L^q$-norms with polynomial weight of order $k \in \R$ in momentum:
\[
    \|f\|_{L^q_k(\Omega)}
        = \|\pp^k f\|_{L^q(\Omega)},
\]
where $\Omega$ is any subset of $\R^3_p$, $\R^3_x\times\R^3_p$, or $\R_t \times \R^3_x\times\R^3_p$. 

We also use H\"older spaces based on the kinetic scaling. The notation for these spaces is explained in \Cref{s:schauder}. 

We sometimes use the notation $A\lesssim B$ to mean $A\leq CB$ for a constant $C$ depending on the quantities listed in the statement of a given lemma or theorem. We also write $A\approx B$ when $A\lesssim B$ and $B\lesssim A$. 

\subsection{Outline of the paper}

In \Cref{s:prelim}, we discuss some results from the literature that are needed in our analysis. \Cref{s:pointwise} presents upper bounds for the coefficients $a^f$, $b^f$, and $c^f$, and \Cref{s:decay} combines these coefficient estimates with barrier arguments to prove our main decay result, \Cref{t:decay}. The remainder of the paper is devoted to our proof of higher regularity:  \Cref{s:schauder} establishes our relativistic Schauder estimate for linear equations, \Cref{s:coeffs} presents regularity estimates for the coefficients $a^f$, $b^f$, and $c^f$, and \Cref{s:higher} proves our main regularity result, \Cref{t:main}. Finally, \Cref{s:appendix} contains some technical lemmas about our kinetic H\"older spaces.

\section{Preliminaries and known results}\label{s:prelim}

\subsection{Form of the collision operator}\label{s:Q}

For $f,g:\R^3\to \R$, the bilinear relativistic Landau-Coulomb collision operator $\QRL(f,g)$ is defined by~\eqref{e.collision}, where the matrix $\Phi(p,q)$ is given by~\eqref{e.Phi-def}. It is often convenient to write $\QRL(f,g)$ as a diffusion operator in $p$, either in nondivergence form
\begin{equation}\label{e.nondivergence}
\QRL(f,g) =  \tr(a^f D_p^2 g) + b^f\cdot \nabla_p g + c^f g,
\end{equation}
or divergence form 
\begin{equation}\label{e.divergence}
\QRL(f,g) = \nabla_p\cdot ( a^f \nabla_p g) + B^f \cdot \nabla_p g + c^f g,
\end{equation}
where, for any $f:\R^3\to \R$, 
\begin{align}
 a_{ij}^f(p) &=  \int_{\R^3} \Phi^{ij}(p,q)f(q) \dd q, \label{e.a}\\
 b_i^f(p) &= \sum_{j=1}^3 \left(\int_{\R^3}  \partial_{p_j} \Phi^{ij}(p,q) f(q) \dd q - \int_{\R^3} \Phi^{ij}(p,q) \partial_{q_j} f(q) \dd q\right), \quad \text{ and}
    \\
 c^f(p) &= -\sum_{i,j=1}^3 \partial_{p_i} \left( \int_{\R^3} \Phi^{ij}(p,q) \partial_{q_j}f(q) \dd q\right).
\end{align}
These formulas follow by a formal differentiation and the fact that $\Phi^{ij} = \Phi^{ji}$. After a non-obvious calculation involving integration by parts and the symmetry properties of $\Phi^{ij}$, as demonstrated in \cite{strain2004stability, strain2019relativistic}, one obtains the formulas
\begin{align}
 b_i^f(p) &= 2\int_{\R^3} \Lambda(p,q)(\tau-2)(p_i+q_i) f(q) \dd q,\label{e.b}\\
 B_i^f(p) &= 2\int_{\R^3} \Lambda(p,q)[(\tau-1)p_i-q_i] f(q) \dd q, \quad \text{ and}\label{e.other-B}\\
 c^f(p) &= 4\int_{\R^3} \frac 1 {\pp\qq} \frac{\tau-1}{\sqrt{\tau(\tau-2)}} f(q) \dd q + \kappa(p) f(p),\label{e.c}
\end{align}
where
\begin{equation}\label{e.kappa}
\kappa(p) = 2^{\sfrac{7}{2}} \pi \pp \int_0^\pi (1+|p|^2\sin^2\theta)^{-\sfrac{3}{2}} \sin\theta \dd \theta.
\end{equation}
We use these formulas for the coefficients throughout our paper.

Next, let us discuss estimates that are known for the collision operator. First, we quote the uniform ellipticity estimates that were shown in \cite{strain2019relativistic} as a by-product of the proof of entropy dissipation estimates. This is important when we apply regularity estimates in the sequel. We should note that the following proposition is a key difference between the relativistic and classical cases; the diffusion matrix $\bar a^f(t,x,v)$ in classical Landau satisfies upper and lower ellipticity estimates that degenerate in an anisotropic way as $|v|\to \infty$. (See, e.g. \cite{cameron2017landau}.)

\begin{proposition}{\cite[Lemmas 13 and 14]{strain2019relativistic}}\label{p:A-elliptic}
\,\\
\begin{enumerate}
\item[(a)]  Let $a^f$ be defined by~\eqref{e.a}, for a function $f \in L^1_s(\R^3) \cap L^3(\R^3)$ with $s>2$. Then
\[
a^f_{ij}(p) \xi_i \xi_j \leq C|\xi|^2,\quad \xi\in \R^3,
\]
for a constant $C>0$ depending on $s$, $\|f\|_{L^1_s(\R^3)}$, and $\|f\|_{L^3(\R^3)}$.

\item[(b)] If one only has $f\in L^1_1(\R^3) \cap L^3(\R^3)$, then 
\[
a^f_{ij}(p) \xi_i \xi_j \leq C\pp|\xi|^2,\quad \xi\in \R^3,
\]
for a constant $C>0$ depending on $\|f\|_{L^1_1(\R^3)}$ and $\|f\|_{L^3(\R^3)}$. 

\item[(c)]  For $f\geq 0$ satisfying $0< m_0 \leq M_f$, $E_f\leq E_0$, and $\overline H_f \leq H_0$, and $a^f$ defined by~\eqref{e.a}, there holds
\[
a^f_{ij} \xi_i \xi_j \geq c|\xi|^2, \quad \xi \in \R^3,
\]
for a constant $c>0$ depending on $m_0$, $E_0$, and $H_0$. 
\end{enumerate}
\end{proposition}

We also have a useful upper bound for the integral kernel matrix $\Phi^{ij}$:

\begin{lemma}{\cite[Lemma 12]{strain2019relativistic}}\label{l:Phi-upper}
 There exists a constant $C>0$ such that for any $p,q\in \R^3$,
 \[
 |\Phi^{ij}(p,q)| \leq 
C
\begin{cases}
 \dfrac{\pp^{\sfrac{1}{2}} \qq^{\sfrac{1}{2}}}{|p-q|}, & \text{if } \tau-2 < 1/8,\\[10pt]
\dfrac{\pp}{\qq} + \dfrac{\qq}{\pp}, & \text{if } \tau-2 \ge 1/8.
\end{cases}
\]
\end{lemma}

Let us also quote two useful facts about expressions that appear as part of $\Phi(p,q)$. Recalling the definition $\tau = \pp\qq - p\cdot q + 1$, we have the following estimate from \cite[Lemma~3.1]{glasseystrauss}:
\begin{lemma}\label{l:GS}
For $p,q\in \R^3$,
\begin{equation}
\frac{|p-q|^2 + |p\times q|^2}{2\pp\qq} \leq \tau - 2 \leq \frac 1 2 |p-q|^2.
\end{equation}
\end{lemma}
It is also useful to know that 
\begin{equation}
\tau \geq 2, 
\end{equation}
which follows from \Cref{l:GS}.

\subsection{H\"older estimates for linear kinetic equations}

The first step in establishing regularity of our solutions is to apply local $L^\infty$ and $C^\alpha$ estimates of De Giorgi-Nash-Moser type for linear relativistic Fokker-Planck equations of the form 
\begin{equation}\label{e.RFP-divergence}
	\partial_t f + \frac p {\pp}\cdot \nabla_x f
		= \nabla_p\cdot (A \nabla_p f) + B\cdot \nabla_p f + s,
\end{equation}
defined for $(t,x,p)\in Q_1 = (-1,0]\times B_1 \times B_1$. One assumes $A$, $B$, and $s$ are bounded in $L^\infty(Q_1)$, and that $A$ satisfies the ellipticity assumption
\begin{equation}\label{e.A-elliptic}
	\lambda \Id
		\leq A(t,x,p)
		\leq \Lambda \Id
\end{equation}
for all $(t,x,p) \in Q_1$. 

As mentioned in the introduction, a H\"older estimate for this equation has been derived by Zhu \cite{zhu2021averaging}, but, when trying to apply the estimate on a cylinder far from the origin by composing with a Lorentz boost, one obtains an equation with ellipticity constants depending on $p_0$. Since this H\"older estimate needs to be plugged into the collision operator to obtain regularity for the nonlocal coefficients, it is necessary to quantify precisely the asymptotics as $|p_0|\to \infty$.  This quantification is not straightforward, partly due to the lack of scaling symmetries in the equation~\eqref{e.RFP-divergence}. Instead, we use a change of variables to exchange momentum with velocity, apply a $C^\alpha$-estimate, and change variables back. (This is the same approach we take when applying Schauder estimates.) Therefore, the $C^\alpha$-estimate we need is the following, from the work of Golse-Imbert-Mouhot-Vassur \cite{golse2019}:

\begin{proposition}{\cite[Theorem 1.4]{golse2019}}\label{p:Calpha}
Let $f(t,x,v)$ be a solution of
\[
\partial_t f + v\cdot\nabla_x f = \nabla_v \cdot (A\nabla_v f) + B\cdot \nabla_v f + s,
\]
where $A, B, s$ are bounded and measurable and $A$ is uniformly elliptic in the sense that $\lambda \Id \leq A(t,x,v) \leq \Lambda \Id$ for all $(t,x,v) \in Q_1$. Then 
\[
\|f\|_{C^\alpha_\cG(Q_{\sfrac12})} \leq C\left( \|f\|_{L^2(Q_1)} + \|s\|_{L^\infty(Q_1)}\right),
\]
for some $\alpha\in (0,1)$ and $C>0$ depending on $\lambda$, $\Lambda$, and the $L^\infty$-norm of $B$. 
\end{proposition}

The H\"older norm $C^\alpha_\cG$ is defined below in \Cref{s:schauder}.

\section{Pointwise estimates for the coefficients}\label{s:pointwise}

In this section, we obtain various pointwise bounds on the coefficients that are required to apply the regularity estimates that we derive in the sequel. 
First, we have a simple upper bound for the integral kernel appearing in $c^f(z)$:
\begin{lemma}\label{l:G-upper}
Let $G(p,q)$ be defined by
\[
G(p,q) := \frac 1 {\pp\qq}\frac {\tau - 1}{\sqrt{\tau(\tau-2)}}, \quad p,q \in \R^3.
\]
Then for any $p,q\in \R^3$, the upper bound
\[
G(p,q) \leq \frac C {|p-q|}
\]
holds, where $C>0$ is a constant independent of $p$ and $q$.
\end{lemma}
\begin{proof}
From \Cref{l:GS} and $\tau^{-\sfrac{1}{2}} \leq (\tau - 1)^{-\sfrac{1}{2}}$, we have
\begin{equation}
\frac{\tau - 1}{\sqrt{\tau(\tau-2)}} \leq \frac{\sqrt 2 \sqrt{\pp\qq}\sqrt{\tau -1}}{|p-q|}. 
\end{equation}
Using this inequality, we estimate $G(p,q)$ as follows:
\begin{equation}
G(p,q) \lesssim \frac{1}{\sqrt{\pp \qq}} \frac{\sqrt{\tau - 1}}{|p-q|} = \frac{1}{\sqrt{\pp \qq}} \frac{\sqrt{\pp \qq - p\cdot q}}{|p-q|}.
\end{equation}
Since $|p\cdot q|\lesssim \pp\qq$, we have
\begin{equation}
	G(p,q) \lesssim \frac 1 {|p-q|},
\end{equation}
as claimed.
\end{proof}

\begin{lemma}\label{l:c-upper}
For $g\in L^\infty(\R^3)\cap L^1(\R^3)$ and $c^g$ defined by~\eqref{e.c}, there holds
\[
|c^g(p)| \leq  C \left(\|g\|_{L^\infty(\R^3)} + \|g\|_{L^1(\R^3)}\right), \quad p\in \R^3,
\]
for a constant $C>0$ independent of $g$.
\end{lemma}

\begin{proof}
First, we estimate $\kappa(p)$ as defined in~\eqref{e.kappa}. If $|p|\geq 1$,
\[
\begin{split}
\kappa(p) &\lesssim \pp |p|^{-3} \int_0^\pi (|p|^{-2} + \sin^2\theta)^{-\sfrac{3}{2}} \sin\theta\dd\theta\\
&\lesssim \pp^{-2} \left(\int_0^{|p|^{-1}}  + \int_{|p|^{-1}}^{\pi - |p|^{-1}} + \int_{\pi-|p|^{-1}}^\pi \right)(|p|^{-2}+\sin^2\theta)^{-\sfrac{3}{2}} \sin \theta \dd \theta\\
&\lesssim \pp^{-2}\left( \int_0^{|p|^{-1}} |p|^3 \sin \theta \dd \theta + \int_{|p|^{-1}}^{\pi-|p|^{-1}} \sin^{-2}\theta \dd \theta +\int_{\pi-|p|^{-1}}^\pi |p|^3 \sin\theta \dd \theta \right)\\
&\lesssim \pp^{-2} \left( |p| - \cot(\pi-|p|^{-1}) +\cot(|p|^{-1})\right) \lesssim \pp^{-1},
\end{split}
\]
where $\cot z = \cos z/\sin z$. If $|p|< 1$, then $\kappa(p) \lesssim \pp \int_0^\pi \sin\theta \lesssim 1 \lesssim \pp^{-1}$. We conclude 
\[
\kappa(p) g(p) \lesssim \|g\|_{L^\infty(\R^3)}.
\]

Next, from \Cref{l:G-upper} and a standard convolution estimate, we have
\[
\begin{split}
\int_{\R^3} \frac 1 {|p-q|} g(q)\dd q &\leq \|g\|_{L^\infty(\R^3)}\int_{B_1(p)} \frac 1 {|p-q|} \dd q + \|g\|_{L^1(\R^3\setminus B_1(p))} \sup_{\R^3\setminus B_1(p)} \frac 1 {|p-q|}\\
& \lesssim \|g\|_{L^\infty(\R^3)} + \|g\|_{L^1(\R^3)},
\end{split}
\]
which establishes the conclusion of the lemma.
\end{proof}

\begin{lemma}\label{l:b-upper}
Suppose that $g\in L^\infty(\R^3) \cap L^1_1(\R^3)$.  Then there is a constant $C>0$, independent of $g$, such that the following hold.  
\begin{enumerate}

\item[(a)] For all $p\in\R^3$, 
\[
|b^g(p)| \leq C \pp^{\sfrac{1}{3}}\left(\|g\|_{L^\infty(\R^3)} + \|g\|_{L^1_1(\R^3)}\right).
\]

\item[(b)] If, in addition, $g\in L^\infty_1(\R^3)$, then, for all $p\in \R^3$,
\[
    |b^g(p)|
        \leq C \left(\|g\|_{L^\infty_1(\R^3)} + \|g\|_{L^1_1(\R^3)}\right).
\]

\end{enumerate}
\end{lemma}
Statement (a) will be used in the proof of exponential decay (\Cref{p:decay}) because it uses weaker norms of $f$. The benefit of statement (b) is that the right-hand side does not grow for large $|p|$, so it will be used in the regularity bootstrapping in the proof of \Cref{t:main}.
\begin{proof}
Using \Cref{l:GS} to bound $(\tau-2)^{-\sfrac{1}{2}}$, we have
\[
\begin{split}
\Lambda(p,q) (\tau - 2) 
    &= \frac{(\tau-1)^2}{\pp\qq \tau^{\sfrac{3}{2}} (\tau-2)^{\sfrac{1}{2}}} 
    \lesssim \frac {(\tau-1)^2}{\pp^{\sfrac{1}{2}} \qq^{\sfrac{1}{2}} \tau^{\sfrac{3}{2}}}  \frac {1} {|p-q|}. 
\end{split}
\]
Recalling that $\tau > 1$, we have $\tau^{-\sfrac{3}{2}} \leq (\tau-1)^{-\sfrac{3}{2}}$. Therefore,
%If $|p|>2$, we divide the integral according to $B_{\sfrac{|p|}{2}}(p)$. When $q\in B_{\sfrac{|p|}{2}}(p)$, we have $|q|\approx |p|$ and 
\[
\begin{split}
\frac {(\tau-1)^2}{\pp^{\sfrac{1}{2}} \qq^{\sfrac{1}{2}} \tau^{\sfrac{3}{2}}}
    & \leq \frac {(\tau-1)^{\sfrac{1}{2}}}{\pp^{\sfrac{1}{2}}\qq^{\sfrac{1}{2}}} 
    = \frac{ [\pp\qq - p\cdot q]^{\sfrac{1}{2}}}{\pp^{\sfrac{1}{2}} \qq^{\sfrac{1}{2}}}\\
    &= \left[ 1 - \frac{p\cdot q}{\pp \qq}\right]^{\sfrac{1}{2}}
    \lesssim 1,
\end{split}
\]
globally in $p$ and $q$. 
Returning to the integral defining $b_j^g$ in~\eqref{e.b}, we have
\begin{equation}\label{e.bfj}
|b^g_j| \lesssim  \int_{\R^3} \frac{|p_j+q_j|}{|p-q|} g(q) \dd q.
\end{equation}
Assuming $|p|>2$, we divide this integral into three parts. First, for $q \in B_{\sfrac{|p|}2}(0)$, since $|p-q| \approx |p|\approx |q|$, we have
\begin{equation}\label{e.Bp2}
    \int_{B_{\sfrac{|p|}{2}}(0)} \frac {|p_j+q_j|}{|p-q|} g(q) \dd q
        \lesssim \int_{B_{\sfrac{|p|}{2}}(0)} g(q) \dd q \leq \|g\|_{L^1(\R^3)}.
\end{equation}
Next, with $r = \pp^{-\sfrac13}$, we have using $|q|\approx |p|$ for $q\in B_r(p)$,
\[
    \int_{B_r(p)} \frac{|p_j+q_j|}{|p-q|} g(q) \dd q 
        \lesssim \pp \|g\|_{L^\infty(\R^3)} \int_{B_r(p)} \frac 1 {|p-q|} \dd q
        \lesssim \pp r^2 \|g\|_{L^\infty(\R^3)} 
        \lesssim \pp^{\sfrac13} \|g\|_{L^\infty(\R^3)}.
\]
Finally, for $q\in A:= \R^3\setminus (B_{r}(p)\cup B_{\sfrac{|p|}{2}}(0))$, we have $|p-q| \gtrsim r$, and
\[
\begin{split}
    \int_A \frac{|p_j+q_j|}{|p-q|} g(q) \dd q
        & \lesssim  \frac{1}{r}  \int_{A} \frac{|p|+|q|}{|q|}  |q| g(q) \dd q
        \lesssim \pp^{\sfrac{1}{3}} \|g\|_{L^1_1(\R^3)},
\end{split}
\]
since $(|p|+|q|)/|q| \lesssim 1$ in $A$.

On the other hand, if $|p|< 2$, starting from~\eqref{e.bfj} we have the simpler bound
\[
\begin{split}
    b_j^g &\lesssim \left(\int_{B_{1}(p)} + \int_{\R^3\setminus B_1(p)} \right)\frac{1 + |q|}{|p-q|} g(q) \dd q  \\
    & \lesssim \|g\|_{L^\infty(\R^3)} \int_{B_1(p)} \frac {|q|} {|p-q|} \dd q + \int_{\R^3\setminus B_1(p)} (1+|q|) g(q) \dd q\\
    &\lesssim (|p|+1)\|g\|_{L^\infty(\R^3)} + \|g\|_{L^1_1(\R^3)} \lesssim \|g\|_{L^\infty(\R^3)} + \|g\|_{L^1_1(\R^3)}.
\end{split}
\]
This establishes (a). 

To prove (b), notice that it is a consequence of (a) when $|p| \leq 2$.  We, thus, focus on the case $|p| > 2$.  We derive~\eqref{e.bfj} in the same way as in the proof of (a). We bound the integral over $B_{\sfrac{|p|}{2}}(0)$ as in~\eqref{e.Bp2}. Now, taking $r=1$, we still have $|q|\approx |p|$ in $B_r(p)$, so that
\[
    \int_{B_1(p)} \frac{|p_j+q_j|}{|p-q|} g(q) \dd q 
    \lesssim \pp \|g\|_{L^\infty_1(\R^3)} \int_{B_1(p)} \frac {\qq^{-1}} {|p-q|} \dd q  
    \lesssim  \|g\|_{L^\infty_1(\R^3)}.
\]
Next, we have to bound the integral over $A' := \R^3\setminus (B_1(p) \cup B_{\sfrac{|p|}{2}}(0))$, in which $|p-q| \gtrsim 1$ and $|q|\gtrsim |p|$, yielding
\[
\int_{A'}  \frac{|p_j+q_j|}{|p-q|} g(q) \dd q  \lesssim \int_{A'} \frac{|p| + |q|}{|q|} |q| g(q) \dd q \lesssim \|g\|_{L^1_1(\R^3)}.
\]
This concludes the proof.
\end{proof}

We also need an upper bound for the first order coefficient $B^g(p)$ that appears in the divergence-form representation of $\QRL$ (see~\eqref{e.divergence}).

\begin{lemma}\label{l:other-B-upper}
Suppose that $g\in L^\infty(\R^3) \cap L^1_1(\R^3)$.  Then there is a constant independent of $g$ such that the following holds.
\begin{enumerate}

\item[(a)] For all $p\in\R^3$,
\[
    |B^g(p)|
        \leq C \pp^{\sfrac{5}{3}}\left(\|g\|_{L^\infty(\R^3)} + \|g\|_{L^1_1(\R^3)}\right).
\]

\item[(b)] If, in addition, $g\in L^\infty_{\sfrac{5}{2}}(\R^3)$, then, for all $p\in\R^3$,
\[
|B^g(p)| \leq C \left(\|g\|_{L^\infty_{\sfrac{5}{2}}(\R^3)} + \|g\|_{L^1_1(\R^3)}\right).
\]
\end{enumerate}

\end{lemma}
\begin{proof}
Recall from~\eqref{e.other-B} that
\[
B_i^g(p) = 2\int_{\R^3} \Lambda(p,q)[(\tau-1)p_i-q_i] g(q) \dd q.
\]

Writing
\[
(\tau - 1)p_i - q_i = (\tau - 2)p_i + p_i - q_i,
\]
and using the definition of $\Lambda(p,q)$, we have
\[
\begin{split}
    \Lambda(p,q)[(\tau-1)p_i-q_i]
        &= \frac{(\tau-1)^2}{\pp\qq} \frac{1}{(\tau (\tau-2))^{\sfrac{3}{2}}} [(\tau-2)p_i + p_i-q_i] \\
        &=  \frac{(\tau-1)^2}{\pp\qq \tau^{\sfrac{3}{2}}}\left[ \frac {p_i} {(\tau - 2)^{\sfrac{1}{2}}} + \frac{p_i - q_i}{(\tau - 2)^{\sfrac{3}{2}}}\right].
\end{split}
\]
As in the proof of \Cref{l:b-upper}, we use \Cref{l:GS} to bound $(\tau-2)$ from below, as well as the inequality $\tau^{-\sfrac{3}{2}}\leq (\tau -1)^{-\sfrac{3}{2}}$, and obtain
\begin{equation}\label{e.Bifp}
\begin{split}
|B_i^g(p)|& \leq 2\int_{\R^3} \frac{(\tau-1)^{\sfrac{1}{2}}}{\pp\qq } \left[ \frac{ \sqrt 2 |p_i| \pp^{\sfrac{1}{2}} \qq^{\sfrac{1}{2}}}{|p-q|} + \frac{ 2^{\sfrac{3}{2}} \pp^{\sfrac{3}{2}}\qq^{\sfrac{3}{2}}|p_i - q_i|}{|p-q|^3} \right]g(q) \dd q\\
&\lesssim \int_{\R^3} \left[  \frac{|p|}{|p-q|} + \frac{ \pp \qq}{|p-q|^2} \right] g(q) \dd q,
\end{split}
\end{equation}
after using $(\tau - 1) \lesssim \pp\qq$. The first term in this integral can be bounded above in exactly the same way as the expression~\eqref{e.bfj} in the proof of \Cref{l:b-upper}, giving
\[
\int_{\R^3} \frac{|p|}{|p-q|} g(q) \dd q \lesssim \pp^{\sfrac{1}{3}}(\|g\|_{L^1_1(\R^3)} + \|g\|_{L^\infty(\R^3)}).
\]
For the second term in~\eqref{e.Bifp}, we let $r = \pp^{-\sfrac{1}{3}}$ and divide the integral into $B_r(p)$ and $\R^3\setminus B_r(p)$. First, we have
\[
    \int_{B_r(p)} \frac{\pp\qq}{|p-q|^2} g(q) \dd q 
        \lesssim \pp^2 \|g\|_{L^\infty} \int_{B_r(p)} |p-q|^{-2} \dd q
        \lesssim \pp^2 \|g\|_{L^\infty} r
        \lesssim \pp^{\sfrac{5}{3}} \|g\|_{L^\infty},
\]
since $\qq \lesssim \pp$ when $q\in B_r(p)$. Next, we have
\[
\int_{\R^3\setminus B_r(p)} \frac{\pp\qq}{|p-q|^2} g(q) \dd q \leq r^{-2} \pp \int_{\R^3\setminus B_r(p)} \qq g(q) \dd q \leq \pp^{\sfrac{5}{3}} \|g\|_{L^1_1(\R^3)}.
\]
This concludes the proof of (a).

To prove (b), we derive the upper bound~\eqref{e.Bifp} in the same way as (a). For the first term in~\eqref{e.Bifp}, we proceed as in the proof of \Cref{l:b-upper}(b) to obtain
\[
    \int_{\R^3} \frac{|p|}{|p-q|} g(q) \dd q
        \lesssim \|g\|_{L^1_1(\R^3)} + \|g\|_{L^\infty_1(\R^3)}.
\]
For the second term in~\eqref{e.Bifp}, we re-do the calculation in the proof of (a) with the new choice $r=\pp^{\sfrac{1}{2}}$. We now have
\[
\begin{split}
    \int_{B_r(p)} \frac{\pp\qq}{|p-q|^2} g(q) \dd q 
    &\lesssim \pp \|g\|_{L^\infty_{\sfrac{5}{2}}(\R^3)} \int_{B_r(p)} \qq^{-\sfrac{3}{2}}|p-q|^{-2} \dd q \\
    &\lesssim \pp^{-\sfrac{1}{2}}\|g\|_{L^\infty_{\sfrac{5}{2}}(\R^3)} r
    \lesssim\|g\|_{L^\infty_{\sfrac{5}{2}}(\R^3)},
\end{split}
\]
and
\[
    \int_{\R^3\setminus B_r(p)} \frac{\pp\qq}{|p-q|^2} g(q) \dd q
    \leq r^{-2} \pp \int_{\R^3\setminus B_r(p)} \qq g(q) \dd q
    \leq \|g\|_{L^1_1(\R^3)},
\]
as claimed.
\end{proof}

\section{Polynomial and exponential decay}\label{s:decay}

This section is devoted to the proof of our first main result, \Cref{t:decay}. We begin by showing that polynomial decay estimates in $p$ are propagated forward in time. The proof uses a barrier argument that relies on precise upper bounds for the coefficients $a^f$, $b^f$, and $c^f$.

\begin{proposition}\label{p:decay}
Let $f\geq 0$ be a solution to the relativistic Landau equation~\eqref{e.main} on $[0,T]\times\R^3_x\times\R^3_p$ for some $T>0$, satisfying 
\[
    E_f(t,x) \leq E_0
    \quad\text{ and }\quad
    \|f(t,x,\cdot)\|_{L^\infty(\R^3_p)} \leq F_0
\]
 uniformly in $(t,x) \in [0,T]\times\R^3$, for some constants $E_0$ and $F_0$. 
 Assume that, for some $k>0$, the initial data $f_{\rm in}: \R^3_x\times\R^3_p\to [0,\infty)$ satisfies 
\[
    \pp^k f_{\rm in}(x,p) \leq M_k
\]
uniformly in $(x,p)\in \R^6$, for some $M_k>0$. Then there is $\beta>0$, depending only on $k$, $E_0$, and $F_0$, such that, for all $(t,x,p) \in [0,T]\times \R^3_x \times \R^3_p$,
\[
\pp^k f(t,x,p) \leq e^{\beta t}M_k, \quad (t,x,p) \in [0,T]\times\R^6.
\]
\end{proposition}

\begin{proof}

For $\beta>0$ to be chosen later, and $\eps>0$ an arbitrary small constant, define the barrier 
\[
g(t,x,p) =  e^{\beta t}\left[M_k \pp^{-k} + \eps(\langle x \rangle + \pp)\right].
\]
By construction, $f(0,x,p) < g(0,x,p)$ for all $(x,p)\in \R^6$. Assume by contradiction that the inequality $f<g$ is not true in all of $[0,T]\times\R^6$. Then we claim there is a point $(t_0,x_0,p_0)$ with $t_0>0$, where $f$ and $g$ cross for the first time. Indeed, since $f$ is bounded and $g(t,x,p) \geq \eps (\langle x \rangle + \pp)$, we must have $f< g$ outside of  $[0,T]\times B_R \times B_R$ for some $R>0$ depending on $\eps$. Uniform continuity in $t$, and the fact that $f<g$ at $t=0$, then implies the existence of $(t_0,x_0,p_0)$ with $t_0>0$, where $f(t_0,x_0,p_0)=g(t_0,x_0,p_0)$ for the first time. 

At the point $(t_0,x_0,p_0)$, we have
\[
    \partial_tg \leq \partial_t f, \quad \nabla_x f = \nabla_x g ,\quad \nabla_p f = \nabla_p g, \quad D_p^2 f \leq D_p^2 g.
\]
Using these facts and the equation~\eqref{e.main} satisfied by $f$, we have, at $(t_0,x_0,p_0)$,
\begin{equation}\label{e.contra}
    \begin{split}
    \partial_t g + \frac {p_0}{\langle p_0\rangle} \cdot \nabla_x g
    &
    \leq \partial_t f + \frac {p_0} {\langle p_0\rangle} \cdot \nabla_x f
    = \tr (a^f D_p^2 f) + b^f \cdot \nabla_p f + c^f f 
    \\&
    \leq \tr( a^f D_p^2 g) + b^f \cdot \nabla_p g + c^f g,
    \end{split}
\end{equation}
where we also used the fact that $a^f$ is nonnegative-definite (which follows from \Cref{p:A-elliptic}). To keep the notation clean, we write $(t,x,p)$ instead of $(t_0,x_0,p_0)$ for the rest of this proof.

We will use~\eqref{e.contra} to derive a contradiction.  First, we bound the left side from below:
\[
\partial_t g + \frac p {\pp} \cdot \nabla_x g = \beta g  + \eps e^{\beta t} \frac p {\pp} \cdot \frac x {\langle x\rangle} \geq \beta g - \eps e^{\beta t}  \geq (\beta - 1) g.
\]
To estimate the right side of~\eqref{e.contra} from above, we use the upper bounds
\begin{equation}\label{e.deriv-bounds}
\begin{split}
|\nabla_p g| &= e^{\beta t}\left|-k M_k \pp^{-k-2} p   + \eps \frac p{\pp}\right| \lesssim \pp^{-1} g,\\
|D_p^2 g| &=e^{\beta t}\left| -kM_k\pp^{-k-2}\left( I- (k+2) \frac{p \otimes p} {\pp^2}\right)  + \eps \frac 1 {\pp} \left( I - \frac {p \otimes p}{\pp^2} \right)\right|\\
& \lesssim  e^{\beta t} \left(M_k \pp^{-k-2} + \eps \pp^{-1}\right)\\
& \lesssim \pp^{-2} g.
\end{split}
\end{equation}
Using the estimates for the coefficients contained in \Cref{p:A-elliptic}(b), \ref{l:c-upper}, and \ref{l:b-upper}(a), we have
\begin{equation}\label{e.abc-upper-bounds}
\begin{split}
a^f(t,x,p) &\lesssim \pp C_{\|f(t,x,\cdot)\|_{L^1_1(\R^3)},\|f(t,x,\cdot)\|_{L^3(\R^3)}} \lesssim \pp,\\
b^f(t,x,p) &\lesssim C\pp^{\sfrac{1}{3}}\left( \|f(t,x,\cdot)\|_{L^\infty(\R^3)} + \|f(t,x,\cdot)\|_{L^1_1(\R^3)}\right) \lesssim \pp^{\sfrac{1}{3}},\\
c^f(t,x,p) &\lesssim C\left(\|f(t,x,\cdot)\|_{L^\infty(\R^3)} + \|f(t,x,\cdot)\|_{L^1(\R^3)}\right) \lesssim 1,
\end{split}
\end{equation}
by our assumptions on $f$. These implied constants depend on $E_0$ and $F_0$. (Note that we have interpolated the $L^1$ and $L^3$-norms of $f$ between $L^1_1$ and $L^\infty$.) With~\eqref{e.deriv-bounds}, this implies
\[
\begin{split}
\tr(a^f D_p^2 g) + b^f \cdot \nabla_p g + c^f g &\lesssim \pp |D_p^2 g| + \pp^{\sfrac{1}{3}} |\nabla_p g| +  g \lesssim g,
\end{split}
\]
for an implied constant depending only on the quantities in the statement of the lemma. 
 Choosing $\beta$ so that $\beta -1$ is larger than this constant, we have derived a contradiction with~\eqref{e.contra}, and the proof is complete after sending $\eps\to 0$. 
\end{proof}

Next, we prove the corresponding statement about exponential decay estimates. The proof uses another barrier argument, but also relies on the polynomial upper bounds of the previous proposition, which is why the final estimate depends on $t$ in a double-exponential way. 

\begin{proposition}\label{p:decay2}
Let $f$ satisfy the assumptions of \Cref{p:decay}, and assume in addition that there exists $\sigma, N>0$ such that, for all $(x,p) \in \R^3\times\R^3$,
\[
    e^{\sigma \pp} f_{\rm in}(x,p)
        \leq N.
\]
Then $f$ satisfies , for all $(t,x,p) \in [0,T]\times \R^3\times\R^3$,
\[
e^{\sigma \pp} f(t,x,p) \leq N \exp(CN t e^{C t} ),
\]
where $C>0$ depends on $E_0$, $F_0$, and $\sigma$.
\end{proposition}

\begin{proof}
This proposition is also proven using a barrier argument. However, the upper bounds for the coefficients $a^f$ and $b^f$ in~\eqref{e.abc-upper-bounds} are not sufficient to close our argument. Therefore, we first get control on the $L^\infty_k$-norm of $f$, for some fixed $k>5$, by applying \Cref{p:decay}.  To be concrete, we arbitrarily choose $k=6$.  This yields 
\[
    f(t,x,p) 
        \lesssim N e^{\beta t}\pp^{-6},
\]
for an implied constant depending on $\sigma$, where $\beta$ is the constant from \Cref{p:decay} which depends only on $E_0$ and $F_0$. Note that this decay implies $f\in L^\infty_{t,x}(L^1_s)_p$ for some $s>2$. From \Cref{p:A-elliptic}(b), \Cref{l:c-upper}, and \Cref{l:b-upper}(b), we therefore have
\begin{equation}\label{e.abc-upper-bounds2}
\begin{split}
    a^f(t,x,p) &\lesssim  N e^{\beta t},\\
    b^f(t,x,p) &\lesssim N e^{\beta t},\\
    c^f(t,x,p) &\lesssim  1,
\end{split}
\end{equation}
for each $(t,x,p) \in [0,T]\times\R^6$, with implied constants depending on $F_0$, $E_0$, and $\sigma$. Crucially, these upper bounds do not grow for large $|p|$. 

With $\gamma>0$ to be determined, and $\eps>0$ a small constant, define
\[
h(t,x,p) = e^{\gamma t}\left[N e^{-\sigma \pp} + \eps(\langle x \rangle + \pp)\right].
\]
Assuming that $f< h$ is false, then following the same arguments as in the proof of \Cref{p:decay}, we conclude there is a point $(t_0,x_0,p_0)$ with $t_0>0$ where $f$ and $h$ cross for the first time, and that at this point,
\[
(\gamma - 1) h \leq \tr(a^f D_p^2 h) + b^f\cdot \nabla_p h + c^f h.
\]
By direct calculation, we have 
\[
\begin{split}
|\nabla_p h| &= e^{\gamma t} \left| \frac p {\pp} \left( - \sigma N e^{-\sigma\pp} + \eps\right)\right| \lesssim h,\\
|D_p^2 h| &= e^{\gamma t} \left| \sigma N \left[\sigma  \frac{p_i p_j}{\pp^2} - \left(\frac{\delta_{ij}}{\pp} - \frac{p_i p_j}{\pp^3}\right)\right] e^{-\sigma \pp} + \eps \left(\frac{\delta_{ij}}{\pp} - \frac{p_i p_j}{\pp^3}\right)\right| \lesssim  h.
\end{split}
\]
Therefore, using the upper bounds~\eqref{e.abc-upper-bounds2}, we have at the crossing point $(t_0,x_0,p_0)$,
\begin{equation}\label{e.contradiction2}
    (\gamma - 1) h
        \leq \tr(a^f D_p^2 h) + b^f\cdot \nabla_p h + c^f h \leq C N e^{\beta t} h
\end{equation}
for some constant $C$. Now, for any $\tau\in (0,T]$, the choice $\gamma = 1+ C N e^{\beta \tau}$ implies that~\eqref{e.contradiction2} cannot occur for any $(t_0,x_0,p_0)$ with $t_0\leq \tau$. This implies $f< h$ in the domain $[0,\tau]\times\R^6$, or
\[
    f(t,x,p)
        \leq N e^{(1+CN e^{\beta \tau})t} \left( e^{-\sigma \pp} + \eps(\langle x\rangle + \pp)\right)
\]
for all $(t,x,p)\in[0,T]\times\R^3\times\R^3$. 
Sending $\eps\to 0$, this gives, at $t=\tau$,
\[
f(\tau,x,p) \leq N e^{(1+CN e^{\beta \tau})\tau} e^{-\sigma \pp}.
\] 
Since $\tau\in (0,T]$ was arbitrary, the proof is complete. 
\end{proof}

Together, Propositions \ref{p:decay} and \ref{p:decay2} imply \Cref{t:decay}.

\section{Relativistic Schauder estimate}\label{s:schauder}

This section derives a Schauder estimate for the linear relativistic kinetic equation 
\begin{equation}\label{e.RFP}
	\left(\partial_t  + \frac p {\pp}\cdot \nabla_x\right) f
		= \tr(A D_p^2 f) + B\cdot \nabla_p f + s,
\end{equation}
 on $[0,T]\times \R^3_x\times\R^3_p$, 
where $A, B, s$ are bounded and H\"older continuous, with $A$ satisfying the ellipticity assumption~\eqref{e.A-elliptic} everywhere in $[0,T]\times \R^3_x\times\R^3_p$.

\subsection{Notation}

As part of our proof, we need to work with Galilean-invariant H\"older norms as in \cite{henderson2020smoothing}, for functions of $(t,x,v)$, where velocity $v$ is related to momentum $p$ by 
\begin{equation}\label{e.phi_psi}
    v = \varphi(p) = \frac p \pp
    \qquad\text{ and }\qquad
    p = \psi(v) = \frac{v}{\sqrt{1-|v|^2}}. 
\end{equation}
These norms are based on the Galilean shift
\begin{equation}\label{e.gal-shift}
z_0\circ z =  (t_0,x_0,v_0) \circ (t,x,v) = (t_0+t, x_0 + x + tv_0, v_0+v),
\end{equation}
with inverse
\begin{equation}\label{e.inverse-Gal}
(t_0,x_0,v_0)^{-1} \circ (t,x,v) = (t-t_0, x-x_0 - (t-t_0)v_0, v-v_0).
\end{equation}
In this section, we use the notation $z = (t,x,v)$, when there is no possible confusion with the notation $z = (t,x,p)$. With 
\begin{equation}
    \|(t,x,v)\|= |t|^{\sfrac{1}{2}} + |x|^{\sfrac{1}{3}} +|v|,    
\end{equation}
the Galilean kinetic distance is defined as
\begin{equation}\label{e.dG-def}
d_{\mathcal G}((t,x,v),(t_0,x_0,v_0)) = \|(t_0,x_0,v_0)^{-1} \circ (t,x,v)\| = |t-t_0|^{\sfrac{1}{2}} + |x - x_0 - (t-t_0)v_0|^{\sfrac{1}{3}} + |v-v_0|.
\end{equation}
The distance $d_\cG$ satisfies the key left-invariance property
\begin{equation}\label{e.left-invariance}
d_\cG(z_0\circ z_1, z_0 \circ z_2) = d_\cG(z_1,z_2).
\end{equation}
The corresponding H\"older norm is defined, for any $\Omega \subset \R^7$, via
\[
\begin{split}
[g]_{C^\alpha_{\mathcal G}(\Omega)} &= \sup_{(t,x,v),(t_0,x_0,v_0)\in \Omega} \frac {|g(t,x,v) - g(t_0,x_0,v_0)|}{d_{\mathcal G}((t,x,v),(t_0,x_0,v_0))^\alpha},\\
\|g\|_{C^\alpha_\cG(\Omega)} &= \|g\|_{L^\infty(\Omega)} + [g]_{C^\alpha_\cG(\Omega)}.
\end{split}
\]
We sometimes need to measure H\"older continuity in only the $t$ or $x$ variable. 
For this, we define the seminorms
\begin{equation}\label{e.Ca-tx}
\begin{split}
[g]_{C^\alpha_{\mathcal G, t}(\Omega)} 
&= \sup_{\substack{(t,x,v),(t_0,x,v)\in \Omega\\t\neq t_0}} \frac{ |g(t,x,v) - g(t_0,x,v)|} {|t-t_0|^\alpha + |v(t-t_0)|^{\frac{2\alpha}{3}}},\\
[g]_{C^\alpha_{\mathcal G, x}(\Omega)} 
&= \sup_{\substack{(t,x,v),(t,x_0,v)\in \Omega\\x\neq x_0}} \frac{ |g(t,x,v) - g(t,x_0,v)|} {|x-x_0|^\alpha}.
\end{split}
\end{equation}
Let us also define the seminorms of order $1+\alpha$, $2+\alpha$, and $3+\alpha$: for $\Omega\subset \R^7$,
\begin{equation}\label{e.seminorm-def}
\begin{split}
	[g]_{C^{1+\alpha}_{\mathcal G}(\Omega)} 
		&= [\nabla_v g]_{C^\alpha_{\mathcal G}(\Omega)} + [g]_{C^{(1+\alpha)/2}_{\mathcal G, t}(\Omega)} + [g]_{C^{(1+\alpha)/3}_{\mathcal G, x}(\Omega)},\\
	[g]_{C^{2+\alpha}_{\mathcal G}(\Omega)} 
		&= [D_v^2 g]_{C^\alpha_{\mathcal G}(\Omega)} + [g]_{C^\beta_{\mathcal G, t}(\Omega)} + [g]_{C^{(2+\alpha)/3}_{\mathcal G, x}(\Omega)}, \quad\text{ and}\\
	[g]_{C^{3+\alpha}_{\mathcal G}(\Omega)} 
		&= [D_v^3 g]_{C^\alpha_{\mathcal G}(\Omega)} + [\partial_t g]_{C^\alpha_{\mathcal G}(\Omega)} + [\nabla_x g]_{C^\alpha_{\mathcal G}(\Omega)}.
\end{split}
\end{equation}
Here, $\beta\in (0,1)$ is a fixed number that can be chosen arbitrarily close to 1. By an abuse of notation, we do not make $\beta$ explicit in the $C^{2+\alpha}_{\mathcal G}$ seminorm notation.  These definitions and notation choices are adapted from \cite{henderson2020smoothing}. %, and these seminorms allow us to obtain Schauder estimates that can be iterated to prove smoothness of solutions to~\eqref{e.main}.  

From \eqref{e.left-invariance}, it is clear that all the norms defined in this subsection are invariant under the Galilean shift $z \mapsto z_0 \circ z$. 

Let us also define the kinetic cylinder of radius $r>0$ centered at $z_0 \in \R^7$:
\begin{equation}\label{e.cyl}
\begin{split}
Q_r(z_0) 
& = \{(t,x,v): -r^2< t-t_0 \leq 0,  \, |x-x_0 - (t-t_0)v_0|< r^3, |v-v_0| < r\}.
\end{split}
\end{equation}
One often writes $Q_r = Q_r(0)$. Let $\delta_r(t,x,v)) = (r^2 t, r^3 x, rv)$ be the appropriate scaling operator.  It is straightforward to see that $Q_r(z_0) = (z_0^{-1} \circ \delta_r)Q_1$.

See Appendix \ref{s:appendix} for some useful properties of the distance $d_\cG$ and the $C^\alpha_\cG$-norm.

\subsection{Proof of the Schauder estimate}

Our method is based on transforming~\eqref{e.RFP} into a non-relativistic kinetic equation via the change of variables $u(t,x,v) = f(t,x,\psi(v))$, applying Schauder estimates, and transforming back via $f(t,x,p) = u(t,x,\phi(p))$. The non-relativistic Schauder estimate we use, which is a modification of \cite[Theorem 2.9]{henderson2020smoothing}, is contained in the next proposition. As mentioned above in \Cref{s:schauder-intro}, there is a large literature on Schauder estimates that apply to~\eqref{e.RFP}.  
The estimate of \cite{henderson2020smoothing} is convenient for us because it gives explicit polynomial dependence on the H\"older norms of the coefficients. The exponent of $\|A\|_{C^\alpha_\cG}$ here is likely not sharp, and with more work, we expect a similar estimate to hold with exponent $1+\frac 2 \alpha$, based on related Schauder estimates in other contexts such as \cite[Theorem~1.1]{HendersonWang}.  We note that this affects the sharpness in $\alpha$ in some of our results below, such as \Cref{t:first_schauder}.

In our proof of \Cref{p:old-schauder}, we need to modify the estimate in \cite{henderson2020smoothing}, which considered equations with no first order term in $v$.

\begin{proposition}\label{p:old-schauder}
Let $u$ be a solution to the linear kinetic equation
\[
\partial_t u + v\cdot \nabla_x u = \textup{\tr}(A D_v^2 u) + B\cdot \nabla_v u + s,
\]
on $Q_r(z_0) $ for some $r\in (0,1]$ and $z_0 = (t_0,x_0,v_0) \in \R^7$, and assume $A, B, s\in C^\alpha_{\mathcal G}(Q_r(z_0))$, and that
\[
	\lambda \Id
		\leq A(t,x,v)
		\leq \Lambda \Id, \quad (t,x,v) \in Q_r(z_0).
\] 
Then $u \in C^{2+\alpha}_{\mathcal G}(Q_{r/2}(z_0))$, and
\[
  [u]_{C^{2+\alpha}_{\mathcal G}(Q_{r/2}(z_0))}
    \leq Cr^{-(2+\alpha)}\left([s]_{C^\alpha_{\mathcal G}(Q_r(z_0))} + \left(1+\|A\|_{C^\alpha_{\mathcal G}(Q_r(z_0))}^{3+\alpha+\sfrac2\alpha} + \|B\|^{2+\alpha}_{C^\alpha_{\mathcal G}(Q_r(z_0))}\right) \|u\|_{L^\infty(Q_r(z_0))}\right),
\]
for a constant $C>0$ depending only on $\alpha$, $\beta$, $\lambda$, and $\Lambda$.

Here, the  H\"older seminorm in $(t,x,v)$ of order $2+\alpha$ is defined in~\eqref{e.seminorm-def}.
\end{proposition}

\begin{proof}
We first prove the inequality when $(t_0,x_0,v_0) = (0,0,0)$ and $r=1$. We will then use a scaling argument (and translation invariance of the relevant norms) to obtain the inequality
for general kinetic cylinders and $r \in (0,1]$.

From \cite[Theorem 2.12(a)]{henderson2020smoothing} with $g = B\cdot \nabla_v u + s$, we have
\[
[u]_{C^{2+\alpha}_{\mathcal G}(Q_{\sfrac12})}
\leq C_0\left([B\cdot \nabla_v u + s]_{C^\alpha_{\mathcal G}(Q_1)} + \|A\|_{C^\alpha_{\mathcal G}(Q_1)}^{3+\alpha+\frac2\alpha} \|u\|_{L^\infty(Q_1)}\right).
\]
Using \Cref{l:holder-product}, 
we have
\[
\begin{split}
[B\cdot \nabla_v u + s]_{C^\alpha_{\mathcal G}(Q_1)} &\leq [B\cdot \nabla_v u]_{C^\alpha_{\mathcal G}(Q_1)} + [s]_{C^\alpha_{\mathcal G}(Q_1)}\\
&\leq \|B\|_{L^\infty(Q_1)} [\nabla_v u]_{C^\alpha_{\mathcal G}(Q_1)} + [B]_{C^\alpha_{\mathcal G}(Q_1)} \|\nabla_v u\|_{L^\infty(Q_1)} + [s]_{C^\alpha_{\mathcal G}(Q_1)}.
\end{split}
\]
From the interpolation inequalities in \Cref{l:interp}, we then have
\[
[\nabla_v u]_{C^\alpha_{\mathcal G}(Q_1)} \leq \eps [u]_{C^{2+\alpha}_{\mathcal G}(Q_1)} + C\eps^{-1-\alpha} \|u\|_{L^\infty(Q_1)},
\]
for any $\eps >0$. Choosing $\eps = \frac 1 4 C_0^{-1}\|B\|_{L^\infty(Q_1)}^{-1}$, we obtain
\[
 \|B\|_{L^\infty(Q_1)} [\nabla_v u]_{C^\alpha_{\mathcal G}(Q_1)}
    \leq \frac 1 4 [u]_{C^{2+\alpha}_{\mathcal G}(Q_1)}
        + C \|B\|_{L^\infty(Q_1)}^{2+\alpha} \|u\|_{L^\infty(Q_1)}.
 \]
Similarly, in the interpolation inequality
\[
\|\nabla_v u\|_{L^\infty(Q_1)} \leq \eps^{1+\alpha} [u]_{C^{2+\alpha}_{\mathcal G}(Q_1)} + C\eps^{-1} \|u\|_{L^\infty(Q_1)},
\]
we choose $\eps^{1+\alpha} = \frac 1 4 C_0^{-1}[B]_{C^\alpha_{\mathcal G}(Q_1)}^{-1}$. Overall, we find
\begin{equation}\label{e.Q1est}
	\begin{split}
		[u]_{C^{2+\alpha}_{\mathcal G}(Q_{\sfrac12})}
			\leq &\frac 1 2 [u]_{C^{2+\alpha}_{\mathcal G}(Q_{1})}
			\\&
			 + C\left((1+\|B\|_{C^\alpha_{\mathcal G}(Q_1)}^{2+\alpha}) \|u\|_{L^\infty(Q_1)} + [s]_{C^\alpha_{\mathcal G}(Q_1)} + \|A\|_{C^\alpha_{\mathcal G}(Q_1)}^{3+\alpha +\frac2\alpha}\|u\|_{L^\infty(Q_1)}\right),
	\end{split}
\end{equation}
for some constant $C>0$. 

We would like to absorb the term $\frac 1 2 [D_v^2 u]_{C^\alpha_{\mathcal G}(Q_1)}$ into the left-hand side, even though it is defined on a larger domain $Q_1$. There is a standard way to get around this difficulty, which we sketch for the convenience of the reader: for $z_0=(t_0,x_0,v_0)\in Q_1$ and $r\in (0,1]$ such that $Q_r(z_0) \subset Q_1$, define the rescaled solution $u_r = u(t_0+r^2t, x_0 + r^3 x + r^2t v_0, v_0+ rv)$, which solves 
\begin{equation}\label{eq:prop6.1-rescaled}
\partial_t u_r + v\cdot \nabla_x u_r = \tr(A_r D_v^2 u_r) + r B_r\cdot \nabla_v u_r + r^2 s_r,
\end{equation}
where $A_r(t,x,v) = A(t_0+r^2t, x_0 + r^3 x + r^2t v_0, v_0+ rv)$ and $B_r, s_r$ are defined similarly. Applying~\eqref{e.Q1est} to $u_r$ and translating back to $u$, we obtain
\[
\begin{split}
	[u]_{C^{2+\alpha}_{\mathcal G}(Q_{r/2}(z_0))}
		&\leq \frac 1 2 [u]_{C^{2+\alpha}_{\mathcal G}(Q_r(z_0))}
			+ \frac{C}{r^{2+\alpha}} \left( \left(1+ r^{(1+\alpha)(2+\alpha)}\|B\|_{C^\alpha_{\mathcal G}(Q_r(z_0))}^{2+\alpha}\right) \|u\|_{L^\infty(Q_r(z_0))}\right.\\
		&\qquad \qquad \qquad\left. + r^\alpha [s]_{C^\alpha_{\mathcal G}(Q_r(z_0))} + r^{\alpha^2+3\alpha+2}\|A\|_{C^\alpha_{\mathcal G}(Q_r(z_0))}^{3+\alpha +\frac2\alpha}\|u\|_{L^\infty(Q_r(z_0))}\right)\\
		&\leq \frac 1 2  [u]_{C^{2+\alpha}_{\mathcal G}(Q_r(z_0))}
			+ \frac{C}{r^2} \left(\left( 1 +\|B\|_{C^\alpha_{\mathcal G}(Q_r(z_0))}^{2+\alpha}\right) \|u\|_{L^\infty(Q_r(z_0))} \right.\\
		&\qquad \qquad \qquad \left. +  [s]_{C^\alpha_{\mathcal G}(Q_r(z_0))} + \|A\|_{C^\alpha_{\mathcal G}(Q_r(z_0))}^{3+\alpha +\frac2\alpha}\|u\|_{L^\infty(Q_r(z_0))}\right),
\end{split}
\]
keeping only the most negative power of $r$. Next, for any $0< R_1 < R_2<1$, we apply the last estimate with $z_0\in Q_{R_1}$ and $r=R_2-R_1$, and taking a supremum over $z_0\in Q_{R_1}$, we obtain
\[
\begin{split}
[u]_{C^{2+\alpha}_{\mathcal G}(Q_{R_1})}
	\leq &\frac 1 2 [u]_{C^{2+\alpha}_{\mathcal G}(Q_{R_2})}
	\\&
	+ \frac{C}{(R_2-R_1)^{2}} \left((1+\|B\|_{C^\alpha_{\mathcal G}(Q_{R_2})}^{2+\alpha}) \|u\|_{L^\infty(Q_{R_2})} +  [s]_{C^\alpha_{\mathcal G}(Q_{R_2})} + \|A\|_{C^\alpha_{\mathcal G}(Q_{R_2})}^{3+\alpha +\frac2\alpha}\|u\|_{L^\infty(Q_{R_2})}\right).
\end{split}
\]
Defining $\omega(R) = [u]_{C^{2+\alpha}_{\mathcal G}(Q_{R})}$, we have shown
\[
	\omega(R_1)
		\leq \frac 1 2 \omega(R_2) + \frac{C}{(R_2-R_1)^{2}}\left((1+\|B\|_{C^\alpha_{\mathcal G}(Q_{1})}^{2+\alpha}) \|u\|_{L^\infty(Q_{1})} +  [s]_{C^\alpha_{\mathcal G}(Q_{1})} + \|A\|_{C^\alpha_{\mathcal G}(Q_{1})}^{3+\alpha +\frac2\alpha}\|u\|_{L^\infty(Q_{1})}\right),
\]
for any $0<R_1<R_2<1$. We can now apply \cite[Lemma 4.3]{hanlin} and obtain
\[
	\omega(R_1)
		\leq \frac{C}{(R_2-R_1)^2}\left((1+\|B\|_{C^\alpha_{\mathcal G}(Q_{1})}^{2+\alpha}) \|u\|_{L^\infty(Q_{1})} +  [s]_{C^\alpha_{\mathcal G}(Q_{1})} + \|A\|_{C^\alpha_{\mathcal G}(Q_{1})}^{3+\alpha +\frac2\alpha}\|u\|_{L^\infty(Q_{1})}\right),
\]
for any $R_1, R_2 \in (0,1)$. Choosing $R_1 = 1/2$ and $R_2 = 3/4$ yields
\[
    [u]_{C^{2+\alpha}_{\mathcal G}(Q_{\sfrac12})}
        \leq C_0 \left( [s]_{C^\alpha_{\mathcal G}(Q_1)} + \left( 1+\| A \|^{3+\alpha+\frac{2}{\alpha}}_{C^\alpha_{\mathcal G}(Q_1)} + \| B \|^{2+\alpha}_{C^\alpha_{\mathcal G}(Q_1)} \right) \| u \|_{L^\infty(Q_1)} \right) .
\]
If we now apply the above inequality to $u_r = u(t_0+r^2t, x_0 + r^3 x + r^2t v_0, v_0+ rv)$,
which solves~\eqref{eq:prop6.1-rescaled}, with modified coefficients $A_r$, $rB_r$, and $r^2 s_r$, we obtain
\[
\begin{split}
    &r^{2+\alpha} [u]_{C^{2+\alpha}_{\mathcal G}(Q_{r/2}(z_0))}
        \leq C_0 r^{2+\alpha} [s]_{C^\alpha_{\mathcal G}(Q_r(z_0))} \\
        &\quad\quad + C_0 \left( 1+ r^{3\alpha + \alpha^2 +2}\| A \|^{3+\alpha+\frac{2}{\alpha}}_{C^\alpha_{\mathcal G}(Q_r(z_0))} + r^{2+3\alpha +\alpha^2}\| B \|^{2+\alpha}_{C^\alpha_{\mathcal G}(Q_r(z_0))} \right) \| u \|_{L^\infty(Q_r(z_0))} .
\end{split}
\]
Remembering that $r \in (0,1]$, this completes the proof.
\end{proof}

Next, we have a higher-order Schauder-like estimate that is a modification of \cite[Theorem 2.10]{henderson2020smoothing}. It is needed in the second iteration of our bootstrapping regularity proof.  Indeed, one might hope to differentiate the equation, say in $v$, and re-apply the Schauder estimates to obtain regularity estimates of $D_v f$; however, the commutator of $\partial_t + v\cdot\nabla_x$ and $D_v$ involves a full derivative in $x$, which is not controlled by the estimate of \Cref{p:old-schauder}. 

Let us note that the exponent is likely not sharp here; see the discussion preceding \Cref{p:old-schauder}.  Also, the proof of \Cref{p:old-schauder2} is almost identical to the proof of \ref{p:old-schauder}, and is omitted.

\begin{proposition}\label{p:old-schauder2}
Let $u$, $A$, $B$, and $s$ be as in \Cref{p:old-schauder}, and assume in addition that $A, B, s\in C^{1+\alpha}_{\mathcal G}(Q_r(z_0))$. 
Then $u \in C^{3+\alpha}_{\mathcal G}(Q_{r/2}(z_0))$, and
\[
	[u]_{C^{3+\alpha}_{\mathcal G}(Q_{r/2}(z_0))}
		\leq Cr^{-(3+\alpha)}\left(\|s\|_{C^{1+\alpha}_{\mathcal G}(Q_r(z_0))}
			+ \left(1+\|A\|_{C^{1+\alpha}_{\mathcal G}(Q_r(z_0))}^{5+\alpha+\frac6\alpha}
				+ \|B\|^{3+\alpha}_{C^{1+\alpha}_{\mathcal G}(Q_r(z_0))}\right)
			\|u\|_{L^\infty(Q_r(z_0))}\right),
\]
for a constant $C>0$ depending only on $\alpha$, $\lambda$, and $\Lambda$.

Here, the  H\"older seminorms of order $1+\alpha$ and $3+\alpha$ are defined in~\eqref{e.seminorm-def}.
\end{proposition}

In order to apply the previous two propositions, we need to study the effect of the change of variables $v \mapsto p= \psi(v)$  on our equation.

\begin{lemma}\label{l:transform}
Let $f$ solve~\eqref{e.RFP} in 
some transformed cylinder 
\[
\widetilde Q_r(t_0,x_0,p_0) = \{(t,x,p)\in \R^7 : (t,x,\varphi(p)) \in Q_r(t_0,x_0,\varphi(p_0))\}.
\]
Then the function $u$ defined by
\[
\widetilde f(t,x,v) = f(t,x,\psi(v)),
\]
solves the non-relativistic kinetic Fokker-Planck equation
\begin{equation}\label{e.FP}
\partial_t \widetilde f + v\cdot \nabla_x \widetilde f = \textup{\tr} (\widetilde A D_v^2 \widetilde f) + \widetilde B \cdot \nabla_v \widetilde f + \widetilde s,
\end{equation}
in $Q_r(t_0,x_0,\varphi(p_0))$, with
\be\label{e.c020802}
\begin{split}
	\widetilde A(t,x,v)
    &= (1-|v|^2) (\Id - v\otimes v) A(t,x,\psi(v)) (\Id - v\otimes v),
    \\
	\widetilde B(t,x,v)
    &= \sqrt{1-|v|^2} (\Id - v\otimes v) B(t,x,\psi(v))
    \\
        &\quad
        + (1-|v|^2)\left[ \Big(3v\cdot A(t,x,\psi(v))v   - \textup{\tr} A(t,x,\psi(v))\Big)v - 2A(t,x,\psi(v))v\right],\\
	\widetilde s(t,x,v) &= s(t,x,\psi(v)).
\end{split}
\ee
\end{lemma}
\begin{proof} 
The proof is a direct calculation based on the transformation $\varphi(p) = p/\pp$. With $\widetilde f(t,x,v) = f(t,x,\psi(v))$, we clearly have $f(t,x,p) = \widetilde f(t,x,\varphi(p))$. The left side of~\eqref{e.RFP} becomes
\[
	(\partial_t + \varphi(p)\cdot \nabla_x) f(t,x,p)
		= (\partial_t + v\cdot \nabla_x) \widetilde f(t,x,v).
\]
For the right-hand side of~\eqref{e.RFP}, we first note that
\[
\partial_{j} \varphi_i(p) = \frac 1 \pp \left(\delta_{ij} - \frac{p_i p_j}{\pp^2}\right), \quad i,j=1,2,3,
\]
and 
\[
\partial_{jk}\varphi_i(p) = \frac{3p_ip_jp_k}{\pp^5} - \frac{p_i \delta_{jk} + p_j \delta_{ik} + p_k \delta_{ij}} {\pp^3}, \quad i,j,k = 1,2,3.
\]
Using $p = \psi(v) = v/\sqrt{1-|v|^2}$ and $\pp = 1/\sqrt{1-|v|^2}$, we have
\[
\begin{split}
\nabla_p f(t,x,p) &= D_p \varphi(p) \nabla_v \widetilde f(t,x,v)\\
&= \frac 1 \pp \left(I - \frac{p\otimes p}{\pp^2}\right) \nabla_v \widetilde f(t,x,v)\\
&= \sqrt{1-|v|^2} \left(I - v \otimes v\right) \nabla_v \widetilde f(t,x,v),
\end{split}
\]
and, using the convention of summing over repeated indices,
\begin{equation}\label{e.D2p}
\begin{split}
	[D_p^2 &f(t,x,p)]_{ij}
		= \partial_{p_j} (\partial_{p_i} f)\\
		&= \partial_{p_j} (\partial_{p_i} \varphi_k(p) \partial_{v_k} \widetilde f)\\
		&= \partial_{p_i p_j} \varphi_k(p) \partial_{v_k} \widetilde f + \partial_{p_i}\varphi_k(p) \partial_{p_j}\varphi_\ell(p) \partial_{v_kv_\ell} \widetilde f\\
		&= \left(\frac {3p_ip_jp_k} {\pp^5} - \frac{p_i \delta_{jk} + p_j \delta_{ik} + p_k \delta_{ij}} {\pp^3}\right) \partial_{v_k} \widetilde f + \frac 1 {\pp^2} \left( \delta_{ik} - \frac{p_i p_k}{\pp^2}\right)\left( \delta_{j\ell} - \frac{p_j p_\ell}{\pp^2}\right) \partial_{v_k v_\ell} \widetilde f\\
		&= (1-|v|^2)\Big[(3v_iv_jv_k - v_i \delta_{jk} - v_j \delta_{ik} - v_k\delta_{ij}) \partial_{v_k}\widetilde f +  \left(\delta_{ik} - v_i v_k\right)\left( \delta_{j\ell} - v_jv_\ell\right) \partial_{v_k v_\ell} \widetilde f\Big].
\end{split}
\end{equation}
The terms in the right-hand side of~\eqref{e.RFP} therefore become
\[
\begin{split}
B(t,x,p)\cdot \nabla_p f(t,x,p) &= \sqrt{1-|v|^2} (I- v\otimes v) B(t,x,\psi(v))\cdot\nabla_v \widetilde f,
\end{split}
\]
and, writing $A_{ij} = A_{ij}(t,x,\psi(v))$, 
\[
\begin{split}
A_{ij} [D_p^2 f(t,x,p)]_{ij} &=  A_{ij} (1-|v|^2) \left[(\delta_{ik} - v_iv_k)(\delta_{j\ell} - v_j v_\ell) \partial_{v_k v_\ell} \widetilde f(t,x,v) \right.\\
&\qquad \qquad\qquad\quad\left. + 3v_i v_j v_k \partial_{v_k} \widetilde f - v_i \partial_{v_j} \widetilde f - v_j\partial_{v_i} \widetilde f - v_k \delta_{ij} \partial_{v_k} \widetilde f\right],
\end{split}
\]
which implies
\[
\begin{split}
&\tr(A D_p^2 f(t,x,p))  \\
&= (1-|v|^2) \left[\tr((\Id - v\otimes v) A (\Id - v\otimes v) D_v^2 \widetilde f) + (3 v\cdot (A v) - \tr(A)) v\cdot \nabla_v \widetilde f - 2(Av)\cdot \nabla_v \widetilde f\right].
\end{split}
\]
Collecting terms, we obtain the statement of the lemma.
\end{proof}

The change of variables described in \Cref{l:transform} results in a diffusion matrix whose ellipticity estimates degenerate as $|v|\to 1$, in an anisotropic way. In particular, a direct calculation shows that
\begin{equation}
    \xi \cdot \tilde A \xi
    \approx
    \begin{cases}
        (1-|v|^2) |\xi|^2
        \qquad &\text{ if } \xi \perp v,\\
        (1-|v|^2)^3 |\xi|^2
        \qquad &\text{ if } \xi \parallel v.
    \end{cases}
\end{equation}
% \[
% \begin{cases}
% \lambda (1-|v|^2) |\xi|^2 \leq \xi\cdot [\widetilde A \xi] \leq \Lambda (1-|v|^2) |\xi|^2
% \qquad &\text{ if } \xi \perp v,\\
% \lambda (1-|v|^2)^3 |\xi|^2 \leq \xi\cdot [\widetilde A \xi] \leq \Lambda (1-|v|^2)^3 |\xi|^2
% \qquad & \text{ if } \xi \parallel v,
% \end{cases}
% \]
% where $\widetilde A = \widetilde A(t,x,v)$, which is given in \eqref{e.c020802}.
To compensate for this, we introduce the following change of variables: for some $z_0 = (t_0,x_0,v_0) \in (0,T]\times\R^3\times\R^3$, and $r>0$, 
let $P:\R^3\to \R^3$ be the linear transformation such that
\begin{equation}\label{e.P}
P\xi =
\begin{cases}
(1-|v_0|^2)^{\sfrac12} \xi \quad&\text{ if }\xi \perp v_0,\\
(1-|v_0|^2)^{\sfrac32} \xi \quad&\text{ if }\xi \parallel v_0,
\end{cases}
\end{equation}
and define 
\begin{equation}\label{e.g-def}
g(t,x,v) = \widetilde f(t_0+t,x_0+Px + tv_0,v_0+Pv), \quad (t,x,v) \in Q_{r}(0),
\end{equation}
with $\widetilde f$ as above. The function $g$ satisfies
\begin{equation}\label{e.g-eqn}
\left(\partial_t  + v\cdot \nabla_x\right) g = \tr(A_P D_v^2 g) + B_P \cdot \nabla_v g + s_P,
\end{equation}
where
\begin{equation}\label{e:AP}
\begin{split}
A_P(t,x,v) &= P^{-1} \widetilde A(t_0+t,x_0+Px +tv_0,v_0+Pv) P^{-1},\\
B_P(t,x,v) &= P^{-1} \widetilde B(t_0+t,x_0+Px +tv_0,v_0+Pv),\\
s_P(t,x,v) &= \widetilde s(t_0+t,x_0+Px +tv_0,v_0+Pv).
\end{split}
\end{equation}

Let us record some basic properties of the change of variables $v\mapsto v_0 + Pv$:
\begin{lemma}\label{l:ratio}
For $v_0 \in B_1$ and $v\in B_r$ with $r \leq\frac 1 3 (1-|v_0|^2)^{\sfrac12}$,  there holds
\begin{equation}\label{e:v0Pv}
|v_0 + Pv| < 1,
\end{equation}
and
\begin{equation}\label{e:ratio}
\frac 2 9 \leq \frac{1-|v_0+Pv|^2}{1-|v_0|^2} \leq  2.
\end{equation}
\end{lemma}
\begin{proof}
First, we note that 
\begin{equation}\label{e:Pv}
    |v_0 + Pv|
    \leq |v_0| + |Pv|
    \leq |v_0| + r (1-|v_0|^2)^{\sfrac12}
    \leq |v_0| + \frac13(1-|v_0|^2).
\end{equation}
It is straightforward to check that the last quantity is smaller than $1$ for $|v_0|<1$.  Hence, we have established \eqref{e:v0Pv}.

Let us consider the first inequality in \eqref{e:ratio}. Using \eqref{e:Pv}, we find
% \[
% |v_0+Pv| \leq |v_0| + |Pv| \leq |v_0| + \frac 1 3 (1-|v_0|^2),
% \]
% and therefore
\[
\begin{split}
1 - |v_0+Pv|^2  &\geq 1 - \left(|v_0| + \frac 1 3 (1-|v_0|^2)\right)^2 \\
&= 1- |v_0|^2 - \frac 1 9 (1-|v_0|^2)^2 - \frac 2 3 |v_0| (1-|v_0|^2) \\
&\geq \frac 2 9 (1-|v_0|^2),
\end{split}
\]
where we used $(1-|v_0|^2)^2 \leq (1-|v_0|^2)$ and $|v_0|\leq 1$.

Now we consider the second inequality in \eqref{e:ratio}.  The case when $|v_0|^2 \leq \sfrac{1}{2}$ is straightforward.  Indeed,
\begin{equation}
    1 - |v_0 + Pv|^2
    \leq 1
    \leq 2 (1 - |v_0|^2).
\end{equation}
Thus, we restrict to the case $|v_0|^2 \geq \sfrac{1}{2}$. Then 
\[
|v_0 + Pv| \geq |v_0| - |Pv| \geq |v_0| - \frac 1 3 (1-|v_0|^2).
\]
We note that the expression on the right is positive because
\begin{equation}
    \frac 1 3 (1- |v_0|^2)
    <\frac 1 6
    < \frac 1 {\sqrt 2}
    \leq |v_0|.
\end{equation}
We then have
\begin{equation}
\begin{split}
    1 - |v_0+Pv|^2
    &\leq 1 - \left( |v_0| - \frac13 (1- |v_0|^2)\right)^2
    \\&= 1- |v_0|^2 - \frac 1 9 (1-|v_0|^2)^2 + \frac 2 3 |v_0| (1-|v_0|^2)
    \\&
    \leq
    (1- |v_0|^2) + \frac 2 3 |v_0| (1-|v_0|^2)
    \leq \frac 5 3 (1-|v_0|^2).
\end{split}
\end{equation}
We used $|v_0|< 1$ in the last inequality.  This completes the proof.
\end{proof}

The next lemma establishes the key property satisfied by the matrix $A_P(t,x,v)$: uniform ellipticity with constants independent of $|v_0|$. 

\begin{lemma}\label{l:P}
Let the matrix $A_P(t,x,v)$ be defined by \eqref{e:AP}. There exists $c\in (0,\frac 1 3)$, depending only on $\lambda$ and $\Lambda$ defined in~\eqref{e.A-elliptic}, such that for any $v_0\in B_1$ and any $v\in B_r$ with $r= c(1-|v_0|^2)^{\sfrac12}$, there holds
\[
\theta |\xi|^2 \leq \xi\cdot [A_P(t,x,v) \xi] \leq \Theta |\xi|^2
    \qquad \text{ for all } \xi \in \R^3.
\]
Here $\Theta > \theta$ are positive constants depending only on $\lambda$ and $\Lambda$. 
\end{lemma}

\begin{proof}
For any $\xi \in\R^3$, we can write $\xi = \beta v_0 + \eta$, with $\eta \perp v_0$. We then have
\begin{equation}\label{e:APexpression}
\xi\cdot [A_P(t,x,v) \xi] = \beta^2 v_0\cdot [A_P(t,x,v) v_0] + 2\beta v_0 \cdot [A_P(t,x,v) \eta] + \eta \cdot [A_P(t,x,v) \eta].
\end{equation}
%Since $P^{-1} v_0 = (1-|v_0|^2)^{-3/2} v_0$ and $P^{-1} \eta = (1-|v_0|^2)^{-\sfrac12} \eta$, we can write
Our first step is to write this expression in terms of the matrix $A$. Define the matrix
\[
M_0 = \Id - (v_0+Pv) \otimes (v_0+Pv),
\]
and write $\widetilde A = \widetilde A(t,x,v)$ and $A = A(t_0+t,x_0+Px+t v_0, \psi(v_0 + Pv))$ for shorthand.  Let us show the computation completely only for the first term on the right in \eqref{e:APexpression} is
\begin{equation}\label{e:v0APv0}
\begin{split}
v_0\cdot [A_P(t,x,v) v_0] &= (P^{-1} v_0) \cdot [\widetilde A P^{-1} v_0] = (1-|v_0|^2)^{-3} v_0 \cdot [\widetilde A v_0] \\
&= (1-|v_0|^2)^{-3}v_0 \cdot\left[ (1-|v_0+Pv|^2) M_0 A M_0 v_0    \right]\\
&= (1-|v_0|^2)^{-3} (1-|v_0+Pv|^2)  ( M_0 v_0) \cdot (AM_0 v_0),
\end{split}
\end{equation}
where we used $P^{-1} v_0 = (1-|v_0|^2)^{-\sfrac32} v_0$. Similar calculations apply to the second and third terms in~\eqref{e:APexpression} yield
\begin{equation}\label{e:AP2}
\begin{split}
\xi\cdot [A_P(t,x,v) \xi] &= (1-|v_0+Pv|^2) \left[ (1-|v_0|^2)^{-3} \beta^2 (M_0 v_0)\cdot(AM_0v_0) \right. \\
&\quad \left. + 2(1-|v_0|^2)^{-2} \beta M_0 v_0 \cdot (AM_0 \eta) + (1-|v_0|^2)^{-1} (M_0\eta)\cdot (AM_0\eta)\right].%\\
%&= (1-|v_0+Pv|^2)  (1-|v_0|^{2})^{-1}\left[ T_1 + T_2\right],
\end{split}
\end{equation}
Next, we write
\begin{equation}
\begin{split}
M_0 v_0 &= (\Id - (v_0+Pv) \otimes (v_0+Pv)) v_0\\
& = m_1 v_0 + m_2 Pv,
\end{split}
\end{equation}
where we define
\begin{equation}
m_1 = 1-|v_0|^2 - (v\cdot v_0)(1-|v_0|^2)^{\sfrac32}
\qquad\text{ and }\qquad
m_2=  - |v_0|^2 + (v\cdot v_0)(1-|v_0|^2)^{\sfrac32}.
\end{equation}
Let us record some useful estimates for $m_1$ and $m_2$: since $|v_0|< 1$ and $|v| < \frac 1 3$, we see that
\begin{equation}\label{e:m1m2}
\begin{split}
\frac 2 3 (1-|v_0|^2) \leq m_1 &\leq \frac 4 3 (1-|v_0|^2) \qquad\text{ and}\\
|m_2| &\leq \frac 4 3|v_0|.
\end{split}
\end{equation}
We also have
%
%For the third term on the right in \eqref{e:APexpression}, we have
 %$P^{-1}\eta = (1-|v_0|^2)^{-\sfrac12} \eta$, and 
%\[
%\begin{split}
%\eta\cdot [A_P(t,x,v) \eta] &= (1-|v_0|^2)^{-1} \eta\cdot [\widetilde A \eta] \\
%&= (1-|v_0|^2)^{-1}  (1-|v_0+Pv|^2) (M_0 \eta)\cdot (AM_0\eta),
%\end{split}
%\]
%with
\[
\begin{split}
M_0 \eta &= (\Id - (v_0+Pv) \otimes (v_0+Pv)) \eta\\
%&= \eta - (v_0+Pv) (Pv\cdot \eta)\\
&= \eta - (v_0+Pv)(v\cdot P\eta)\\
&= \eta - (v_0+Pv)(v\cdot \eta) (1-|v_0|^2)^{\sfrac12}.
\end{split}
\]
Using these expressions for $M_0 v_0$ and $M_0 \eta$ in \eqref{e:AP2}, we have shown
%
%Now, returning to \eqref{e:APexpression} and using our above expressions for $v_0 \cdot [A_P v_0]$ and $\eta\cdot [A_P \eta]$, 
\[
\begin{split}
\xi\cdot [A_P(t,x,v) \xi] %&= (1-|v_0+Pv|^2) \left[ (1-|v_0|^2)^{-3} \beta^2 (M_0 v_0)\cdot(AM_0v_0) \right. \\
%&\quad \left. + (1-|v_0|^2)^{-2} \beta M_0 v_0 \cdot (AM_0 \eta) + (1-|v_0|^2)^{-1} (M_0\eta)\cdot (AM_0\eta)\right]\\
&= \frac{1-|v_0+Pv|^2}{1-|v_0|^{2}}\left[ T_1 + T_2\right],
\end{split}
\]
where
\begin{equation}\label{e.T_1}
T_1 =  \frac{m_1^2 \beta ^2}{(1-|v_0|^2)^2} v_0 \cdot Av_0 + \frac{2m_1\beta}{1-|v_0|^2} v_0 \cdot A\eta  + \eta \cdot A\eta
\end{equation}
and
\[\begin{split}
T_2 &= \frac{\beta^2}{(1-|v_0|^2)^2} \left(2m_1 m_2 v_0 \cdot APv + m_2^2 Pv\cdot APv\right)\\
&\quad + \frac{\beta}{1-|v_0|^2} \left[m_2 Pv\cdot A\eta-(m_1 v_0+m_2Pv)\cdot A(v_0+Pv)(v\cdot \eta) (1-|v_0|^2)^{\sfrac12} \right]\\
&\quad  + (v\cdot \eta)^2(1-|v_0|^2)(v_0+Pv) \cdot A(v_0+Pv) + 2(v\cdot \eta) (1-|v_0|^2)^{\sfrac12} \eta \cdot A(v_0+Pv).
%T_2 &= (1-|v_0|^2)^{-2} \beta^2 (2m_1 m_2 v_0 \cdot APv + m_2^2 Pv\cdot APv)\\
% &\quad + (1-|v_0|^2)^{-1} \beta \left[m_2 Pv\cdot A\eta-(m_1 v_0+m_2Pv)\cdot A(v_0+Pv)(v\cdot \eta) (1-|v_0|^2)^{\sfrac12} \right]\\
% &\quad  + (v\cdot \eta)^2(1-|v_0|^2)(v_0+Pv) \cdot A(v_0+Pv) + 2(v\cdot \eta) (1-|v_0|^2)^{\sfrac12} \eta \cdot A(v_0+Pv).
\end{split}
\]
\Cref{l:ratio} then implies
\begin{equation}\label{e:T1T2}
\frac 2 9 [T_1 + T_2] \leq \xi\cdot [A_P(t,x,v) \xi] \leq 2 [T_1 + T_2].
\end{equation}
The term $T_1$ will satisfy good upper and lower bounds, and the term $T_2$ can be made negligible compared to $T_1$ by choosing $c$ small enough, which, in turn, shrinks $r$. This is the goal of the rest of the proof.

We begin with the upper bound on $T_1$. Using Young's inequality and then~\eqref{e:m1m2} yields
\[
\begin{split}
T_1  &\leq \frac{2 m_1^2 \beta^2}{(1-|v_0|^2)^2} v_0\cdot Av_0 + 2 \eta\cdot A \eta
\lesssim \Lambda \left( \beta^2 |v_0|^2  + |\eta|^2 \right) = \Lambda |\xi|^2.
\end{split}
\]

Next, we consider a bound on $|T_2|$.  Recall that we aim to show that $T_2$ is negligible if $c$ is chosen sufficiently small.  We consider each of the four terms in $T_2$ separately.  For the first term, begin by observing that, by the definition~\eqref{e.P} of $P$ and the choice $r = c(1-|v_0|^2)^{\sfrac12}$,
\begin{equation}\label{e.Pv}
    |Pv| < r(1-|v_0|^2)^{\sfrac12} = c(1-|v_0|^2).
\end{equation}
Using this and~\eqref{e:m1m2}, we find
%     \qquad
%     m_1 \leq \frac 4 3 (1-|v_0|^2),
%     \qquad\text{ and }\qquad
%     |m_2| \leq \frac 4 3 |v_0|.
% \end{equation}
%Hence,
\[
\begin{split}
&\left|\frac{\beta^2}{(1-|v_0|^2)^2} (2m_1m_2v_0\cdot APv + m_2^2 Pv\cdot APv) \right|
\\&\qquad
\lesssim \frac{\Lambda \beta^2 |v_0|^2}{1-|v_0|^2} |Pv| + \frac{\Lambda \beta^2 |v_0|^2}{(1-|v_0|^2)^2}|Pv|^2
\leq \Lambda \beta^2 |v_0|^2 (c+c^2)
\lesssim c\Lambda \beta^2 |v_0|^2
\leq c \Lambda |\xi|^2.
\end{split}
\]
Next, we consider the second term in $T_2$. Using~\eqref{e:m1m2} and then~\eqref{e.Pv}, we find
\[
\begin{split}
&\left| \frac{\beta}{1-|v_0|^2} \left[m_2 Pv\cdot A\eta-(m_1 v_0+m_2Pv)\cdot A(v_0+Pv)(v\cdot \eta) (1-|v_0|^2)^{\sfrac12} \right]\right| \\
&\qquad \qquad \qquad \lesssim  \frac{\Lambda \beta |v_0|}{1-|v_0|^2} \left(  |Pv| |\eta| + \left( (1-|v_0|^2) + |Pv|\right) |v_0+Pv|  |v| |\eta| (1-|v_0|^2)^{\sfrac12}\right)
\\&\qquad \qquad \qquad
    \leq 
    \frac{\Lambda \beta |v_0|}{1-|v_0|^2} \left(
    c(1-|v_0|^2) |\eta|
    + \left( (1-|v_0|^2) + c(1-|v_0|^2)\right) |v_0+Pv|  |v| |\eta| (1-|v_0|^2)^{\sfrac12}\right).
\end{split}
\]
Canceling all $1-|v_0|^2$ factors, using that $v \in B_r$, and applying \Cref{l:ratio} yields
\begin{equation}
    \begin{split}
&\left| \frac{\beta}{1-|v_0|^2} \left[m_2 Pv\cdot A\eta-(m_1 v_0+m_2Pv)\cdot A(v_0+Pv)(v\cdot \eta) (1-|v_0|^2)^{\sfrac12} \right]\right| \\
\\&\qquad%\qquad\qquad
    \leq  \Lambda  \beta|v_0| \left( c|\eta| + (1+c) c |\eta|(1-|v_0|^2) \right)
%\\&\qquad\qquad\qquad
    \lesssim c \Lambda \beta |v_0| |\eta|
    \lesssim  c \Lambda \left(  \beta^2|v_0|^2 + |\eta|^2\right)
    = c \Lambda |\xi|^2.
\end{split}
\end{equation}
Consider now the last two terms in $T_2$.  By \Cref{l:ratio} and the fact that $v \in B_r \subset B_c$, we find
\[
\begin{split}
&\left|(v\cdot \eta)^2(1-|v_0|^2)(v_0+Pv) \cdot A(v_0+Pv) + 2(v\cdot \eta) (1-|v_0|^2)^{\sfrac12} \eta \cdot A(v_0+Pv)\right| \\
&\qquad \qquad \qquad \leq \Lambda \left(c^2 (1-|v_0|^2) |\eta|^2 + 2|v| |\eta|(1-|v_0|^2)^{\sfrac12} |\eta|\right)
\lesssim \Lambda \left( c^2 |\eta|^2 + c|\eta|^2\right)
\lesssim c \Lambda |\xi|^2.
\end{split}
\]
Adding the last three estimates, we have
\begin{equation}\label{e:T2}
|T_2| \lesssim c \Lambda |\xi|^2.
\end{equation}
Combining this with~\eqref{e:T1T2}, we have shown
\begin{equation}\label{e.xi_AP}
    \xi\cdot [A_P(t,x,v) \xi] \lesssim \Lambda |\xi|^2.
\end{equation}

%For the upper bound, we have
%\[
%\begin{split}
%\xi\cdot [A_P(t,x,v) \xi] &\leq 2 \beta^2 v_0\cdot [A_P(t,x,v) v_0] + 2\eta \cdot [A_P(t,x,v) \eta]\\
%&\leq 16\Lambda (\beta^2 |v_0|^2 + |\eta|^2) = 16\Lambda |\xi|^2,
%\end{split}
%\]
%by \eqref{e:v0-upper} and \eqref{e:eta-upper}. 

We now establish the lower bound. Since $T_2$ is small (depending on $c$), it suffices to find a lower bound on $T_1$ and then choose $c$ sufficiently small.
The first term in $T_1$ is clearly positive and, after a moment of rumination, will provide the lower bound we want.  Hence, we need to estimate the mixed term involving $v_0 \cdot A \eta$ in $T_1$ in more detail. This follows mainly from the ellipticity bounds on $A$.  Indeed, these bounds yield the following two inequalities:
\[
\begin{split}
\lambda (\beta^2|v_0|^2 + |\eta|^2) \leq \xi\cdot (A\xi) = \beta^2 v_0\cdot Av_0 + 2\beta v_0 \cdot A\eta + \eta \cdot A \eta,
\end{split}
\]
and
\[
\lambda (\beta^2|v_0|^2 + |\eta|^2) \geq \frac \lambda \Lambda \left( \beta^2 v_0 \cdot A v_0 + \eta \cdot A \eta\right).
\]
Combining these two inequalities and rearranging, we obtain
\[
2\beta v_0 \cdot A \eta  \geq - \left( 1 - \frac \lambda \Lambda\right) \left( \beta^2 v_0 \cdot A v_0 + \eta \cdot A \eta\right).
\]
In fact, the same inequality holds if we replace $\xi = \beta v_0 + \eta$ with $\eps^{\sfrac12} \beta v_0 + \eps^{-\sfrac12} \eta$ for any $\eps>0$. With the choice 
%\[
%2\beta v_0 \cdot A \eta \geq -  \left( 1 - \frac \lambda \Lambda\right) \left( \eps \beta^2 v_0 \cdot A v_0 + \eps^{-1} \eta \cdot A \eta\right).
%\]
%Using this with 
$\eps = m_1 (1-|v_0|^2)^{-1}$, we have
\begin{equation}\label{e:good-lower-bound}
    \frac{2\beta m_1}{1- |v_0|^2}
    v_0 \cdot A \eta
    \geq -  \left( 1 - \frac \lambda \Lambda\right)
    \left( \frac{m_1^2 \beta^2}{(1-|v_0|^2)^2}  v_0 \cdot A v_0 + \eta \cdot A \eta\right).
\end{equation}
Using this estimate, as well as \eqref{e:m1m2}, in the definition~\eqref{e.T_1} of $T_1$, we have
\begin{equation}\label{e:T1}
\begin{split}
T_1
&=  \frac{m_1^2 \beta ^2}{(1-|v_0|^2)^2} v_0 \cdot Av_0 + \frac{2m_1\beta}{1-|v_0|^2} v_0 \cdot A\eta  + \eta \cdot A\eta
\\&\geq
    \frac \lambda \Lambda  \left(
    \frac{m_1^2 \beta^2}{(1-|v_0|^2)^2} v_0 \cdot A v_0 + \eta \cdot A \eta\right)
\gtrsim
    \frac{\lambda^2} \Lambda \left(\beta^2 |v_0|^2 + |\eta|^2\right)
\approx \frac {\lambda^2}{\Lambda} |\xi|^2.
\end{split}
\end{equation}
Recalling \eqref{e:T2}, if we make the choice
\[
c \approx \frac{\lambda^2}{\Lambda^2},
\]
then \eqref{e:T1T2},~\eqref{e:T2}, and~\eqref{e:T1} imply
\[
\xi\cdot [A_P(t,x,v) \xi] \gtrsim \frac {\lambda^2} {\Lambda} |\xi|^2,
\]
as desired.
\end{proof}

Next, we record bounds for the coefficients $B_P$ and $s_P$ in \eqref{e:AP}:
\begin{lemma}\label{l:Bs}

    Let $A_P$, $B_P$, and $s_P$ be defined by \eqref{e:AP}. Assume that $\xi\cdot A(t,x,p) \xi \leq \Lambda |\xi|^2$ and $B,s\in L^\infty([0,T]\times \R^6)$.
    
    For any $v_0\in B_1$ and $r\leq (\sfrac13) (1-|v_0|^2)^{\sfrac12}$, there holds
    \[
    \begin{split}
        |B_P(t,x,v)| &\leq \frac{12}{1-|v_0|^2} (\|B\|_{L^\infty([0,T]\times \R^6)} + \Lambda) \quad\text{ and}\\
        |s_P(t,x,v)| &\leq \|s\|_{L^\infty([0,T]\times \R^6)}.
    \end{split}
    \]
\end{lemma}

\begin{proof}
    From~\eqref{e.c020802}, we have for any $v\in B_1$,
    \[
    |\widetilde B(t,x,v)| \leq (1-|v|^2)^{\sfrac12}\|B\|_{L^\infty} + 6(1-|v|^2) \Lambda.
    \]
    With \Cref{l:ratio}, this implies
    \[
    \begin{split}
    |\widetilde B(t_0 +t, x_0+Px+tv_0,v_0+Pv)| &\leq \sqrt 2 \sqrt{1-|v_0|^2} \|B\|_{L^\infty} + 12(1-|v_0|^2) \Lambda\\
    &\leq 12\sqrt{1-|v_0|^2} (\|B\|_{L^\infty} +\Lambda).
    \end{split}
    \]
    By~\eqref{e.P}, observe that $|P^{-1}w| \leq (1-|v_0|^2)^{-\sfrac32}|w|$ for any $w\in \R^3$.  Then the formula for $B_P$ in \eqref{e:AP} implies
    \[
    |B_P(t,x,v)| \leq \frac{|\widetilde B(t_0+t,x_0+Px+tv_0,v_0+Pv)|}{(1-|v_0|^2)^{\sfrac32}}
    \leq \frac{12}{1-|v_0|^2} (\|B\|_{L^\infty} + \Lambda).
    \]
    The claimed estimate for $s_P$ follows immediately from the formulas \eqref{e:AP} and \eqref{e.c020802}.
 \end{proof}

Next, we examine how H\"older norms change under this transformation:

\begin{lemma}\label{l:hP-holder}
    Let $h(t,x,v)$ be defined on $Q_r(z_0)$ for some $z_0 \in \R^7$ and $r>0$. For some $v_0 \in B_1$, define
    \[
    h_P(t,x,v) =  h(t_0+t,x_0+Px+tv_0, v_0+Pv).
    \]
    where $P$ is given by \eqref{e.P}. Define the set
    \begin{equation}\label{e:PQ}
    \mathcal PQ_r(z_0) = \{ (t_0+t,x_0+Px + tv_0, v_0 + Pv) : (t,x,v) \in Q_r\}.
    \end{equation}
     
    Then $h_P$ is defined in $Q_{r}$, and
    \begin{align}
        [h_P]_{C^\alpha_\cG(Q_r)} &\leq  [h]_{C^\alpha_\cG(\mathcal P Q_r(z_0))},
        \label{e:holder1}
        \\
        [h]_{C^\alpha_\cG(Q_{r'}(z_0))} &\leq \frac{[h_P]_{C^\alpha_\cG(Q_r)}}{(1-|v_0|^2)^{\frac{3\alpha}{2}}},
        \label{e:holder2}
        \\
        [h_P]_{C^{1+\alpha}_\cG(Q_r)} &\leq [h]_{C^{1+\alpha}_\cG(\mathcal P Q_r(z_0))},
        \label{e:holder3}
        \\
        [h]_{C^{2+\alpha}_\cG(Q_{r'}(z_0))} &\leq 
        \frac{[h_P]_{C^{2+\alpha}_\cG(Q_r)}}{(1-|v_0|^2)^{3(1+\sfrac\alpha2)}},
        \label{e:holder4}
        \\
        [h]_{C^{3+\alpha}_\cG(Q_{r'}(z_0))} &\leq \frac{[h_P]_{C^{3+\alpha}_\cG(Q_r)}}{(1-|v_0|^2)^{\frac{3(3+\alpha)}{2}}},
        \label{e:holder5}
    \end{align}
     where $r' = (1-|v_0|^2)^{\sfrac32} r$, and $\alpha\in (0,1)$ is such that the right-hand sides are finite.
\end{lemma}

\begin{proof}
First, recalling the Galilean shift defined in \eqref{e.gal-shift}, we define
%\[
%P(Q_r(0)) = \{(t,Px,Pv) : (t,x,v) \in Q_r(0)\}.
%\]
%or 
%\[
%z_0 \circ z = (t_0+t, x_0 + x + tv_0, v_0 + v),
%\]
%and
\[
h_{z_0}(t,x,v) = h(z_0\circ z).
\]
Then $h_P(t,x,v) = h_{z_0}(t,Px,Pv)$. Since the distance $d_\cG$ is left-invariant (see \eqref{e.left-invariance}), we have
\begin{equation}\label{e.holder-shift}
[h_{z_0}]_{C^\alpha_\cG(Q_r)} = [h]_{C^\alpha_\cG(Q_r(z_0))},
\end{equation}
for any $r>0$.

It remains to estimate the effect of the change of variables $(t,x,v) \mapsto (t,Px,Pv)$ on H\"older norms. Recall from~\eqref{e.P} that 
\begin{equation}
    |Pw| \leq |w|
    \quad\text{ and }
    |P^{-1} w| \leq \frac{|w|}{(1-|v_0|^2)^{\sfrac32}}
    \quad\text{ for all }
    w\in \R^3.
\end{equation}
Therefore,  for $(t_i,x_i,v_i) \in Q_r$,
%\[
%Q_r^P(z_0) =\{z_0 \circ (t, Px, Pv) : (t,x,v) \in Q_r(0)\}.
%\]
%Note that $Q_r^P(z_0) \subset Q_r(z_0)$.
%
%
%First, we estimate how much the $P$ change of variables stretches the Galilean distance:
\begin{equation}\label{e:dG-compare1}
\begin{split}
d_\cG((t_1,Px_1,Pv_1),(t_2,Px_2,Pv_2))
    &= |t_2 - t_1|^{\sfrac12} + |Px_2 - Px_1 - (t_2-t_1)Pv_1|^{\sfrac13} + |Pv_2 - Pv_1|\\
&\leq  d_\cG((t_1,x_1,v_1),(t_2,x_2,v_2)),
\end{split}
\end{equation}
and
\begin{equation}\label{e:dG-compare2}
\begin{split}
d_\cG((t_1,x_1,v_1),(t_2,x_2,v_2))
&= |t_2 - t_1|^{\sfrac12} + |P^{-1}(Px_2 - Px_1 - (t_2-t_1)Pv_1)|^{\sfrac13} + |P^{-1}(Pv_2 - Pv_1)|
\\
&\leq \frac{d_\cG((t_1,Px_1,Pv_1),(t_2,Px_2,Pv_2))}{(1-|v_0|^2)^{\sfrac32}}.
\end{split}
\end{equation}
Let us also note that for $z = (t,x,v) \in Q_r$, there holds $(t,Px,Pv) \in \mathcal P Q_r$. 
We then have, for $(t_i,x_i,v_i)\in Q_r$, 
\[
\begin{split}
\frac{|h_P(t_1,x_1,v_1) - h_P(t_2,x_2,v_2)|}{d_{\cG}((t_1,x_1,v_1),(t_2,x_2,v_2))^\alpha} &= \frac{|h_{z_0}(t_1,Px_1,Pv_1) -  h_{z_0}(t_2, Px_2, Pv_2)|}{d_{\cG}((t_1,x_1,v_1),(t_2,x_2,v_2))^\alpha}\\
&\leq \frac{|h_{z_0}(t_1,Px_1,Pv_1) - h_{z_0}(t_2, Px_2, Pv_2)|}{d_{\cG}((t_1,Px_1,Pv_1),(t_2,Px_2,Pv_2))^\alpha}\\
&\leq   [ h_{z_0}]_{C^\alpha_\cG(\mathcal P Q_r)} = [h]_{C^\alpha_\cG(\mathcal P Q_r(z_0))},
\end{split}
\]
by \eqref{e.holder-shift}. Taking the supremum over $Q_r$, we obtain \eqref{e:holder1}. Conversely, if $(t_i,x_i,v_i) \in Q_{r'}$, then $(t,P^{-1}x,P^{-1}v) \in Q_{r}$, since $r'(1-|v_0|^2)^{-\sfrac32} = r$. In this case, we have
\[
\begin{split}
\frac{|h_{z_0}(t_1,x_1,v_1) - h_{z_0}(t_2,x_2,v_2)|}{d_{\cG}((t_1,x_1,v_1),(t_2,x_2,v_2))^\alpha}
&= \frac{|h_P(t_1,P^{-1}x_1,P^{-1}v_1) - h_P(t_2, P^{-1}x_2, P^{-1}v_2)|}{d_{\cG}((t_1,x_1,v_1),(t_2,x_2,v_2))^\alpha}
\\
&\leq \frac{1}{(1-|v_0|^2)^{\frac{3\alpha}{2}}}\frac{|h_P(t_1,P^{-1}x_1,P^{-1}v_1) - h_P(t_2, P^{-1}x_2, P^{-1}v_2)|}{d_{\cG}((t_1,P^{-1}x_1,P^{-1}v_1),(t_2,P^{-1}x_2,P^{-1}v_2)^\alpha}
\\
&\leq \frac{[h_P]_{C^\alpha_\cG(Q_{r})}}{(1-|v_0|^2)^{\frac{3\alpha}{2}}}.
\end{split}
\]
Taking the supremum over $Q_{r'}$ and using \eqref{e.holder-shift} again, the proof of \eqref{e:holder2} is complete.

Next, we consider the $C^{1+\alpha}_{\cG}$ seminorm. Since $\nabla_v h_{P}(t,x,v) = P \nabla_v h_{z_0}(t,Px,Pv)$, \eqref{e:dG-compare1} and \eqref{e.holder-shift} imply
\[
[\nabla_v h_{P}]_{C^\alpha_\cG(Q_{r})} \leq (1-|v_0|^2)^{\sfrac12} [\nabla_v h]_{C^\alpha_{\cG}(\mathcal P Q_r(z_0))}.
\]
For the $[\cdot]_{C^{(1+\alpha)/2}_{\cG, t}}$ term, for any $(t_1,x,v), (t_2,x,v) \in Q_r$, since $|Pv(t_2-t_1)| \leq  |v(t_2-t_1)|$, we have
\[
\frac{|h_P(t_1,x,v) - h_P(t_2,x,v)|}{|t_2-t_1|^{\frac{(1+\alpha)}2}+ |v(t_2-t_1)|^{\frac{(1+\alpha)}3}}
\leq \frac{|h_{z_0}(t_1,Px,Pv) - h_{z_0}(t_2,Px,Pv)|} {|t_2-t_1|^{\frac{(1+\alpha)}2}+ |Pv(t_2-t_1)|^{\frac{(1+\alpha)}3}} \leq [h_{z_0}]_{C^\alpha_{\cG,t}(\mathcal P Q_r)}
\]
The $[\cdot]_{C^\alpha_{\cG,t}}$ seminorms are also left-invariant, so after taking the supremum, we have 
\[
[h_P]_{C^{(1+\alpha)/2}_{\cG,t}(Q_r)} \leq [h]_{C^{(1+\alpha)/2}_{\cG,t}(\mathcal P Q_r(z_0))}.
\]
Similarly, since $|P(x_2-x_1)|^{(1+\alpha)/3} \leq |x_2 - x_1|^{(1+\alpha)/3}$, we find
\[
[h_P]_{C^{(1+\alpha)/3}_{\cG,x}(Q_r)} \leq [h]_{C^{(1+\alpha)/3}_{\cG,x}(\mathcal P Q_r(z_0))}.
\]
Adding up the three terms of $[h_P]_{C^{1+\alpha}_{\cG}(Q_r)}$, we have shown \eqref{e:holder3}. 

For the $C^{2+\alpha}_\cG$ seminorm, we note that 
\[
|D_v^2 h(t_0+t,x_0+x+tv_0,v_0+v)| = |D_v^2 h_{z_0}(t,x,v)| \leq |P^{-1}|^{2} |D_v^2 h_P(t,P^{-1} x, P^{-1}v)|,
\]
so \eqref{e:dG-compare2}  and \eqref{e.holder-shift} imply
\begin{equation}\label{e.second-v-deriv}
[D_v^2 h]_{C^\alpha_\cG(Q_{r'}(z_0))}
    \leq \frac{[D_v^2 h_P]_{C^\alpha_\cG(Q_{r})}}{(1-|v_0|^2)^{3 + \frac{3\alpha}2}}.
\end{equation}
Next, fix any $(t_1,x,v), (t_2,x,v) \in Q_{r'}$.  By~\eqref{e.P}, we have
\begin{equation}
    |P^{-1} v (t_2 - t_1)| \leq \frac{|v| |t_2-t_1|}{(1-|v_0|^2)^{\sfrac32}}.
\end{equation}
Hence, we have, for $\beta\in (0,1)$,
\[
\frac{|h_{z_0}(t_1,x,v) - h_{z_0}(t_2,x,v)|}{|t_2-t_1|^{\beta} + |v(t_2 - t_1)|^{\frac{2\beta}3}}
\leq \frac{1}{(1-|v_0|^2)^\beta} \frac{|h_{z_0}(t_1,P^{-1}x,P^{-1}v) - h_{z_0}(t_2,P^{-1}x,P^{-1}v)|}{|t_2 - t_1|^{\beta} + |P^{-1} v (t_2 - t_1)|^{\frac{2\beta}3} }.
\]
Taking the supremum and using left-invariance, we obtain
\[
[h]_{C^{\beta}_{\cG,t}(Q_{r'}(z_0))} \leq [h_P]_{C^\beta_{\cG,t}(Q_{r})}.
\]
Also, fix any $(t,x_1,v), (t,x_2,v)\in Q_r$. By~\eqref{e.P}, we have
\begin{equation}
    |P^{-1}(x_2-x_1)|^{\frac{(2+\alpha)}3}
    \leq \frac{|x_2-x_1|^\frac{2+ \alpha}{3}}{(1-|v_0|^2)^{\frac{(2+\alpha)}2}}.
\end{equation}
Using this, we find
\[
[h]_{C^{(2+\alpha)/3}_{\cG,x}(Q_{r'}(z_0))} \leq \frac{[h_P]_{C^{(2+\alpha)/3}_{\cG,x}(Q_{r})}}{(1-|v_0|^2)^{\frac{(2+\alpha)}2} }.
\]
Adding up the last three estimates, we have shown \eqref{e:holder4}.

Finally, for the $C^{3+\alpha}_\cG$ seminorm, we first note that
\[
[D_v^3 h]_{C^\alpha_\cG(Q_{r'}(z_0))}
\leq \frac{[D_v^3 h_P]_{C^\alpha_\cG(Q_r)}}{(1-|v_0|^2)^{\frac92+\frac{3\alpha}2}},
\]
by the same argument that was used to derive \eqref{e.second-v-deriv}. Next, we have
\[
\nabla_x h(z_0\circ z) = \nabla_x h_{z_0}(z) = P^{-1} \nabla_x h_P(z),
\]
which together with \eqref{e:dG-compare2} implies
\[
[\nabla_x h]_{C^\alpha_\cG(Q_{r'}(z_0))} \leq (1-|v_0|^2)^{-3/2 - 3\alpha/2} [h_P]_{C^\alpha_\cG(Q_r)}.
\]
For the term involving $\partial_t h$, we note that
\[
h(z')
    = h_{z_0}(z_0^{-1}\circ z')
    = h_{z_0}(t' - t_0, x' - x_0 - (t'-t_0)v_0, v'-v_0),
\]
so we have $\partial_t h(z') = \partial_t h_{z_0}(z_0^{-1} \circ z') - v_0 \cdot \nabla_x h_{z_0}(z_0^{-1}\circ z')$, and
\[
\begin{split}
[\partial_t h]_{C^\alpha_\cG(Q_{r'}(z_0))}
    &\leq
    \frac{[\partial_t h_P]_{C^\alpha(Q_r)}}{(1-|v_0|^2)^{\frac{3\alpha}2}} + \frac{[\nabla_x h_P]_{C^\alpha_\cG(Q_r)}}{(1-|v_0|^2)^{\frac32+\frac{3\alpha}2}}
    \leq
    \frac{[h_P]_{C^{3+\alpha}_\cG(Q_r)}}{(1-|v_0|^2)^{\frac32+\frac{3\alpha}2}},
\end{split}
\]
since $|v_0|\leq 1$. This completes the proof of \eqref{e:holder5}. The proof is finished.
\end{proof}

We are now ready to prove our main linear Schauder estimates. They are stated in terms of the H\"older norm with kinetic scaling: let
\begin{equation}\label{e.dk}
	 d_{\rm kin}(z_1,z_2)
     = |t_1-t_2|^{\sfrac12}+|x_1-x_2|^{\sfrac13} + |p_1 - p_2|,
\end{equation}
and
\begin{equation}\label{e.k-seminorm}
[g]_{C^\alpha_{\rm kin}(\Omega)}
=  \sup_{z_1,z_2\in \Omega} \frac{|g(z_1) - g(z_2)|}{d_{\rm kin}(z_1,z_2)^\alpha},
\end{equation}
and $\|g\|_{C^\alpha_{\rm kin}(\Omega)} = \|g\|_{L^\infty(\Omega)} + [g]_{C^\alpha_{\rm kin}(\Omega)}$. 

The corresponding seminorms of order $1+\alpha$, $2+\alpha$, and $3+\alpha$ are defined by
\begin{equation}\label{e.1alpha}
	\begin{split}
	[g]_{C^{1+\alpha}_{\rm kin}(\Omega)}
    &= [\nabla_p g]_{C^\alpha_{\rm kin}(\Omega)}
    + [g]_{C^{\frac{1+\alpha}2}_t(\Omega)}
    + [g]_{C^{\frac{1+\alpha}3}_x(\Omega)},
    \\
	[g]_{C^{2+\alpha}_{\rm kin}(\Omega)}
    &= [D_p^2 g]_{C^\alpha_{\rm kin}(\Omega)}
    + [g]_{C^{\frac{2\beta}3}_t(\Omega)}
    + [g]_{C^{\frac23+\frac{3\alpha}2}_x(\Omega)},
    \\
	[g]_{C^{3+\alpha}_{\rm kin}(\Omega)}
    &= [D_p^3 g]_{C^\alpha_{\rm kin}(\Omega)}
    + [\partial_t g]_{C^{\alpha}_{\rm kin}(\Omega)}
    + [\nabla_x g]_{C^{\alpha}_{\rm kin}(\Omega)},
	\end{split}
\end{equation}
where, as above, $\beta\in (0,1)$ is an arbitrary constant. These non-standard seminorms are  shorthand for the various combinations of $(t,x,v)$ regularity that arise from the proof of our estimate.

\begin{theorem}\label{t:first_schauder}
Let $f$ solve~\eqref{e.RFP} 
in $[0,T]\times \R^6$ for some $T>0$.  Assume that  $A$ satisfies the uniform ellipticity hypothesis~\eqref{e.A-elliptic} for some $\Lambda > \lambda >0$. 

For any $\alpha\in (0,1)$, there exists $m_1>0$ depending only on $\alpha$, such that if 
\[
	 A, B, \pp^{m_1} s \in C^\alpha_{\rm kin}\left([0,T]\times\R^6\right),
\]
and $\pp^{m_1} f \in L^\infty([0,T]\times\R^6)$, then, with $\alpha' = \sfrac{4\alpha}9$,
\begin{equation}\label{e.first-schauder}
		\begin{split}
		  %  [D_p^2 f]_{C^{2\alpha/3}_{k}([\tau,T]\times\R^6)} &+ [f]_{C^{2\beta/3}_t([\tau,T]\times\R^6)} + [f]_{C^{(2+\alpha)/3}_x([\tau,T]\times\R^6)}\\
			[f]_{C^{2+\alpha'}_{\rm kin}([\tau,T]\times\R^6)}
             &\leq  C \left(1 + \tau^{-(1+\frac\alpha3)}\right)\left(  \|\pp^{m_1} s\|_{C^\alpha_{\rm kin}([0,T]\times\R^6)}
			\right. 
            \\
            &\quad \left.
            + \left( \|A\|_{C^{\alpha}_{\rm kin}([0,T]\times\R^6)}^{3 + \frac{2\alpha}3 + \frac3\alpha} + \| B\|_{C^\alpha_{\rm kin}([0,T]\times\R^6)}^{2+\frac{2\alpha}3}\right) \|\pp^{m_1} f\|_{L^\infty([0,T]\times\R^6)}  \right),
		\end{split}
\end{equation}
%for any $\beta\in (0,1)$, 
where the constant $C>0$ depends on $\lambda$, $\Lambda$, and $\alpha$. 

Furthermore, there exists $m_2>0$ depending only on $\alpha$, such that if the coefficients of~\eqref{e.RFP} additionally satisfy
\[
 A,  B, \pp^{m_2} s \in C^{1+\alpha}_{\rm kin}([0,T]\times\R^6),
\]
and $\pp^{m_2} f \in L^\infty([0,T]\times\R^6)$, then $f$ satisfies the further estimate
\begin{equation}\label{e.second-schauder}
		\begin{split}
		%&[D_p^3 f]_{C^{2\alpha/3}_{k}([\tau,T]\times\R^6)} + [\partial_t f]_{C^{2\alpha/3}_k([\tau,T]\times\R^6)} + [\nabla_x f]_{C^{2\alpha/3}_k([\tau,T]\times\R^6)}\\
		[f]_{C^{3+\alpha'}_{\rm kin}([\tau,T]\times\R^6)}&\leq  C \left(1 + \tau^{-\left(\frac32+\frac\alpha3\right)}\right)\left(  \|\pp^{m_2} s\|_{C^{1+\alpha}_{\rm kin}([0,T]\times\R^6)}
		\right.
        \\
		&\quad \quad \left.
        + \left( \|A\|_{C^{1+\alpha}_{\rm kin}([0,T]\times\R^6)}^{5 + \frac{2\alpha}3 + \frac9\alpha} + \| B\|_{C^{1+\alpha}_{\rm kin}([0,T]\times\R^6)}^{3+\frac{2\alpha}3}\right) \|\pp^{m_2} f\|_{L^\infty([0,T]\times\R^6)}  \right).     .
	\end{split}
	%	[f]_{C^{3+\alpha}_{\mathcal L}(Q^{\rm rel}_{R'}(z_0))}
	%		\leq C \ppo^\omega \left( \|s\|_{C^{1+\alpha}_{\mathcal L}(Q_R^{\rm rel}(z_0))}
	%		\left(1+ \|A\|_{C^{1+\alpha}_{\mathcal L}(Q^{\rm rel}_R(z_0))}^{5+\alpha+\frac6\alpha} \|B\|_{C^{1+\alpha}_{\mathcal L}(Q_R^{\rm rel}(z_0))}^{3+\alpha}\right)
	%		\|f\|_{L^\infty(Q^{\rm rel}_R(z_0))} \right),
\end{equation}
The constant $C$ depends only on $\alpha$, $\lambda$, and $\Lambda$.
\end{theorem}

%We note that the exponents $\omega$ in \Cref{t:first_schauder} are not the same as the exponent in \Cref{p:old-schauder} or \Cref{p:old-schauder2}. Indeed, in the following proof, the value of $\omega$ changes from line to line. We only aim to show that the dependence on $\ppo$ in the inequalities \eqref{e.first-schauder} and \eqref{e.second-schauder} is polynomial, though we do not compute the explicit exponent.

%Recall the definition ?????? of the $C^{\alpha}_{k}$ and $C^{1+\alpha}_{k}$ seminorms. 

\begin{proof}

%For $R>0$ given, we will study the function $f$ in the domain $[0,T]\times \R^3_x\times B_R$. At the end of the proof, we will study the limit as $R\to \infty$. 

Define $\widetilde f:[0,T]\times\R^3_x \times B_1 \to [0,\infty)$ as in \Cref{l:transform}. 
For $R>0$ and $\tau \in [0,T]$, define
\begin{equation}\label{e.Omega-def}
\begin{split}
\Omega_{\tau,R}
    &= [\tau,T]\times \R^3_x \times \{|p|\leq R\}
    \qquad\text{and}
    \\
\widetilde \Omega_{\tau, \frac{R}{\langle R \rangle}}
    &= [\tau,T]\times\R^3_x \times \left\{|v|\leq \frac{R}{\langle R \rangle}\right\}.
\end{split}
\end{equation}
We will derive a weighted estimate on domains of this form, and send $R\to \infty$ at the end of the proof.

Fix $\tau>0$ and any $(t_0,x_0,v_0) \in \widetilde \Omega_{\tau, R/\langle R\rangle}$.  Define 
\[
r= c(1-|v_0|^2)^{\sfrac12}\min\{1, \sqrt{\tau/2}\},
\]
where $c$ is the constant from \Cref{l:P}, and define $g$ on $Q_r$ as in \eqref{e.g-def}. By \Cref{l:ratio}, $v_0+Pv \in B_1$ if $v\in B_r$, so  $g$ is well-defined on $Q_r$. 

The function $g$ satisfies the non-relativistic kinetic equation \eqref{e.g-eqn} in $Q_r$, and by \Cref{l:P}, the matrix $A_P(t,x,v)$ satisfies uniform lower and upper ellipticity estimates with constants $\Theta>\theta>0$ independent of $v_0$. 
These uniform ellipticity bounds allow us to apply \Cref{p:old-schauder} and obtain
\begin{equation}\label{e.g-est}
  [g]_{C^{2+\alpha}_{\mathcal G}(Q_{r/2})}
    \lesssim \frac{1}{r^{2+\alpha}} \left(\left[s_P\right]_{C^\alpha_{\mathcal G}(Q_{r})} + \left(\|A_P\|_{C^\alpha_{\mathcal G}(Q_{r})}^{3+\alpha+\frac2\alpha} + \|B_P\|^{2+\alpha}_{C^\alpha_{\mathcal G}(Q_{r})}\right) \|g\|_{L^\infty(Q_{r})}\right).
\end{equation}
Next, we translate this into an estimate for $f$. 

We first understand how the source term behaves in the new variables.  
From \Cref{l:hP-holder}, we have
\[
[s_P]_{C^\alpha_\cG(Q_r)} \leq [\widetilde s]_{C^\alpha_\cG(\mathcal P Q_r(z_0))}.% \leq (1-|v_0|^2)^{-\alpha/6} \|s\|_{C^\alpha_\cG([0,T]\times\R^3_x\times B_{R+r})}.
\]
\Cref{l:ratio} implies that for $(t,x,v) \in \mathcal P Q_r(z_0)$, one has $1-|v|^2 \gtrsim (1-|v_0|^2)$.  Hence, by~\eqref{e.phi_psi},
\begin{equation}\label{e.psi upper bound}
|\psi(v)|
    = \left|\frac v {\sqrt{1-|v|^2}} \right|
    \lesssim  \left|\frac v {\sqrt{1-|v_0|^2}}\right|
    \lesssim \langle p_0 \rangle,
\end{equation}
where $p_0 = \psi(v_0)$. Since $|v_0| \leq \sfrac{R}{\langle R\rangle}$, we have $\ppo \leq R$, and, if $(t,x,v) \in \cP Q_r(z_0),$
\[
|\psi(v)| \lesssim \ppo \leq R.
\]
We have shown the following useful fact:
\[
\{(t,x,\psi(v)) : (t,x,v) \in \mathcal PQ_r(z_0)\} \subset \Omega_{0,CR}.
\]
Applying \Cref{l:holder-compare}, we now have
\[
[s_P]_{C^\alpha_\cG(Q_r)} \lesssim \langle R \rangle^{\frac{9\alpha}2} \| s\|_{C^{3\alpha/2}_{\rm kin}(\Omega_{0,CR})} . %\leq C\| \pp^{9\alpha/2} s\|_{C^{3\alpha/2}_k([0,T]\times\R^6)}.
\]

We now investigate how the diffusion coefficient behaves in the new variables. 
For $A_P$, we see from \Cref{l:P} that $\|A_P\|_{L^\infty(Q_r)}$ is bounded by $\Lambda \approx \|A\|_{L^\infty(Q_r(z_0))}$. Next, from the definition \eqref{e:AP} of $A_P$ and \Cref{l:hP-holder}, we have
\[
[A_P]_{C^\alpha_\cG(Q_r)} \leq (1-|v_0|^2)^{-3} [\widetilde A]_{C^\alpha_\cG(\mathcal PQ_r(z_0))}.
\]
Recall from~\eqref{e.c020802} that
\begin{equation}
    \widetilde A(t,x,v) = (1-|v|^2)(\Id - v\otimes v) A(t,x,\psi(v)) (\Id - v\otimes v). 
\end{equation}
Since the functions $1-|v|^2$ and $(\Id - v\otimes v)$ are uniformly Lipschitz in $B_1$, we use \Cref{l:holder-product} to find
\[
\begin{split}
[\widetilde A]_{C^\alpha_\cG(\mathcal P Q_r(z_0))} &\lesssim 
    \left[\|A(t,x,\psi(v))\|_{L^\infty(\mathcal P Q_r(z_0))}
    + [A(t,x,\psi(v))]_{C^\alpha_\cG(\mathcal P Q_r(z_0))}\right].
\end{split}
\]
Arguing as in our above estimate of $[s_P]_{C^\alpha_\cG}$ and using that $(1-|v_0|^2)^{-\sfrac12} = \ppo \leq C R$, we find 
\[
\begin{split}
[A_P]_{C^\alpha_\cG(Q_r)}
    &\lesssim  R^6(\|A\|_{L^\infty(\Omega_{0,CR})}+ R^{\frac{9\alpha}2}\|A\|_{C^{3\alpha/2}_k(\Omega_{0,3R})} )\\
    &\lesssim
    R^{6+\frac{9\alpha}2}\| A\|_{C^{\frac{3\alpha}2}_k(\Omega_{0,CR})}.
%&\leq C\|\pp^{6+9\alpha/2} A\|_{C^{3\alpha/2}_k([0,T]\times\R^6)}.
\end{split}
\]
%by \Cref{l:holder-product}. 

The argument for $[B_P]_{C^\alpha_\cG}$ is similar, but we need to account for the fact that $(1-|v|^2)^{\sfrac12}$ is not uniformly $\alpha$-H\"older continuous if $\alpha>\sfrac 1 2$. First, using~\eqref{e.phi_psi}, we have 
\begin{equation}
    |\nabla_v (1-|v|^2)^{\sfrac12}|
    = \left|\frac{v}{(1-|v|^2)^{\sfrac12}}\right|
    = |\psi(v)|
    \lesssim R.
\end{equation}
In the last inequality, we used~\eqref{e.psi upper bound}.
Then, since $r\leq c(1-|v_0|^2)^{\sfrac12} \leq c \langle R\rangle$, we have
\[
    \frac{(1-|v|^2)^{\sfrac12} - (1-|v'|^2)^{\sfrac12}}{|v-v'|^\alpha}
    \leq R |v-v'|^{1-\alpha}
    \lesssim \langle R\rangle^\alpha,
\]
when $|v-v'|< r$. Using this, and arguments similar to those applied to $s_P$ and $A_P$, we eventually find
\[
    [B_P]_{C^\alpha_\cG(Q_r)}
    \lesssim R^{\max\{\alpha, \frac{9\alpha}2 - 1\}}\left(\|B\|_{C^{\frac{3\alpha}2}_{\rm kin}(\Omega_{0,CR})} + \| A\|_{C^{\frac{3\alpha}2}_{\rm kin}(\Omega_{0,CR})}\right).
\]
From \Cref{l:Bs}, we find
\begin{equation}
    \|B_P\|_{L^\infty(Q_r)}
    \lesssim 
    \langle R\rangle^2 \left(\|B\|_{L^\infty(\Omega_{0,CR})} + \|A\|_{L^\infty(\Omega_{0,CR})}\right)    
\end{equation}
Combining this with the H\"older seminorm, we have
\[
\|B_P\|_{C^\alpha_\cG(Q_r)}
    \lesssim
    R^2\left( \|B\|_{C^{\frac{3\alpha}2}_{\rm kin}(\Omega_{0,CR})}
    + \| A\|_{C^{\frac{3\alpha}2}_{\rm kin}(\Omega_{0,CR})}\right),
\]
since $\alpha < \sfrac23$. 

Using these upper bounds in the right-hand side of \eqref{e.g-est}, and applying \Cref{l:hP-holder} to the left-hand side of \eqref{e.g-est}, we have
\begin{equation}\label{e.intermediate}
\begin{split}
[\widetilde f]_{C^{2+\alpha}_\cG(Q_{\frac{r'}2}(z_0))}
&\leq
    \frac{[ g]_{C^{2+\alpha}_\cG(Q_{r/2})}}{(1-|v_0|^2)^{3(1+\frac{\alpha}2)}} \\
&\lesssim 
    \frac{\langle R \rangle^{6(1+\frac\alpha2)}}{r^{(2+\alpha)}}
    \bigg[  \|s\|_{C^{\frac{3\alpha}2}_{\rm kin}(\Omega_{0,CR})}
\\
&\quad
    +   \left( R^{(6+\frac{9\alpha}2)(3+\alpha+\frac{2}{\alpha})} \| A\|_{C^{\frac{3\alpha}2}_{\rm kin}(\Omega_{0,CR})}^{3 + \alpha + \frac2\alpha} + R^{4+2\alpha}\| B\|_{C^{\frac{3\alpha}2}_{\rm kin}(\Omega_{0,CR})}^{2+\alpha}\right) \|f\|_{L^\infty(\Omega_{0,CR})} \bigg]
\\
&\lesssim
    \left(1 + \frac{1}{\tau^{1 + \frac\alpha2}}\right)\bigg[  \|\pp^{m_1} s\|_{C^{\frac{3\alpha}2}_{\rm kin}([0,T]\times\R^6)}
\\
&\quad
    +  \left(1+ \|A\|_{C^{\frac{3\alpha}2}_{\rm kin}([0,T]\times\R^6)}^{3 + \alpha + \frac2\alpha} + \| B\|_{C^{\frac{3\alpha}2}_{\rm kin}([0,T]\times\R^6)}^{2+\alpha}\right) \|\pp^{m_1} f\|_{L^\infty([0,T]\times\R^6)} \bigg],
\end{split}
\end{equation}
where $r' = (1-|v_0|^2)^{\sfrac32} r$, and $m_1>0$ is a constant depending on $\alpha$, which may change line-by-line throughout the rest of the proof.

For $\tau\in (0,T]$, our next step is to replace the left side of \eqref{e.intermediate} with an estimate for $\widetilde f$ over $\widetilde \Omega_{\tau,R/\langle R \rangle}$ (recall~\eqref{e.Omega-def}). For any $z,z_0\in \widetilde \Omega_{\tau, R/\langle R \rangle}$ with $t\leq t_0$, we define $r$ and $r'$ as above, based on the point $z_0$. If $d_\cG(z,z_0) \geq \sfrac{r'}2$, then we have
\[
\begin{split}
\frac{|D_v^2 \widetilde f(z) - D_v^2 \widetilde f(z_0)|} {d_\cG(z,z_0)^\alpha}
    &\leq 2 \left(\frac{2}{r'}\right)^{\alpha} \|D_v^2 \widetilde f\|_{L^\infty(\widetilde \Omega_{\tau,R/\langle R \rangle})}
\\
&\leq
    \frac 1 2 [\widetilde f]_{C^{2+\alpha}_\cG(\widetilde \Omega_{\tau,R/\langle R \rangle})}
    + \frac{C}{(r')^{2+\alpha}} \|\widetilde f\|_{L^\infty(\widetilde \Omega_{\tau,R/\langle R \rangle})},%\\
%&\leq C (1-|v_0|^2)^{-2\alpha}\left( [\widetilde f]_{C^{2+\alpha}_\cG(\widetilde \Omega_{\tau, R/\langle R \rangle})} + \|\widetilde f\|_{L^\infty(\widetilde \Omega_{\tau, R/\langle R \rangle})}\right),
\end{split}
\]
where we used the interpolation inequality from \Cref{l:interp} in the second inequality above.  If $d_\cG(z,z_0) < \sfrac{r'}2$, then a quick calculation with \eqref{e.dG-def} shows that $z\in Q_{r'/2}(z_0)$, and we can apply \eqref{e.intermediate}. Together, these two cases imply 
\[
[D_v^2 \widetilde f]_{C^\alpha_\cG(\widetilde \Omega_{\tau, R/\langle R \rangle})}
\leq 
\frac 1 2 [\widetilde f]_{C^{2+\alpha}_\cG(\widetilde \Omega_{\tau,R/\langle R \rangle})} + \mathcal T,
\]
where we use $\mathcal T$ to denote the right-hand side of \eqref{e.intermediate}, up to a multiplicative constant. Note that we used
\begin{equation}
    \frac{1}{(r')^{2+\alpha}}\|\widetilde f\|_{L^\infty(\widetilde \Omega_{\tau, R/\langle R \rangle})}
    \lesssim \mathcal T.
\end{equation} 

%since $r' \leq c(1-|v_0|^2)^{3/2} r \leq c (1-|v_0|^2)^2$.

Similarly, for $\beta\in (0,1)$ and $(t_0,x_0,v_0), (t,x_0,v_0) \in [\tau,T]\times\R^3_x\times B_{R/\langle R \rangle}$ with $t\leq t_0$, if $|t-t_0| \geq \sfrac{r'}4$, we have
\[
\begin{split}
\frac{|\widetilde f(t_0,x_0,v_0) - \widetilde f(t,x_0,v_0)|}{|t-t_0|^\beta + |v_0(t-t_0)|^{2\beta/3}} 
    &\leq
    2\left(\frac{2}{r'}\right)^{\beta} \|\widetilde f\|_{L^\infty(\widetilde \Omega_{\tau,R/\langle R \rangle})}. %\\
%&\leq C (1-|v_0|^2)^{-2}\|f\|_{L^\infty([\tau,T]\times\R^3_x\times B_{R/\langle R\rangle})}.
\end{split}
\]
On the other hand, if $|t-t_0| < \sfrac{r'}4$, then since $|v_0|< 1$, we can show $(t,x_0,v_0) \in Q_{r'/2}(z_0)$ and \eqref{e.intermediate} applies. After applying a similar argument to the seminorm $[\widetilde f]_{C^{\alpha}_{\cG,x}}$, we conclude
\[
[\widetilde f]_{C^{2+\alpha}_\cG(\widetilde \Omega_{\tau,R/\langle R \rangle})} \leq \frac 1 2 [\widetilde f]_{C^{2+\alpha}_\cG(\widetilde \Omega_{\tau,R/\langle R \rangle})} + \mathcal T.
\]
The first term on the right can clearly be absorbed into the left side.

Next, we relate the $C^{2+\alpha}_\cG$-norm of $\widetilde f$ to regularity norms of $f$. 
From the calculations in the proof of \Cref{l:transform}, we see that
\[
	\begin{split}
		D_p^2 f(t,x,p)  
			= (1-|v|^2) \big[&\left(\Id - v \otimes v\right) (D_v^2 \widetilde f) \left( \Id - v \otimes v\right)
				\\&
				+ 3 v\cdot \nabla_v \widetilde f(v\otimes v) - v\otimes \nabla_v \widetilde f - \nabla_v \widetilde f \otimes v - (v\cdot\nabla_v \widetilde f) \Id\big].
	\end{split}
\]
Since $\Id - v\otimes v$ and $1-|v|^2$ are uniformly Lipschitz on $B_1$, we use \Cref{l:holder-product} and \Cref{l:holder-compare} to obtain
\[
\begin{split}
[D_p^2 f(t,x,p)]_{C^{\frac{2\alpha}3}_{\rm kin}(\Omega_{\tau,R})}
    &\lesssim
    \left(
        [D_v^2 \widetilde f]_{C^\alpha_\cG(\widetilde \Omega_{\tau, R/\langle R\rangle})}
        + \|\widetilde f\|_{L^\infty(\widetilde \Omega_{\tau, R/\langle R \rangle})}
    \right).% \\
%&\leq C\left( [\widetilde f]_{C^\alpha_\cG(\widetilde \Omega_{\tau, R/\langle R \rangle})} + \|\widetilde f\|_{L^\infty(\widetilde \Omega_{\tau, R/\langle R \rangle})},
\end{split}
\]
where we used interpolation (see \Cref{l:interp}) for the terms involving $\nabla_v \widetilde f$.

The regularity norms of $f$ in $t$ and $x$ are easier to handle, since the change of variables $p \mapsto \varphi(p)$ does not affect $t$ or $x$. In particular, we have 
\[
\frac{|f(t_0,x_0,p_0) - f(t,x_0,p_0)|}{|t-t_0|^{\frac{2\beta}3}}
    \leq 2 \frac {|f(t_0,x_0,p_0) - f(t,x_0,p_0)|}{|t-t_0|^\beta + |v_0(t-t_0)|^{\sfrac{2\beta}3}}, 
\]
and therefore, 
\[
[f]_{C^{\frac{2\beta}3}_t(\Omega_{\tau,R})} \leq 2 [\widetilde f]_{C^\beta_{\cG,t}(\widetilde \Omega_{\tau, R/\langle R \rangle})}.
\]
An even simpler argument, using the definition \eqref{e.Ca-tx} of $[\cdot]_{C^\alpha_{\cG,x}}$, shows
\[
[f]_{C^{\frac{2+\alpha}3}_x(\Omega_{\tau,R})} = [\widetilde f]_{C^{\frac{2+\alpha}3}_{\cG,x}(\widetilde \Omega_{\tau, R/\langle R \rangle})}.
\]
 Combining these upper bounds, we have shown
\begin{equation}\label{e.fu}
	[f]_{C^{2+\frac{2\alpha}3}_{\rm kin}(\Omega_{\tau,R})}
    = [D_p^2 f]_{C^{\frac{2\alpha}3}_{\rm kin}(\Omega_{\tau,R})} + [f]_{C^{\frac{2\beta}3}_t(\Omega_{\tau,R})} + [f]_{C^{\frac{2+\alpha}3}_x(\Omega_{\tau,R})}
		\leq \mathcal T.
\end{equation}
Sending $R\to\infty$, and renaming $\alpha$ to $\sfrac{2\alpha}3$, we obtain~\eqref{e.first-schauder}, the first conclusion of the theorem.

The proof of the second conclusion~\eqref{e.second-schauder} follows the same outline, starting with \Cref{p:old-schauder2} applied to $g$ and translating to $\widetilde f$ and $f$ as above. One uses \Cref{l:hP-holder} to control the $C^{1+\alpha}$-norms of the coefficients, along with formulas like 
\[
\partial_{v_i} \widetilde s = \partial_{v_i} \psi_j(v) \partial_{p_j} s = \frac 1 {\sqrt{1-|v|^2}}\left(\delta_{ij} + \frac {v_i v_j}{1-|v|^2}\right) \partial_{p_j} s,
\]
which contribute polynomial factors in $R$. %that are $\approx 1$ since $v$ is close to the origin. One also must keep track of extra powers of $\langle R \rangle$ arising from the chain rule applied to the transformation $(t,x,v) \mapsto (t,Px,Pv)$. 
We omit the details of the proof of~\eqref{e.second-schauder}.
\end{proof}

\subsection{Divergence-form equations}

Next, we analyze how the change of variables $v \mapsto p=\psi(v)$ transform a divergence-form equation
\begin{equation}\label{e.div-linear}
\left(\partial_t  + \frac p \pp \cdot \nabla_x\right) f = \nabla_p \cdot (A \nabla_p f) + B\cdot \nabla_p f + s,
\end{equation}
with $A$, $B$, and $s$ bounded and measurable, and $A$ uniformly elliptic as in~\eqref{e.A-elliptic}. This is not relevant to the Schauder estimates of this section, but it is needed in the first step of our regularity bootstrap in \Cref{s:higher}. The proof of the following lemma is very similar to the proof of \Cref{l:transform}, so we omit it.

\begin{lemma}\label{l:div-transform}
Let $f$ solve~\eqref{e.div-linear} in 
some transformed cylinder 
\[
\widetilde Q_r(t_0,x_0,p_0) = \{(t,x,p)\in \R^7 : (t,x,\varphi(p)) \in Q_r(t_0,x_0,\varphi(p_0))\}.
\]
Then the function $\widetilde f$ defined by
\[
\widetilde f(t,x,v) = f(t,x,\psi(v)),
\]
solves the non-relativistic kinetic Fokker-Planck equation
\begin{equation}
(\partial_t + v\cdot \nabla_x) \widetilde f = \nabla_v \cdot (\widetilde A \nabla_v \widetilde f) + \widetilde B \cdot \nabla_v \widetilde f + \widetilde s,
\end{equation}
in $Q_r(t_0,x_0,\varphi(p_0))$, with
\be\label{e.c020802-div}
\begin{split}
	\widetilde A(t,x,v) &= (1-|v|^2) (\Id - v\otimes v) A(t,x,\psi(v)) (\Id - v\otimes v),\\
	\widetilde B(t,x,v) &= \sqrt{1-|v|^2} (\Id - v\otimes v) B(t,x,\psi(v)) + 3(1-|v|^2) (\Id - v\otimes v) A(t,x,\psi(v))v,\\
	\widetilde s(t,x,v) &= s(t,x,\psi(v)).
\end{split}
\ee
\end{lemma}

Next, we have a H\"older estimate for solutions of~\eqref{e.div-linear}:

\begin{lemma}\label{l:Calpha}
Let $f$ solve~\eqref{e.div-linear} in $[0,T]\times \R^6$ for some $T>0$. 
Assume that 
\[
\pp f, B,\pp^{-3} s\in L^\infty\left([0,T]\times\R^6\right),
\]
and that $A$ satisfies the uniform ellipticity hypothesis~\eqref{e.A-elliptic} for some $\Lambda > \lambda >0$. 

Then, for any $\tau\in (0,T]$,
\[
    \|f\|_{C^\alpha_{\rm kin}([\tau,T]\times \R^6)}
    \leq C \left(1 + \frac{1}{\tau^{\frac\alpha2}}\right)
        \left(
        \|\pp f\|_{L^\infty([0,T]\times \R^6))}
        + \|\pp^{-3} s\|_{L^\infty([0,T]\times \R^6)}\right)
\]
for some $\alpha\in (0,1)$ and $C>0$ depending only on $\lambda$, $\Lambda$, and the $L^\infty$-norm of $B$. 
\end{lemma}

\begin{proof}
As above, we define $\widetilde f$ as in \Cref{l:transform}, on the domain $[0,T]\times\R^3_x\times B_1$. For $R>0$ and $\tau\in (0,T]$, we define $\Omega_{\tau,R}$ and $\widetilde \Omega_{\tau,R/\langle R \rangle}$ as in \eqref{e.Omega-def}. For $(t_0,x_0,v_0) \in \widetilde \Omega_{\tau,R/\langle R \rangle}$, we again let
\begin{equation}
    r = c(1-|v_0|^2)^{\sfrac12}\min\{1, \sqrt{\tau/2}\},
\end{equation} and define $g$ on $Q_r$ as in \eqref{e.g-def}. By direct calculation, $g$ satisfies
\[
    \partial_t g + v\cdot \nabla_x g
        = \nabla_v\cdot ( A_P \nabla_v g) + B_P\cdot \nabla_v g + s_P,
\]
for $(t,x,v) \in Q_{r}$, where 
\[
\begin{split}
A_P(t,x,v) &= P^{-1} \widetilde A(t,Px,Pv) P^{-1},\\
B_P(t,x,v) &= P^{-1} \widetilde B(t,Px,Pv),\\
s_P(t,x,v) &= \widetilde s(t,Px,Pv).
\end{split}
\]
Note that $\widetilde B$ is defined as in \Cref{l:div-transform}. The definitions of $\widetilde A$ and $\widetilde s$ are unchanged from the proof of \Cref{t:first_schauder}. Therefore, \Cref{l:P} gives
\[
\lambda I \leq A_P(t,x,v) \leq \Lambda I
    \qquad \text{ for all } (t,x,v) \in Q_{r},
\]
as well as
\[
\|B_P\|_{L^\infty(Q_{r})} \lesssim \langle R \rangle^{2}\|B\|_{L^\infty(\Omega_{0,R})} + \langle R \rangle\Lambda
\qquad\text{ and }\qquad
\|s_P\|_{L^\infty(Q_{r})} \lesssim \|s\|_{L^\infty(\Omega_{0,R})}.
\]
Here we are using the useful inequalities
\be\label{e.v_0_R}
    |P^{-1}| \leq (1-|v_0|^2)^{-\sfrac32}
    \qquad\text{ and }\qquad
    \frac{1}{(1-|v_0|^2)^{\sfrac12}} \leq \langle R \rangle.
\ee

Next, we would like to apply the $C^\alpha$-estimate of \cite{golse2019}, quoted above as \Cref{p:Calpha}.
Our bound on $B_P$ depends on $R$.  In order to track this dependence through the application of \Cref{p:Calpha}, we let
\[
    \rho = \frac{r}{\RR^2}
    \qquad\text{ and }\qquad
    g_\rho(t,x,v) = g(\rho^2 t, \rho^3 x, \rho v).
\]
By a direct calculation, this solves
\[
\partial_t g_\rho + v\cdot\nabla_x g_\rho = \nabla_v \cdot (A_P(\rho^2 t, \rho^3 x, \rho v) \nabla_v g_\rho) + \rho B_P(\rho^2 t, \rho^3 x, \rho v) \cdot \nabla_v g_\rho + \rho^2 s_P(\rho^2 t, \rho^3 x, \rho v),
\]
in $Q_{r/\rho} = Q_{\langle R \rangle^2}$. Since $Q_1 \subset Q_{\langle R \rangle^2}$, we can apply \Cref{p:Calpha}, which yields
\[
\|g_\rho\|_{C^\alpha_\cG(Q_{\sfrac12})}
    \lesssim \left( \|g_\rho\|_{L^2(Q_1)} +\rho^2 \|s_P\|_{L^\infty(Q_1)}  \right)
    \lesssim \left( \|g_\rho\|_{L^\infty(Q_1)} +\rho^2 \|s_P\|_{L^\infty(Q_1)}  \right),
\]
with the implied constant independent of $v_0$. In terms of $g$, this gives
\[
\begin{split}
\|g\|_{C^\alpha_\cG(Q_{\rho/2})}
    &\lesssim \frac{1}{\rho^\alpha}\left(\|g\|_{L^\infty(Q_{\rho})} + \rho^{2} \|s_P\|_{L^\infty(Q_{\rho})}\right)\\
    &= \frac{\langle R \rangle^{2\alpha}}{r^\alpha} \left(  \|g\|_{L^\infty(Q_{r})} + \frac{1}{\langle R \rangle^4}\|s_P\|_{L^\infty(Q_{r})}\right).
\end{split}
\]
since $Q_\rho \subset Q_r$. As in the proof of \Cref{t:first_schauder}, we have
\[
\begin{split}
\|g\|_{L^\infty(Q_{r})}
    &\leq \|\widetilde f\|_{L^\infty(Q_r(z_0))}
    \leq \|f\|_{L^\infty(\Omega_{0,R})},\\
\|s_P\|_{L^\infty(Q_{r})}
    &\leq \|\widetilde s\|_{L^\infty(Q_r(z_0))}
    \leq \|s\|_{L^\infty(\Omega_{0,R})}.
    \end{split}
\]
Letting $\rho' = (1-|v_0|^2)^{\sfrac32} \rho$ and using \Cref{l:hP-holder} yields
\begin{equation}\label{e.int-holder}
\begin{split}
\|\widetilde f\|_{C^\alpha_\cG(Q_{\frac{\rho'}2}(z_0))}
&\leq
    \frac{1}{(1-|v_0|^2)^{\frac{3\alpha}2}} \|g\|_{C^\alpha_\cG(Q_{\rho/2})}\\
&\lesssim 
    \frac{1}{r^{\alpha}} \langle R \rangle^{5\alpha} \left( \|f\|_{L^\infty(\Omega_{0,R})} + \frac{1}{\langle R \rangle^4} \|s\|_{L^\infty(\Omega_{0,R})}\right).%,\\
%   &\leq C \left(1 + \tau^{-\alpha/2}\right) \langle R \rangle^{5\alpha} \left( \|g\|_{L^\infty(\Omega_{0,R})} + \langle R \rangle^{-4} \|s\|_{L^\infty(\Omega_{0,R})}\right),
\end{split}
\end{equation}

Now, we follow the proof of \Cref{t:first_schauder}. For any $(t,x,v), (t_0,x_0,v_0) \in \widetilde \Omega_{\tau, R/\langle R \rangle}$ with $t\leq t_0$, we define %$r$, $\rho = r\langle R \rangle^{-1}$, and 
$\rho' = (1-|v_0|^2)^{3/2} \rho$ as above. If $d_\cG((t,x,v),(t_0,x_0,v_0)) \geq \sfrac{\rho'}4$, then
\[
    \frac{|\widetilde f(t,x,v) - \widetilde f(t_0,x_0,v_0)|}{d_\cG((t,x,v),(t_0,x_0,v_0))^\alpha}
    \lesssim \frac{\|\widetilde f\|_{L^\infty(\widetilde \Omega_{\tau,R/\RR})}}{(\rho')^\alpha} 
    \lesssim \frac{\RR^{5\alpha}}{r^\alpha} \|\widetilde f\|_{L^\infty(\widetilde \Omega_{\tau,R/\RR})}.
\]
If $d_\cG((t,x,v),(t_0,x_0,v_0)) < \sfrac{\rho'}4$, then $(t,x,v) \in Q_{\rho'/2}((t_0,x_0,v_0))$, and we can apply \eqref{e.int-holder}. In all, we have
\[
\begin{split}
\|\widetilde f\|_{C^\alpha_\cG(\widetilde \Omega_{\tau,R/\RR})}
    &\lesssim \frac{\RR^{5\alpha}}{r^\alpha}
    \left(\|f\|_{L^\infty(\Omega_{0,R})} + \frac{1}{\langle R \rangle^4} \|s\|_{L^\infty(\Omega_{0,R})}\right)
\\
&\lesssim
    \left(1 + \frac{1}{\tau^\frac{\alpha}{2}}\right) \RR^{6\alpha}\left(\| f\|_{L^\infty(\Omega_{0,R})} + \frac{1}{\langle R \rangle^4} \|s\|_{L^\infty(\Omega_{0,R})}\right),
\end{split}
\]
using the definition of $r$ and~\eqref{e.v_0_R}. By \Cref{l:holder-compare} and \Cref{l:holder-product}, this implies
\[
\|f\|_{C^{2\alpha/3}_{\rm kin}(\Omega_{\tau,R})} \lesssim \left( 1 + \frac{1}{\tau^{\frac\alpha2}}\right) \left(\|\pp^{6\alpha} f\|_{L^\infty(\Omega_{0,R})} + \|\pp^{6\alpha - 4}s\|_{L^\infty(\Omega_{0,R})}\right).
\]
We may assume $\pp^{6\alpha} \leq \pp$ by taking $\alpha$ smaller if necessary. The proof is now complete after sending $R\to \infty$ and renaming $\sfrac{2\alpha}3$ as $\alpha$.
\end{proof}

\section{Regularity estimates for the coefficients}\label{s:coeffs}

\subsection{H\"older bounds}

In this subsection, we show that the coefficients $a^g$, $b^g$, and $c^g$ are H\"older whenever $g$ is H\"older. When we apply these estimates, $g$ will be the solution $f$ or one of its derivatives. Because of the nonlocality of the coefficients, we need to measure the H\"older continuity of $g$ in the entire momentum domain, but we also want to be able to localize our estimates in $t$ and $x$. Therefore, we work on $\Omega\times\R^3_p$, where $\Omega$ is a general $(t,x)$ domain.

Let us also record two useful facts that we use repeatedly:
\begin{align}
	\big|p\qq - q\pp \big| &= |(p-q) \qq + q(\qq - \pp)| \lesssim\qq |p-q|,\label{e.pqq}\\
	\pp \qq - p \cdot q &= (\pp - \qq)\qq + q\cdot (q-p) + 1  \lesssim \qq \big(1 + |p-q|\big),\label{e.ppqq}
\end{align}
where we used the derivative estimate $|\nabla_z \langle z \rangle| = |z|\langle z \rangle^{-1} \leq 1$ in the second line.

Since the three coefficients $a^g$, $b^g$, and $c^g$ feature integral kernels with the same order of singularity, the key calculation of these H\"older estimates is the same for all three coefficients. This calculation is encoded in the following lemma.
\begin{lemma}\label{l:integral-holder}
Let $K(p,q):\R^3\times\R^3\to \R$ be an integral kernel satisfying, for all $p,q\in \R^3$,
\be\label{e.K_assumption}
K(p,q) \leq C\qq^{k_1}\left(1+\frac 1 {|p-q|}\right)
\ee
and
\be\label{e.grad_K_assumption}
|\nabla_p K(p,q)| \leq C\qq^{k_2}\left(1 + \frac 1 {|p-q|^2}\right),
\ee
for some constant $C$. For some $\Omega\subseteq \R^4$, assume that $g:\Omega\times\R^3\to \R$ satisfies
\[
\|\pp^{k_0} g\|_{C^\alpha_{\rm kin}(\Omega\times\R^3)} < \infty,
\]
 $\alpha\in (0,1)$, and $k_0>\max\{k_1+2, k_2 +1\}$. Then the function
\[
u(t,x,p) := \int_{\R^3} K(p,q) g(t,x,q) \dd q,
\]
satisfies
\[
\|u\|_{C^\alpha_{\rm kin}(\Omega\times \R^3)} \leq C \|\pp^{k_0}g\|_{C^\alpha_{\rm kin}(\Omega\times \R^3)}.
\]
\end{lemma}

\begin{proof}
Let $z_0 = (t_0,x_0,p_0)$ and $z = (t,x,p)$ be two points of $\Omega\times\R^3$. First, we consider the case when $d_{\rm kin}(z,z_0) < 1$. In this case, we write
\[
\begin{split}
u(z) & - u(z_0)  = u(t,x,p) - u(t_0, x_0, p_0) \\
& =  \int_{\R^3} \Big(K(p,q) g(t,x,q) - K(p_0, q) g(t_0, x_0, q) \Big) \dd q \\
& =  \int_{\R^3} \Big(K(p,q) - K(p_0,q) \Big)g(t,x,q) \dd q +  \int_{\R^3} K(p_0,q) \Big( g(t,x,q) - g(t_0, x_0, q)\Big) \dd q \\
& = \int_{\R^3}  (p-p_0) \cdot \int_0^1  \nabla_p K(\xi p + (1-\xi)p_0, q)  g(t,x,q) \dd\xi\dd q \\
	& \qquad +  \int_{\R^3} K(p_0,q) \Big( g(t,x,q) - g(t_0, x_0, q)\Big) \dd q.
\end{split}
\]
Fubini's theorem and our assumptions on $K$ and $\nabla_p K$ yield
\begin{align}
 |u(z) - u(z_0)| 
& \lesssim  |p-p_0| \int_0^1 \left( \int_{\R^3}  \qq^{k_2} \left( 1  + \frac{1}{|\xi p + (1-\xi) p_0 -q|^2}\right) g(t,x,q) \dd q\right) \dd \xi\\
& \quad +  \int_{\R^3}  \qq^{k_1-k_0} \left(1+ \frac{1}{|p_0-q|} \right) \Big(\qq^{k_0} g(t,x,q) - \qq^{k_0} g(t_0, x_0, q)\Big) \dd q\\
& =: I_1 + I_2.
\end{align}
To estimate $I_1$, we use the uniform upper bound $g(t,x,q)\lesssim \|g\|_{L^\infty_{k_0}(\Omega\times \R^3)} \qq^{-k_0}$ and the integral estimate: for all $a\in \R^3$,
\begin{equation}\label{e.integral-est-a}
    \int_{\R^3} \frac{\qq^{k_2-k_0}}{|a- q|^2} \dd q \lesssim 1,
\end{equation}
which follows since, by assumption, $k_0>k_2 +1$. This yields
\[
I_1 \lesssim |p-p_0| \|g\|_{L^\infty_{k_0}(\Omega\times \R^3)} \lesssim d_{\rm kin}(z,z_0)^\alpha \|g\|_{L^\infty_{k_0}(\Omega\times \R^3)},
\]
since $d_{\rm kin}(z,z_0) < 1$. 

To estimate $I_2$, we use the fact that $\pp^{k_0} g \in C^\alpha_{\rm kin}(\Omega\times \R^3)$ to write
\begin{align}
I_2 
&\lesssim [\pp^{k_0} g]_{C^\alpha_{\rm kin}(\Omega\times \R^3)}  \int_{\R^3}  \qq^{k_1-k_0} \left(1+ \frac{1}{|p_0-q|} \right) d^\alpha_{k}((t,x,q), (t_0, x_0, q)) \dd q.
\end{align}
Note that
\begin{equation}\label{e.dLq}
\begin{split}
    d_{\rm kin}((t,x,q),(t_0,x_0,q))
    &= |t-t_0|^{\sfrac{1}{2}}
        + |x-x_0|^{\sfrac{1}{3}} \leq d_{\rm kin}(z,z_0).
\end{split}
\end{equation}
 Using this in our estimate for $I_2$, we obtain
\[
\begin{split}
I_2 &\lesssim [\pp^{k_0} g]_{C^\alpha_{\rm kin}(\Omega\times \R^3)} d_{\rm kin}(z,z_0)^{\alpha}  \int_{\R^3} \qq^{k_1-k_0}\left(1 + \frac 1 {|p_0 - q|}\right)  \dd q\\
&\lesssim  [\pp^{k_0} g]_{C^\alpha_{\rm kin}(\Omega\times \R^3)} d_{\rm kin}(z,z_0)^{\alpha},
\end{split}
\]
using $k_0> k_1 + 2$ and an estimate similar to \eqref{e.integral-est-a}. Combining our estimates for $I_1$ and $I_2$ completes the proof of the first case.

In the second case, when $d_{\rm kin}(z,z_0)\geq  1$, we have 
$|u(z) - u(z_0)| \leq 2 \|u\|_{L^\infty(\Omega\times\R^3)}d_{\rm kin}(z,z_0)^\alpha$, and for any $z\in \Omega\times\R^3$, 
\[
u(z) =  \int_{\R^3} K(p,q) g(t,x,q) \dd q \leq C \|g\|_{L^\infty_{k_0}(\Omega\times\R^3)} \int_{\R^3} \qq^{k_1} \left(1 + \frac 1 {|p-q|}\right) \qq^{-k_0} \dd q \leq C \|g\|_{L^\infty_{k_0}(\Omega\times\R^3)},
\] 
using $k_0 > k_1 + 2$ as above.  Taking the supremum over $z$, the proof is complete.
\end{proof}

Next, we estimate the gradient of the integral kernel appearing in $c^g$, to guarantee \Cref{l:integral-holder} is applicable. This estimate for $\nabla_p G(p,q)$ will also be used in the estimates for $a^g$ and $b^g$.

\begin{lemma}\label{l:grad-G}
Let $G(p,q)$ be defined as in \Cref{l:G-upper}. 
Then for any $p,q\in \R^3$, the gradient estimate
\[
|\nabla_p G(p,q)| \leq C \qq^{\sfrac{7}{2}}\left(1  +  \frac 1 {|p-q|^2}\right),
\]
holds, where the constant $C>0$ is independent of $p$ and $q$.
\end{lemma}

\begin{proof}
Recalling that $\tau(\tau-2) = (\pp\qq - p\cdot q)^2 - 1$ and differentiating $G(p,q)$ in $p$, we obtain after a direct calculation
\[
\begin{split}
\nabla_p G(p,q) & = \frac{1}{\qq} \nabla_p \left( \frac{\qq - (p\cdot q) \pp^{-1}}{\sqrt{\left( \pp \qq - p \cdot q\right)^2 - 1}}\right)\\
	& =  \frac{1}{\pp^3 \qq} \frac{(p\cdot q) p - q \pp^2 }{ \left( \tau(\tau-2)\right)^{\sfrac{1}{2}}}
		-  \frac{1}{\pp^2\qq} \frac{\Big(\pp \qq - p \cdot q\Big)^2 \Big(p \qq - q\pp\Big)}{ \left(\tau(\tau-2)\right)^{\sfrac{3}{2}}}\\
	& =: g_1 - g_2. 
\end{split}
\]
From \Cref{l:GS} and $\tau \geq 1$, we have
\begin{equation}\label{e.tau-2}
\tau(\tau-2) \geq \frac{|p-q|^2}{2\pp\qq}. 
\end{equation}
This implies
\[
\begin{split}
 |g_1| \lesssim \frac{1}{\pp^{\sfrac{5}{2}} \qq^{\sfrac{1}{2}}} \frac{\left|(p\cdot q) p - q \pp^2\right|}{|p-q|}
 	 \lesssim  \frac{ \pp^2 \qq}{\pp^{\sfrac{5}{2}} \qq^{\sfrac{1}{2}}}  \frac{1}{|p-q|} \lesssim \frac{\qq^{\sfrac{1}{2}}}{|p-q|},
 \end{split}
\]
and
\[
\begin{split}
|g_2| &\lesssim \pp^{-\sfrac{1}{2}} \qq^{\sfrac{1}{2}} \frac{\Big(\pp \qq - p \cdot q\Big)^2 \Big|p \qq - q\pp\Big| }{|p-q|^3}.
\end{split}
\]
To complete the estimate of $g_2$, we need to use~\eqref{e.pqq} and~\eqref{e.ppqq}, which yield
\[
\begin{split}
|g_2|	& \lesssim  \qq^{\sfrac{1}{2}}  \frac{|p-q| \qq^3+ |p-q|^3 \qq^3}{|p-q|^3}\\
	& \lesssim  \qq^{\sfrac{7}{2}} \left(1 + \frac{1}{|p-q|^2} \right).
\end{split}
\]
Combining the upper bounds for $|g_1|$ and $|g_2|$, and using $|p-q|^{-1} \lesssim 1 + |p-q|^{-2}$, we conclude the proof.
\end{proof}

The next three lemmas establish the desired H\"older estimates for $c^g$, $b^g$, and $a^g$.

\begin{lemma}\label{l:c-holder}
With $\Omega\subseteq \R^4$ arbitrary, assume that 
\[
\pp^k g \in C^\alpha_{\rm kin}(\Omega\times \R^3),
\]
for some $\alpha \in (0,1)$ and $k>\sfrac{9}{2}$. Then there holds
\[
\|c^g(z)\|_{C^\alpha_{\rm kin}(\Omega\times \R^3)} \leq C\|\pp^k g\|_{C^\alpha_{\rm kin}(\Omega\times \R^3)},
\]
where $c^g$ is defined in~\eqref{e.c}, and $C$ is a constant independent of $z_0$ and $g$.
\end{lemma}

\begin{proof}
We begin by estimating the H\"older norm of the function $\kappa(p)$. From the proof of \Cref{l:c-upper}, we have $\kappa(p) \lesssim \pp^{-1}$, which is equivalent to 
\[
\int_0^\pi (1+|p|^2\sin^2\theta)^{-\sfrac{3}{2}} \sin\theta\dd \theta \lesssim \pp^{-2}.
\]
Using this, we have
\[
\begin{split}
\kappa(p) - \kappa(p_0) & = 2^{\sfrac{7}{2}} \pi \Big( \pp - \ppo \Big) \int_0^\pi (1 + |p|^2 \sin^2 \theta)^{-\sfrac{3}{2}} \sin \theta \dd\theta\\
	& \qquad \qquad + 2^{\sfrac{7}{2}} \pi  \ppo \int_0^\pi \Big((1 + |p|^2 \sin^2 \theta)^{-\sfrac{3}{2}} - (1 + |p_0|^2 \sin^2 \theta)^{-\sfrac{3}{2}}  \Big) \sin\theta \dd \theta\\
& \lesssim |p -p_0| \pp^{-2} + \ppo \int_0^\pi \Big((1 + |p|^2 \sin^2 \theta)^{-\sfrac{3}{2}} - (1 + |p_0|^2 \sin^2 \theta)^{-\sfrac{3}{2}}  \Big) \sin\theta \dd \theta.
\end{split}
\]
For $g(s) = (1+ s \sin^2\theta)^{-\sfrac{3}{2}}$, we have $g'(s) = -(\sfrac{3}{2})(1+ s \sin^2\theta)^{-\sfrac{5}{2}} \sin^2 \theta$ with $\sup |g'(s)| \lesssim 1$, where the supremum is taken for values of $s$ between $|p|^2$ and $|p_0|^2$. Therefore,
\begin{equation}\label{difference of kappa estimate}
\kappa(p) - \kappa(p_0) 
	\lesssim  |p -p_0| \pp^{-2} + \ppo |p-p_0| \lesssim \ppo |p-p_0|.
\end{equation}
Together with $\kappa(p) \lesssim \pp^{-1}$, this implies
\begin{equation} \label{kappa f estimate}
\begin{split}
\kappa(p_0) g(z_0)  - \kappa(p) g(z)
	  &= \Big(\kappa(p_0) - \kappa(p) \Big) g(z_0) + \kappa(p) \Big( g(z_0) - g(z)\Big) \\
	 & \lesssim  \ppo |p-p_0| g(z_0) + \pp^{-1} \Big( g(z_0) - g(z)\Big).
\end{split}
\end{equation}
Applying this for points with $d_{\rm kin}(z,z_0) < 1$, and the obvious inequality $|\kappa(p_0) g(z_0) - \kappa(p) g(z)| \leq 2\|\kappa\|_{L^\infty}\|g\|_{L^\infty}$ when $d_{\rm kin}(z,z_0) \geq 1$,  we have shown
\[
\|\kappa(p) g\|_{C^\alpha_{\rm kin}(\Omega\times\R^3)} \lesssim \|\pp g\|_{C^\alpha_{\rm kin}(\Omega\times\R^3)}.
\]

Next, we address the term
\[
    c_1(z)
    = 4\int_{\R^3} G(p,q) g(t,x,q) \dd q,
\]
where $G(p,q)$ is defined as in \Cref{l:G-upper}. From \Cref{l:G-upper} and \Cref{l:grad-G}, we see that $G(p,q)$ satisfies the hypotheses of \Cref{l:integral-holder} with $k_1 = \sfrac12$ and $k_2 = \sfrac72$. Therefore, we have
\[
\|c_1\|_{C^\alpha_{\rm kin}(\Omega\times\R^3)} \lesssim  \|\pp^k g\|_{C^\alpha_{\rm kin}(\Omega\times \R^3)},
\]
with $k=k_0 > \sfrac92$. Combining this with our estimate of $\kappa(p) g$ above, the proof is complete.
\end{proof}

\begin{lemma}\label{l:b-holder}
For any $\Omega\subseteq \R^4$, if
\[
\pp^k g \in C^\alpha_{\rm kin}(\Omega\times\R^3),
\]
for some $\alpha \in (0,1)$ and $k>\sfrac{13}{2}$, then there holds
\[
\|b^g(z)\|_{C^\alpha_{\rm kin}(\Omega\times\R^3)} \leq C \|\pp^k g\|_{C^\alpha_{\rm kin}(\Omega\times\R^3)},
\]
where $b^g$ is defined in~\eqref{e.b}, and $C$ is a constant independent of $g$.
\end{lemma}

\begin{proof}
For $j\in \{1,2,3\}$, let us denote by $H_j$ the integral kernel appearing in the definition of $b^g_j$:
\begin{equation}\label{e.H-def}
H_j(p,q) = \Lambda(p,q) (\tau-2) (p_j +q_j) 
	= G(p,q) \frac{\tau-1}{\tau} (p_j +q_j).
\end{equation}
Our goal is to apply \Cref{l:integral-holder} to
\[
    K_j(p,q) = \frac{1}{\pp} H_j(p,q)
\]
with $k_1 = \sfrac32$ and $k_2 = \sfrac92$ in order to conclude that
\[
    u_j(z) := \frac{1}{\pp} b_j^g(z)
\]
satisfies
\[
    \|u_j(z)\|_{C^\alpha_{\rm kin}(\Omega\times\R^3)}
    \lesssim  \|\pp^k g\|_{C^\alpha_{\rm kin}(\Omega\times\R^3)},
\]
whenever $k> \sfrac{11}{2}$. A straightforward localization argument using  $b_j^g(z) = \pp u_j(z)$ and \Cref{l:holder-product} then implies
\[
\|b_j^g\|_{C^\alpha_{\rm kin}} \lesssim \|\pp^{k+1} g\|_{C^\alpha_{\rm kin}(\Omega\times\R^3)}.
\]

To complete the proof, it thus suffices to check the assumptions of \Cref{l:integral-holder}.  We begin by examining $H_j$.  
From \Cref{l:G-upper} and $\sfrac{(\tau-1)}{\tau}\leq 1$, we have the following upper bound (more crude than the bound used in the proof of \Cref{l:b-upper}):
\begin{equation}\label{e.H_j}
    |H_j(p,q)| \le \pp  \qq |G(p,q)| \lesssim \frac{ \pp \qq }{|p-q|}.
\end{equation}
We deduce that
\[
    |K_j(p,q)|
        \lesssim \frac{\qq}{|p-q|}.
\]

Next we consider the gradient:
\[
    \nabla_p K_j(p,q)
        = \frac{p}{\pp^3} H_j(p,q)
            + \pp \nabla_p H_j(p,q).
\]
The first term clearly satisfies~\eqref{e.K_assumption} with $k_2=\sfrac{9}{2}$ due to~\eqref{e.H_j}.  Hence, we need only understand $\nabla_p H_j$.  We have
\[
\begin{split}
\nabla_p H_j(p,q) &=  \frac{\tau-1}{\tau} (p_j +q_j) \nabla_p G(p,q) + G(p,q) \,(p_j + q_j)\, \tau^{-2}\, \nabla_p \tau + G(p,q) \,\frac{\tau-1}{\tau} e_j\\
& =: h_1 + h_2 + h_3.
\end{split}
\]
From \Cref{l:grad-G} and $\sfrac{(\tau-1)}{\tau} \leq 1$, we have
\[
\begin{split}
|h_1|\lesssim   \pp \qq^{\sfrac{9}{2}} \left( 1 +  \frac{1}{|p-q|^2} \right).
\end{split}
\]
To estimate $h_2$,  note that
\[
    \nabla_p \tau
    = \nabla_p\left(\pp\qq - p\cdot q +1\right)
    = \qq \frac{p}{\pp} - q
    = \frac{1}{\pp} \big(\qq p - \pp q\big).
\]
Then, by~\eqref{e.pqq}, we have $|\nabla_p \tau| \lesssim |p-q|$. Since $\tau^{-1}\leq 1$, we therefore have
\[
\begin{split}
|h_2| \lesssim\pp\qq |p-q| G(p,q) \lesssim \pp \qq,
\end{split}
\]
after using \Cref{l:G-upper} to bound $G(p,q)$. 
Finally, 
\[
|h_3| \lesssim |G(p,q)| \lesssim  \frac{1}{|p-q|}.
\]
Combining the estimates of $h_1, h_2$, and $h_3$ yields
\[
|\nabla_p H_j(p,q)|
	\lesssim  \pp \qq^{\sfrac{9}{2}} \left(1+ \frac{1}{|p-q|^2}  \right),
\]
as claimed. We deduce that
\be
    |\nabla_p K_j(p,q)|
        \lesssim \frac{1}{\pp^2}\frac{1}{|p-q|}
            + \qq^{\sfrac{9}{2}} \left(1+ \frac{1}{|p-q|^2}  \right)
        \lesssim \qq^{\sfrac{9}{2}} \left(1+ \frac{1}{|p-q|^2}  \right).
\ee
Hence, the assumptions of \Cref{l:integral-holder} are met, and the proof is complete.
\end{proof}

\begin{lemma}\label{l:a-holder}
For any $\Omega\subseteq \R^4$, if
\[
\pp^k g \in C^\alpha_{\rm kin}(\Omega \times\R^3),
\]
for some $\alpha \in (0,1)$ and $k>\sfrac{19}{2}$, then there holds, for each $i,j\in \{1,2,3\}$,
\[
\|a_{ij}^g(z)\|_{C^\alpha_{\rm kin}(\Omega\times\R^3)} \leq C \|\pp^k g\|_{C^\alpha_{\rm kin}(\Omega\times\R^3)},
\]
where $a_{ij}^g$ is defined in~\eqref{e.a} and $C$ is a constant independent of  $g$.
\end{lemma}

\begin{proof}
Fix $i,j$, and define
\[
K_{ij}(p,q) = \pp^{-4} \Phi^{ij}(p,q).
\]
Our goal is to apply \Cref{l:integral-holder} with $k_1 = 1$ and $k_2 = \sfrac92$ to deduce that
\[
    \|u_{ij}\|_{C^{\alpha}_{k}(\Omega\times\R^3)}
        \lesssim  \|\pp^k g\|_{C_k^\alpha(\Omega\times\R^3)}
    \quad\text{ where } u_{ij} = \int K_{ij}(p,q) g(t,x,q) \dd q.
\]
whenever $k>\sfrac {11}2$. 
As in the proof of \Cref{l:b-holder}, the conclusion follows after a localization argument using \Cref{l:holder-product} and $a^g_{ij} = \pp^4 u_{ij}$.  Hence, the remainder of the proof is checking that $K_{ij}$ satisfies the hypotheses of \Cref{l:integral-holder}.  For simplicity, we drop the $ij$ sub-script for $K_{ij}$ in the sequel.

From \Cref{l:Phi-upper}, we have the following upper bound for all $p,q\in \R^3$:
\be\label{e.K_bound}
    K(p,q)
        = \frac{1}{\pp^{4}}|\Phi^{ij}(p,q)|
        \lesssim \frac{\qq}{\pp^3}\left(1 + \frac 1 {|p-q|}\right).
\ee
Next, we analyze the $p$-derivatives of $K$:
\be\label{e.nabla_p_K}
    |\nabla_p K|
    \lesssim \pp^{-5} \Phi^{ij}
        + \pp^{-4} |\nabla_p\Phi^{ij}|.
\ee
We aim to bound the right hand side above by $\qq^{\sfrac92}(1 + |p-q|^{-2}).$  This is simple for the first term in~\eqref{e.nabla_p_K} due to~\eqref{e.K_bound}.  The second term is more difficult.  
Let us note that our upper bounds in the proof of \Cref{l:other-B-upper} are not easily applicable here, because those bounds involved a sum over first derivatives, while we need to bound each individual first derivative. 

Instead, we take advantage of our earlier estimates for $G(p,q)$, by using the definition~\eqref{e.Phi-def} of $\Phi^{ij}$ to write:
\[
\begin{split}
\Phi^{ij}(p,q) &= G(p,q) \Psi^{ij}(p,q),\\
\Psi^{ij}(p,q) &= (\tau - 1) \delta_{ij} + \frac{\tau-1}{\tau(\tau-2)} (p_i - q_i)(p_j - q_j) + \frac{\tau - 1}\tau (p_i q_j + p_j q_i).
\end{split}
\]
By direct calculation, for any $k\in \{1,2,3\}$,
\[
\begin{split}
\partial_{p_k} \Psi^{ij}(p,q) &= (\partial_{p_k} \tau) \delta_{ij} + \partial_{p_k}\left[ \frac{\tau - 1}{\tau(\tau-2)}\right] (p_i - q_i)(p_j - q_j) + \frac{\tau - 1}{\tau(\tau-2)} [\delta_{ik}(p_j - q_j) + \delta_{kj}(p_i - q_i)]\\
&\quad  + \frac{\partial_{p_k} \tau}{\tau^2} (p_i q_j + p_j q_i) + \frac{\tau - 1}\tau (\delta_{ik}q_j + \delta_{kj}q_i)\\
&= : \ell_1 + \ell_2 + \ell_3 + \ell_4 + \ell_5.
\end{split}
\]
Recalling $\tau = \pp\qq - p\cdot q + 1$ and $\partial_{p_k}\tau = \pp^{-1}(\qq p_k - \pp q_k)$, we have
\[
|\ell_1| \lesssim |\partial_{p_k} \tau| \lesssim  \qq.
\]
For the term $\ell_2$, a straightforward calculation shows
\[
\partial_{p_k}\left[ \frac{\tau - 1}{\tau(\tau-2)}\right] = -\partial_{p_k}\tau \left(\frac{\tau^2 - 2\tau + 2}{\tau^2(\tau-2)^2}\right).
\]
Let us also recall from~\eqref{e.pqq} that $|\partial_{p_k} \tau|\lesssim |p-q|$. These two facts imply
\[
|\ell_2| \lesssim |\partial_{p_k}\tau| \frac{\pp^2\qq^2}{\tau^2 (\tau-2)^2} |p-q|^2 \lesssim \frac{\pp^4\qq^4}{|p-q|},
\]
where we used $(\tau-2) \gtrsim \pp^{-1}\qq^{-1}|p-q|^2$ from \Cref{l:GS}, as well as $\tau\geq 1$.  Next, using the same lower bound for $(\tau-2)$, we have
\[
|\ell_3| \lesssim \frac{\pp\qq |\tau-1|}{|p-q|^2} |\delta_{ik}(p_j - q_j) + \delta_{kj}(p_i - q_i)| \lesssim \frac{\pp^2\qq^2}{|p-q|}.
\]
For $\ell_4$, using $|\partial_{p_k}\tau| \lesssim \qq$ and $\tau \geq 1$, we simply have
\[
|\ell_4| \lesssim \pp\qq^2.
\]
Finally, since $(\tau-1)/\tau \leq 1$, we have
\[
|\ell_5| \lesssim \qq.
\]
Summing up, we have shown
\[
|\nabla_p \Psi^{ij}(p,q)| \lesssim \pp^4\qq^4\left(1 + \frac 1 {|p-q|}\right).
\]
We also note, using again the fact that $(\tau - 2) \gtrsim |p-q|^2 \pp^{-1} \qq^{-1}$, that 
\[
|\Psi^{ij}(p,q)| \lesssim \pp\qq.
\]
With these bounds for $\Psi^{ij}(p,q)$ as well as \Cref{l:G-upper} and \Cref{l:grad-G}, we have
\be\label{e.nabla_p_Phi}
|\nabla_p \Phi^{ij}(p,q)| \leq G(p,q) |\nabla_p\Psi^{ij}(p,q)| + |\nabla_p G(p,q) \Psi^{ij}(p,q)| \lesssim \pp^4 \qq^{\sfrac{9}{2}}\left(1 + \frac 1 {|p-q|^2}\right),
\ee
keeping only the largest exponents of $\pp$ and $\qq$.

Combining~\eqref{e.nabla_p_Phi} with~\eqref{e.nabla_p_K} and~\eqref{e.K_bound}, we find
\[
    |\nabla_p K| 
    \lesssim \qq^{\sfrac{9}{2}}
        \left(1 + |p-q|^{-2}\right).
\]
Hence, \Cref{l:integral-holder} applies, and the proof is complete.
\end{proof}

\subsection{Higher derivative estimates}

As mentioned in the introduction, when differentiating the relativistic Landau equation to apply regularity estimates, $p$-derivatives falling on the integral kernels (such as $\Phi^{ij}(p,q)$) cannot be  easily transferred to $f$ because the integrals are not true convolutions. In particular, $\partial_{p_k} a^f \neq a^{\partial_{p_k} f}$, unlike in the classical Landau case. This is a problem because higher $p$ derivatives of $\Phi^{ij}$ feature higher-order singularities at $p=q$, which eventually will not be integrable. 

To get around this, we use an integration by parts technique developed by Strain and Guo~\cite{strain2004stability} to partially transfer a higher $p$ derivative from $\Phi^{ij}$ to $f$, with all resulting terms featuring only a first order singularity.

Following~\cite{strain2004stability}, we define, for any three-dimensional multi-index $\beta = (\beta^1, \beta^2, \beta^3)$, the differential operator $\Theta_\beta$ according to
\begin{equation}
\label{eq:Theta}
\Theta_\beta := \left( \partial_{p_3} + \frac{\qq}{\pp} \partial_{q_3} \right)^{\beta^3} \left( \partial_{p_2} + \frac{\qq}{\pp} \partial_{q_2} \right)^{\beta^2}
\left( \partial_{p_1} + \frac{\qq}{\pp} \partial_{q_1} \right)^{\beta^1} .
\end{equation}
This operator is well-adapted to the relativistic structure of our equation, in the sense that the Lorentzian inner product $\tau-1 = \pp\qq - p\cdot q$ is in the null space of $\Theta_\beta$
\begin{equation}\label{e.theta-pq}
\Theta_\beta (\pp \qq - p \cdot q ) = 0
\end{equation}
 for any nontrivial multi-index $\beta$.

The following result is similar to \cite[Theorem 3]{strain2004stability}, but without an exponentially decaying weight in $q$:

\begin{proposition}\label{p:strain-guo}
Given $\beta$ a multi-index with $|\beta| > 0$, a function $\Gamma:\R^6\to \R$ that is sufficiently smooth away from $p=q$, and a sufficiently smooth, rapidly decaying function $\mu:\R^6\to \R$, 
we have
\begin{equation}\label{eq:IbP}
\partial_p^\beta \int_{\R^3} \Gamma(p,q) \mu(p,q) \dd q = \sum_{\beta_1+\beta_2+\beta_3 \leq \beta} \int_{\R^3} \Theta_{\beta_1} \Gamma(p,q) \partial_q^{\beta_2} \partial_p^{\beta_3} \mu(p,q) \phi^\beta_{\beta_1, \beta_2, \beta_3}(p,q) \dd q,
\end{equation}
where $\phi^\beta_{\beta_1, \beta_2, \beta_3}(p,q)$ is a smooth function which satisfies, for any multi-indices $\nu_1$ and $\nu_2$,
\begin{equation}\label{eq:phi_bounds}
|\partial_q^{\nu_1} \partial_p^{\nu_2} \phi^{\beta}_{\beta_1,\beta_2,\beta_3} (p,q) | \leq C \qq^{|\beta_2|-|\nu_1|} \pp^{|\beta_1|+|\beta_3|-|\beta|-|\nu_2|}.
\end{equation}
The constant $C$ depends on $\beta$, $\nu_1$, and $\nu_2$.

In particular, if $\mu(p,q)$ is a function of $q$ only, then we have the simpler formula
\begin{equation}
\partial_p^\beta \int_{\R^3} \Gamma(p,q) \mu(q) \dd q = \sum_{\beta_1+\beta_2\leq \beta} \int_{\R^3} \Theta_{\beta_1} \Gamma(p,q) \partial_q^{\beta_2}  \mu(q) \phi^\beta_{\beta_1, \beta_2}(p,q) \dd q,
\end{equation}
where $\phi^\beta_{\beta_1,\beta_2}(p,q)$ is a smooth function satisfying 
\begin{equation}\label{eq:phi_bounds2}
|\partial_q^{\nu_1} \partial_p^{\nu_2} \phi^{\beta}_{\beta_1,\beta_2} (p,q) | \leq C \qq^{|\beta_2|-|\nu_1|} \pp^{|\beta_1|-|\beta|-|\nu_2|},
\end{equation}
for all multi-indices $\nu_1$ and $\nu_2$. 
\end{proposition}

By ``sufficiently smooth,'' we mean that all derivatives on the right-hand side are well-defined, except possibly at $p=q$ in the case of derivatives of $\Gamma$. 

In the current article, we only need the case where $\mu(p,q)$ is independent of $p$, but we include the general version for the sake of potential future applications.

\begin{proof}
First, let us note the useful fact that for any multi-index $\nu$, $|\partial^{\nu} \pp^k| \lesssim \pp^{k - |\nu|}$. This relation is implicitly used many
times in the current proof, in conjunction with the Leibnitz rule, when we estimate the $\phi$ coefficients.

The conclusion is shown by induction on the size of $\beta$. Assume first that $\beta = e_i$. Then we have
\[
\partial_{p_i} = -\frac{\qq}{\pp} \partial_{q_i} + \Theta_{e_i} ,
\]
and so
\[
\begin{split}
\partial_{p_i} \int_{\R^3} \Gamma \mu \dd q &=  \int_{\R^3} \Gamma \partial_{p_i} \mu \dd q +  \int_{\R^3}\mu\left( - \frac{\qq}{\pp}\partial_{q_i} + \Theta_{e_i}\right) \Gamma \dd q\\
&= \int_{\R^3} \Gamma \partial_{p_i} \mu \dd q + \int_{\R^3} \Gamma \frac{\qq}{\pp} \partial_{q_i} \mu \dd q + \int_{\R^3} \Gamma \frac{q_i}{\qq \pp} \mu \dd q
+ \int_{\R^3} \Theta_{e_i} \Gamma \mu \dd q .
\end{split}
\]
This can be put into the form~\eqref{eq:IbP} with
\[
    \phi^{e_i}_{0,0,0} = \frac{q_i}{\qq \pp},
    \quad
    \phi^{e_i}_{e_i,0,0} = \phi^{e_i}_{0,0,e_i} = 1,
    \quad\text{ and }\quad
    \phi^{e_i}_{0,e_i,0} = \frac{\qq}{\pp} ,
\]
which satisfy the decay condition of~\eqref{eq:phi_bounds}.

By induction, we assume the conclusion holds for all multi-indices $\beta$ with $|\beta| \leq n$. Assume now that $\beta'$ is such that $|\beta'| = n+1$, and let $\beta$ be such that $\beta' = \beta + e_m$
where $m = \max \lbrace j: (\beta')^j > 0 \rbrace$. This assumption on $m$ and $\beta$ is needed based on the order of the variables appearing in~\eqref{eq:Theta}. Then we have
\[
\partial_p^{\beta'} \int_{\R^3} \Gamma \mu \dd q = \sum_{\beta_1+\beta_2+\beta_3 \leq \beta} \partial_{p_m} \int_{\R^3} \Theta_{\beta_1} \Gamma \partial_q^{\beta_2} \partial_p^{\beta_3} \mu
\phi^\beta_{\beta_1,\beta_2,\beta_3} \dd q .
\]
Proceeding as in the base case, we have
\[
\begin{split}
\partial_p^{\beta'} \int_{\R^3} \Gamma \mu \dd q &= \sum \int_{\R^3} \Theta_{\beta_1} \Gamma \partial_{p_m} \left( \partial_q^{\beta_2} \partial_p^{\beta_3} \mu \phi^\beta_{\beta_1,\beta_2,\beta_3} \right) \dd q\\
&\quad\quad + \sum \int_{\R^3} \left(-\frac{\qq}{\pp} \partial_{q_m} + \Theta_{e_m} \right) \left(\Theta_{\beta_1} \Gamma \right)\partial_q^{\beta_2} \partial_p^{\beta_3} \mu \phi^\beta_{\beta_1,\beta_2,\beta_3} \dd q \\
&= \sum \int_{\R^3} \Theta_{\beta_1} \Gamma \partial_q^{\beta_2} \partial_p^{\beta_3+e_m} \mu \phi^\beta_{\beta_1,\beta_2,\beta_3} \dd q
+ \sum \int_{\R^3} \Theta_{\beta_1} \Gamma \frac{\qq}{\pp} \partial_q^{\beta_2+e_m} \partial_p^{\beta_3} \mu \phi^\beta_{\beta_1,\beta_2,\beta_3} \dd q\\
&\quad\quad + \sum \int_{\R^3} \Theta_{\beta_1 + e_m} \Gamma \partial_q^{\beta_2} \partial_p^{\beta_3} \mu \phi^\beta_{\beta_1,\beta_2,\beta_3} \dd q
+ \sum \int_{\R^3} \Theta_{\beta_1} \Gamma \partial_q^{\beta_2} \partial_p^{\beta_3} \mu \widetilde{\Theta}_m \phi^\beta_{\beta_1,\beta_2,\beta_3} \dd q \\
&= S_1 + S_2 + S_3 + S_4 ,
\end{split}
\]
where
\[
\widetilde{\Theta}_m := \partial_{p_m} + \frac{\qq}{\pp} \partial_{q_m} + \frac{q_m}{\qq \pp} .
\]
Note that we needed our assumption on $e_m$ in order to combine $\Theta_{e_m} \Theta_{\beta_1}$ into $\Theta_{\beta_1 + e_m}$ without producing commutators. We see then
that this can be combined into the form~\eqref{eq:IbP}, with corresponding coefficients. To verify that~\eqref{eq:phi_bounds} holds, we need to check that the bound applies to each
coefficient in the above sum.

The terms in the summation given by $S_1$ are indexed by $(\beta_1, \beta_2, \beta_3)$, whose sum is less than $\beta$. Because of the extra $p$-derivative hitting $\mu$, each such
term now contributes to the summation in~\eqref{eq:IbP} in the term indexed by $(\beta_1, \beta_2, \beta_3+e_m)$, whose sum is less than $\beta+e_m = \beta'$. Thus, $S_1$ contributes a coefficient of the form $\phi^\beta_{\beta_1,\beta_2,\beta_3}$ into the cumulative coefficient $\phi^{\beta'}_{\beta_1,\beta_2,\beta_3+e_m}$. By induction, we have
\[
    |\partial_q^{\nu_1} \partial_p^{\nu_2} \phi^\beta_{\beta_1,\beta_2,\beta_3} |
    \lesssim \qq^{|\beta_2|-|\nu_1|} \pp^{|\beta_1| + |\beta_3| - |\beta| - |\nu_2|}
    \approx \qq^{|\beta_2|-|\nu_1|} \pp^{|\beta_1| + |\beta_3+e_m| - |\beta'| -|\nu_2|},
\]
which is the desired upper bound, since $|\beta'| = |\beta| + 1$.

The terms in the summation given by $S_2$ contribute to the summation in~\eqref{eq:IbP} in the term indexed by $(\beta_1, \beta_2+e_m, \beta_3)$, and adds a term of the form
$\qq \phi^\beta_{\beta_1,\beta_2,\beta_3} / \pp$ into the cumulative coefficient $\phi^{\beta'}_{\beta_1,\beta_2+e_m,\beta_3}$. Inductively,
\[
    \left| \partial_q^{\nu_1} \partial_p^{\nu_2} \frac{\qq}{\pp} \phi^\beta_{\beta_1,\beta_2,\beta_3} \right|
    \lesssim \qq^{|\beta_2|-|\nu_1|+1} \pp^{|\beta_1| + |\beta_3| - |\beta|-|\nu_2|-1}
    \approx \qq^{|\beta_2+e_m|-|\nu_1|} \pp^{|\beta_1| + |\beta_3| - |\beta'|-|\nu_2|},
\]
which again has the correct decay rate.

The terms in $S_3$ contribute to~\eqref{eq:IbP} in the term indexed by $(\beta_1+e_m, \beta_2, \beta_3)$ and add a term of the form
$\phi^\beta_{\beta_1,\beta_2,\beta_3}$ into the cumulative coefficient $\phi^{\beta'}_{\beta_1+e_m,\beta_2,\beta_3}$. The upper bound~\eqref{eq:phi_bounds}
again follows inductively from
\[
    |\partial_q^{\nu_1} \partial_p^{\nu_2} \phi^\beta_{\beta_1,\beta_2,\beta_3} |
    \lesssim \qq^{|\beta_2|-|\nu_1|} \pp^{|\beta_1| + |\beta_3| - |\beta|-|\nu_2|}
    \approx \qq^{|\beta_2|-|\nu_1|} \pp^{|\beta_1+e_m| + |\beta_3| - |\beta'|-|\nu_2|} .
\]
Lastly, the terms in $S_4$ contribute to~\eqref{eq:IbP} at the index $(\beta_1, \beta_2, \beta_3)$ (this is the reason why the sum must have $\leq \beta$ and not $= \beta$, since
each integration by parts creates lower-order terms), and adds a term of the form $\widetilde{\Theta}_m \phi^\beta_{\beta_1,\beta_2,\beta_3}$ into the coefficient $\phi^{\beta'}_{\beta_1,\beta_2,\beta_3}$. We then finish our proof of~\eqref{eq:phi_bounds} by the inductive estimate
\[
\begin{split}
    \left| \partial_q^{\nu_1} \partial_p^{\nu_2} \left( \partial_{p_m} + \frac{\qq}{\pp} \partial_{q_m} + \frac{q_m}{\qq \pp} \phi^\beta_{\beta_1,\beta_2,\beta_3} \right) \right|
    &\lesssim \qq^{|\beta_2|-|\nu_1|} \pp^{|\beta_1| + |\beta_3| - |\beta| - |\nu_2| - 1} \\
    &\approx \qq^{|\beta_2|-|\nu_1|} \pp^{|\beta_1|+|\beta_3| - |\beta'|-|\nu_2|} .
\end{split}
\]
\end{proof}

Next, we need to control the factors of the form $\Theta_{\beta_i} \Gamma(p,q)$ that appear on the right in \Cref{p:strain-guo}, where $\Gamma$ is chosen as the integral kernels for the coefficients $a^g$, $b^g$, and $c^g$.

\begin{lemma}\label{l:Theta-on-kernels}
Let $\Phi^{ij}(p,q)$ be as defined in~\eqref{e.Phi-def}, let $G(p,q)$ be defined as in \Cref{l:G-upper}, and let $H_j(p,q)$ be defined in~\eqref{e.H-def}. Then
\begin{align}
    \left| \Theta_\beta \Phi^{ij}(p,q) \right|
    &\leq C \frac{\qq^7}{\pp^{|\beta|}} \left(1 + \frac 1 {|p-q|}\right),
    \label{eq:Theta_properties}
    \\
    \left| \Theta_\beta G(p,q) \right| 
    &\leq \frac C {\pp^\beta} \frac  1 {|p-q|},
    \quad\text{ and}
    \label{e.Theta-G}
    \\
 \left|\Theta_\beta H_j(p,q)\right| &\leq C\frac {\qq} {\pp^{|\beta|-1}} \frac 1 {|p-q|},
    \label{e.Theta-H}
\end{align}
for a constant $C>0$ independent of $p$ and $q$.
\end{lemma}

\begin{proof}
The first inequality~\eqref{eq:Theta_properties} follows directly from \cite[Lemma 2]{strain2004stability}, which says
\[
\left| \Theta_\beta \Phi^{ij}(p,q) \right| \leq C\pp^{-|\beta|} \qq^6 \ \ \ \text{ when } \ \ \ |p-q|+|p\times q| \geq \frac{|p|+1}{2}
\]
and
\[
    \left| \Theta_\beta \Phi^{ij}(p,q) \right| \leq C \qq^7 \pp^{-|\beta|} |p-q|^{-1} \ \ \ \text{ when } \ \ \ |p-q|+|p\times q| < \frac{|p|+1}{2} .
\]

To address $G(p,q)$, we first note that~\eqref{e.theta-pq} implies that
\[
    \Theta_{e_i}(\tau)
    = \Theta_{e_i}(\tau-1)
    = \Theta_{e_i}(\tau-2)
    = 0
\]
for each $i\in \{1,2,3\}$. With this, we have
\[
\Theta_\beta G(p,q) = \Theta_\beta \left( \frac{1}{\pp \qq} \frac{\tau -1}{\sqrt{\tau(\tau-2)}} \right) =
\frac{\tau-1}{\sqrt{\tau(\tau-2)}} \Theta_\beta \left(\frac 1{\pp \qq}\right),
\]
and a straightforward calculation using the definition~\eqref{eq:Theta} of $\Theta_\beta$ shows (as already pointed out in \cite{strain2004stability}) that
\[
    \left|\Theta_\beta \frac{1}{\pp \qq}\right|
    \lesssim  \frac{1}{\qq \pp^{1+|\beta|}}.
\]
Applying \Cref{l:G-upper} yields
\[
    |\Theta_\beta G(p,q)|
    \lesssim \frac 1 {\pp^\beta} G(p,q) 
    \lesssim \frac 1 {\pp^\beta} \frac  1 {|p-q|},
\]
as desired.

For $H_j(p,q)$, proceeding similarly to our analysis of $G(p,q)$, we have
\[
\Theta_\beta H_j(p,q) = \Theta_\beta \left( \frac{p_j + q_j}{\pp \qq} \frac{(\tau -1)^2}{\tau\sqrt{\tau(\tau-2)} } \right)
= \frac{(\tau - 1)^2}{\tau \sqrt{\tau (\tau-2)}} \Theta_\beta \left( \frac{p_j + q_j}{\pp \qq} \right) .
\]
By another direct calculation using~\eqref{eq:Theta}, we see that 
\[
\Theta_\beta\left(\frac{p_j + q_j}{\pp \qq}\right) \lesssim \frac{\pp + \qq} {\qq \pp^{1+|\beta|}},
\]
and we obtain, using $0\leq (\tau-1)/\tau \leq 1$,
\[
|\Theta_\beta H_j(p,q)|
    \lesssim \frac{\tau - 1}{\sqrt{\tau(\tau-2)}} \frac{\pp+\qq}{\qq\pp^{1+|\beta|}} = \frac {\pp+\qq} {\pp^{|\beta|}} G(p,q) \lesssim \frac {\pp\qq} {\pp^{|\beta|}} \frac 1 {|p-q|},
\]
as claimed.
\end{proof}

We are now in a position to estimate any derivative of $a^g$, $b^g$, and $c^g$:

\begin{proposition}\label{p:abc-deriv}
Assume that some partial derivative $\partial_t^m \partial_x^{\beta_x}\partial_p^{\beta_p}g$ of $g$ exists for all $(t,x,p)\in \Omega \times \R^3$, where $\Omega\subset\R^4$ is some open set. Furthermore, assume that for all multi-indices $\widetilde \beta \leq \beta_p$, there holds
\begin{equation}\label{e.dg-cond}
|\partial_t^m \partial_x^{\beta_x}\partial_p^{\widetilde\beta}g(t,x,p)| \leq C_k \pp^{-k},
\end{equation}
for some $k>10+|\beta_p|$, all $(t,x,p) \in \Omega\times\R^3$, and some uniform constant $C_k>0$.

Then the coefficients $a^g$, $b^g$, and $c^g$ satisfy
\begin{equation}
\begin{split}
\left|\partial_t^m \partial_x^{\beta_x} \partial_p^{\beta_p} a_{ij}^g(t,x,p)\right| &\leq C \pp^{-|\beta_p|},\\
\left|\partial_t^m \partial_x^{\beta_x} \partial_p^{\beta_p} b_{j}^g(t,x,p)\right| &\leq C \pp^{1-|\beta_p|},\\
\left|\partial_t^m \partial_x^{\beta_x} \partial_p^{\beta_p} c^g(t,x,p)\right| &\leq C \pp^{-|\beta_p|},
\end{split}
\end{equation}
for a constant $C>0$ depending only on $k$, $|\beta_p|$, and $C_k$.

If, in addition, 
\begin{equation}
\|\pp^\ell \partial_t^m \partial_x^{\beta_x} \partial_p^{\widetilde\beta} g\|_{C^\alpha_{\rm kin}(\Omega\times\R^3_p)} \leq C_\ell,
\end{equation}
for some $\ell > 10 + |\beta_p|$ and all $\widetilde \beta \leq \beta_p$, then
\begin{equation}\label{e.da-holder}
\begin{split}
\|\pp^{|\beta_p|} \partial_t^m \partial_x^{\beta_x} \partial_p^{\beta_p} a_{ij}^g(t,x,p)\|_{C^\alpha_{\rm kin}(\Omega\times \R^3_p)} \leq C,\\
\|\pp^{-1+|\beta_p|} \partial_t^m \partial_x^{\beta_x} \partial_p^{\beta_p} b_{j}^g(t,x,p)\|_{C^\alpha_{\rm kin}(\Omega\times \R^3_p)} \leq C,\\
\|\pp^{|\beta_p| } \partial_t^m \partial_x^{\beta_x} \partial_p^{\beta_p} c^g(t,x,p)\|_{C^\alpha_{\rm kin}(\Omega\times \R^3_p)} \leq C,\\
\end{split}
\end{equation}
for a constant $C>0$ depending only on $\ell$, $|\beta_p|$, $\alpha$, and $C_\ell$. 
\end{proposition}

\begin{proof}
Applying \Cref{p:strain-guo} with $\Gamma(p,q) = \Phi^{ij}(p,q)$ and \Cref{l:Theta-on-kernels},
\[
\begin{split}
  |\partial_t^m \partial_x^{\beta_x} \partial_p^{\beta_p} &a_{ij}^g(t,x,p)|
  = \left|\partial_p^{\beta_p} \int_{\R^3} \Phi^{ij}(p,q) \partial_t^m \partial_x^{\beta_x} g(t,x,q) \dd q\right|
  \\
    &= \left|\sum_{\beta_1+\beta_2\leq \beta_p} \int_{\R^3} \Theta_{\beta_1} \Phi^{ij}(p,q) \partial_t^m \partial_x^{\beta_x} \partial_q^{\beta_2} g(t,x,q) \phi^{\beta_p}_{\beta_1,\beta_2}(p,q) \dd q\right|
  \\
    &\lesssim \sum_{\beta_1 + \beta_2 \leq \beta_p} \int_{\R^3} \frac{\qq^7}{\pp^{|\beta_1|}} \left(1 + \frac 1 {|p-q|}\right) |\partial_t^m \partial_x^{\beta_x} \partial_q^{\beta_2}g(t,x,q)| \qq^{|\beta_2|} \pp^{|\beta_1|-|\beta_p|} \dd q.
\end{split}
\]
With our condition~\eqref{e.dg-cond} on $g$, this implies
\[
\left|\partial_t^m \partial_x^{\beta_x} \partial_p^{\beta_p} a_{ij}^g(t,x,p)\right| \lesssim \pp^{-|\beta_p|}\int_{\R^3} \qq^{7+|\beta_p|}\left(1 + \frac 1 {|p-q|}\right) \qq^{-k} \dd q \lesssim \pp^{-|\beta_p|},
\]
since $k>10+|\beta_p|$. 

To prove the H\"older estimate on $\partial_t^m \partial_x^{\beta_x} \partial_p^{\beta_p}a^g_{ij}$, let us first introduce the shorthand $\partial  = \partial_t^m \partial_x^{\beta_x} \partial_p^{\beta_p}$. If $p_i\in \R^3$ is a sequence such that the union of $B_1(p_i)$ covers $\R^3$, then
\begin{equation}
  \|\pp^\ell \partial g\|_{C^\alpha_{\rm kin}(\Omega\times\R^3)}
  \approx
  \sup_i \|\pp^\ell \partial g\|_{C^\alpha_{\rm kin}(\Omega\times B_1(p_i))}.
\end{equation}
Therefore, it suffices to estimate the difference $|\partial g(z) - \partial g (z_0)|$ under the assumption that $|p-p_0| \leq 1$, and therefore $\pp \approx \ppo$. For each $(i,j)$ and each pair $(\beta_1, \beta_2)$ with $\beta_1 + \beta_2 \leq \beta_p$, we define
\[
K_{\beta_1,\beta_2}(p,q) := \Theta_{\beta_1}\Phi^{ij}(p,q) \phi^{\beta_p}_{\beta_1,\beta_2}(p,q)
\]
Rather than apply \Cref{l:integral-holder} directly (which would require us to estimate $|\nabla_p K_{\beta_1,\beta_2}(p,q)|$ pointwise), we write
\[
\begin{split}
\partial a^g_{ij}(z) - \partial a^g_{ij}(z_0) &= [\partial a^g_{ij} (t,x,p) - \partial a^g_{ij} (t,x,p_0)] + [\partial a^g_{ij} (t,x,p_0)  - \partial a^g_{ij}(t_0,x_0,p_0)]\\
&= \int_0^1 (p-p_0) \cdot \nabla_p \partial a^g_{ij}(t,x, \xi p + (1-\xi) p_0) \dd \xi\\
&\quad + \sum_{\beta_1 + \beta_2 \leq \beta_p} \int_{\R^3} K_{\beta_1,\beta_2}(p_0,q) \left(\partial_t^m \partial_x^{\beta_x} \partial_p^{\beta_2} g(t,x,q) - \partial_t^m \partial_x^{\beta_x} \partial_p^{\beta_2}g(t_0,x_0,q)\right) \dd q\\
&=: J_1 + J_2.
\end{split}
\]
The second term $J_2$ is bounded exactly like the term $I_2$ in the proof of \Cref{l:integral-holder}, using \Cref{l:Theta-on-kernels} and the $C^\alpha_{\rm kin}$ bound for $\partial_t^m \partial_x^{\beta_x} \partial^{\beta_2}_p g$. This yields
\[
    J_2
    \lesssim \pp^{-|\beta_p|} [\pp^\ell \partial^{\beta_2} g]_{C^\alpha_{\rm kin}(\Omega\times \R^3)} d_{\rm kin}(z,z_0)^{\alpha} \lesssim \pp^{-|\beta_p|}\|\pp^\ell \partial^{\beta_p} g\|_{C^\alpha_{\rm kin}(\Omega\times\R^3)} d_{\rm kin}(z,z_0)^{\alpha} ,
\]
with $\ell > 9+|\beta_p| \geq 9 + |\beta_2|$, where we used interpolation (\Cref{l:interp}) in the last step.

To bound $J_1$, for each $k=1,2,3$ we let $\partial^{\beta_p'} = \partial_{p_k} \partial^{\beta_p}$ and use \Cref{p:strain-guo} to write
\[
\begin{split}
&(p-p_0)_k \partial_{p_k} \partial a^g_{ij}(t,x,\xi p + (1-\xi) p_0) \\
 &=  \sum_{\beta_1' + \beta_2'\leq \beta_p'} (p-p_0)_k \int_{\R^3} \Theta_{\beta'_1} \Phi^{ij}(\xi p + (1-\xi)p_0,q) \partial_t^m \partial_x^{\beta_x} \partial_q^{\beta_2'}g(t,x,q) \phi^{\beta_p'}_{\beta_1',\beta_2'}(\xi p + (1-\xi)p_0,q) \dd q
\end{split}
\]
From~\eqref{eq:phi_bounds2} and \Cref{l:Theta-on-kernels}, this expression is bounded above, up to a multiplicative constant, by
\[
\begin{split}
  |p-p_0| & \sum_{\beta_1' + \beta_2'\leq \beta_p'} \frac{1}{\pp^{|\beta_1'|}}\int_{\R^3} \qq^7\left( 1 + \frac 1 {|p-q|}\right) \partial_t^m \partial_x^{\beta_x} \partial_q^{\beta_2'} g(t,x,q)  \qq^{|\beta_2'|} \pp^{|\beta_1'|-|\beta_p'|} \dd q
  \\
  &\lesssim |p-p_0| \frac{1}{\pp^{|\beta_1'|}} \int_{\R^3} \left( 1 + \frac 1 {|p-q|}\right) \qq^{7+|\beta_2'|-\ell} \dd q
  \lesssim \frac{|p-p_0|}{\pp^{|\beta_p'|}},
\end{split}
\]
where we have used $\pp^\ell \partial_t^m \partial_x^{\beta_x} \partial_p^{\beta_2'} g \in L^\infty$,  with $\ell>10+|\beta_p|\geq 9 + |\beta_2'|$. Summing over $k=1,2,3$ and integrating over $\xi\in [0,1]$, we conclude 
\[
  J_1
  \lesssim \frac{|p-p_0|}{\pp^{|\beta_p'|}}.
\]
The estimates for $J_1$ and $J_2$ together imply the desired estimate~\eqref{e.da-holder} for $a^g_{ij}$, since $|\beta_p'| = |\beta_p|+1$.

The upper bounds and H\"older estimates for derivatives of $b^g$ and $c^g$ follow the same arguments as those for $a^g$, using \Cref{l:Theta-on-kernels} and the corresponding choices of $\Gamma(p,q)$ in \Cref{p:strain-guo}.
\end{proof}

\section{Proof of higher regularity}\label{s:higher}

In this last section, we prove our main regularity estimates. 
The following theorem proves that functions with {\it a priori} rapid decay are smooth. It will be combined with \Cref{p:decay} to prove our main theorem.
	
\begin{theorem}\label{t:local-reg}
Let $f$ be a solution to the relativistic Landau equation~\eqref{e.main} defined for $(t,x,p) \in [0,T]\times\R^6$. 
%Let $\Omega'$ be a subset of $\Omega$, such that one of the following cases holds:
%\begin{enumerate}
%	\item $\Omega = \Omega' = [0,T]\times\R^3_x$. %= [0,T]\times\R^3_x$ and $\Omega' = [\tau,T]\times\R^3_x$ for some $\tau\in (0,T)$.
%	
%	\item $\Omega$ is an open subset of $(0,\infty)\times\R^3_x$, and $\overline{\Omega'} \subset \Omega$.
%\end{enumerate}
Assume that there is $m_0>0$ such that
\begin{equation}\label{e.mass_below}
    \inf_{[0,T]\times\R^3} \int_{\R^3} f(t,x,p) \dd p \geq m_0,
\end{equation}
and that there is $k>0$ such that
\begin{equation}\label{e.moments}
  \sup_{[0,T]\times \R^6}\ \pp^k f(t,x,p)
    \leq M_k.
\end{equation}
Then, for any multi-index $\beta$ and any $\ell>0$, there is $k(\beta,\ell)$ such that, if $k\geq k(\beta,\ell)$, there holds
\[
  \|\pp^\ell\partial^\beta f\|_{L^\infty([\tau,T]\times \R^6)}
    \leq C,
\]
for any $\tau \in (0,T)$. 
The constant $C>0$ depends on $m_0$, $M_k$, $|\beta|$, and $\tau$.

In particular, if $f \in L^\infty_k([0,T]\times \R^6)$ for all $k>0$, then $f\in C^\infty([\tau,T]\times \R^6)$ for any $\tau\in (0,T)$.
\end{theorem}

\begin{proof}
We break the proof down into several steps, that we signpost in bold.  
Throughout the proof, the parameter $k$ in~\eqref{e.moments} %, which is the number of $L^\infty$-polynomial moments of $f$ that we assume are finite, 
will be taken large enough for each step to be valid. We absorb the quantities $\|f\|_{L^\infty_k([0,T]\times\R^6)}$ into constants without comment.

\medskip
\noindent
{\bf \#\ Step one: weighted H\"older regularity.} 
From \Cref{p:A-elliptic} and our assumptions~\eqref{e.mass_below} and~\eqref{e.moments}, we see that $a^f$ satisfies uniform upper and lower ellipticity conditions of the form
\[
    \lambda I \leq a^f(z) \leq \Lambda I,
\]
for all $z\in [0,T]\times\R^3$. From \Cref{l:c-upper,l:other-B-upper}.(b), we also have upper bounds on $B^f$ and $c^f$ in $L^\infty([0,T]\times\R^3)$. These estimates for the coefficients $a^f$, $B^f$, $c^f$ depend only on $m_0$ and the $L^\infty_5$-norm of $f$. Letting $s = c^f f$, we also have $s\in L^\infty(\Omega\times\R^3)$.  Therefore, \Cref{l:Calpha} implies
\[
 \|f\|_{C^\alpha_{\rm kin}([\tau,T]\times\R^6)} \leq C_\tau\left( \|\pp f\|_{L^\infty([0,T]\times\R^6)} + \|\pp^{-3} s\|_{L^\infty([0,T]\times\R^6)}\right) \leq C_\tau,
\]
for some $\alpha \in (0,1)$. Throughout this proof, $C_\tau$ denotes a constant that depends only on the quantities in the statement of our theorem, and that degenerates as $\tau\to 0$ at some polynomial rate, which may change line-by-line.

By interpolation (\Cref{l:holder-decay}) we can decrease $\alpha$ (but still denote it with the same variable) and obtain
\[
\|\pp^{m_1} f\|_{C^\alpha_{\rm kin}([\tau,T]\times\R^6)}  \leq C_\tau,
\]
for any $m_1>0$ and $\tau\in (0,T)$, provided $k$ is large enough.

\medskip
\noindent
{\bf \#\  Step two: initial regularity via the first Schauder estimate.}
Since $f$ is H\"older continuous, we can choose $m_1$ large enough and pass this weighted H\"older regularity to $a^f$, $b^f$, and $c^f$ via \Cref{l:c-holder}, \Cref{l:b-holder}, and \Cref{l:a-holder}. By \Cref{l:holder-product}, this also gives control on the $C^\alpha_{\rm kin}([\tau,T]\times\R^6)$-norm of $s =c^f f$, with constants depending only on $\alpha$ and the quantities in the statement of our theorem. It is clear that these H\"older bounds on the coefficients also degenerate as $\tau\to 0$ at a polynomial rate.

Applying our first Schauder estimate~\eqref{e.first-schauder} from \Cref{t:first_schauder}, we obtain
\[
       [f]_{C^{2+ \frac{4\alpha}9}_{\rm kin}([\tau,T]\times\R^6)}
        \leq C_\tau,
\]
for any $\tau\in (0,T)$. 
Interpolating again with \Cref{l:holder-decay}, we conclude that
\begin{equation}\label{e.step-two}
     [\pp^{m_2} f]_{C^{2+ \alpha}_{\rm kin}([\tau,T]\times \R^6)} 
    \leq C_\tau,
\end{equation}
for any $m_2>0$, after choosing $k$ large enough (depending on $m_2$) and decreasing $\alpha$.

\medskip
\noindent
{\bf \#\  Step three: slightly higher regularity via the second Schauder estimate.}
For this, we need better regularity of the coefficients. In particular, recalling the definition of the $C^{1+\alpha}_{\rm kin}$ seminorm in \eqref{e.1alpha}, we see via standard interpolations that the regularity in \eqref{e.step-two} controls the $C^{1+\alpha}_{\rm kin}$ seminorm of 
 the coefficients $a^f$, $b^f$, and $c^f$, for some $\alpha$.

Applying our second Schauder estimate~\eqref{e.second-schauder} in \Cref{t:first_schauder}, and interpolating as above, we obtain
\begin{equation}\label{e.step-three}
    [\pp^{m_3} f]_{C^{3+\alpha}_{\rm kin}([\tau,T]\times\R^6)}
    \leq C_\tau,
\end{equation}
for some $\alpha>0$ and any desired $m_3>0$. 
Note that the regularity in \eqref{e.step-three} controls at least one derivative in each variable $(t,x,p)$.

\medskip
\noindent
{\bf \#\  Step four: higher regularity via induction.} 
Next, let us assume by induction that for some $j \geq 1$ and all $\tau \in (0,T)$,
\begin{equation}\label{e.induction}
\|\pp^{k-n(j)}\partial^\beta f\|_{C^{\alpha(j)}_{\rm kin}([\tau,T]\times \R^6)} \leq C_\tau K_j,
\end{equation}
for all multi-indices $\beta$ with $|\beta|\leq j$. Here, $\alpha(j) \in (0,1)$ and $n(j)\geq 0$ depend on $j$, but in the remainder of the proof, we sometimes write $\alpha =\alpha(j)$ and $n = n(j)$ to keep the notation clean. As above, $C_\tau$ is a constant that degenerates polynomially as $\tau\to 0$, at a rate that may depend on $j$.

Let $\beta$ be a muti-index of order $|\beta| = j$. 
 Differentiating the relativistic Landau equation~\eqref{e.main}, with $Q_{RL}(f,f)$ written in nondivergence form as in~\eqref{e.nondivergence}, by $\partial^\beta$, we obtain
\begin{equation}\label{e.partial-f-eqn}
\begin{split}
\left(\partial_t + \frac p \pp \cdot \nabla_x\right) \partial^{\beta} f = \tr(a^f D_p^2 \partial^{\beta} f) + b^f \cdot \nabla_p \partial^{\beta} f + c^f \partial^{\beta} f + R_{\beta}(f),
\end{split}
\end{equation}
with 
\[
  R_{\beta}(f)
    = \sum_{\substack{\beta_1 + \beta_2 = \beta \\ |\beta_1| > 0}}
    \left(- \partial^{\beta_1} \left( \frac p \pp\right) \cdot \nabla_x \partial^{\beta_2} f  
    + \tr( \partial^{\beta_1} a^f D_p^2 \partial^{\beta_2} f)
    + \partial^{\beta_1} b^f \cdot \nabla_p \partial^{\beta_2} f
    + \partial^{\beta_1} c^f \partial^{\beta_2} f \right).
\]
 If $j > 1$, then since $|\beta_2|\leq |\beta|-1 = j -2$, all of the derivatives in $R_{\beta}(f)$ have order at most $j$. If $j = 1$, then the sum in $R_{\beta}(f)$ has only one term, with $\beta_2 = 0$ and $|\beta_1| = j$ in this case as well. Also, if $k$ is sufficiently large, the assumption~\eqref{e.induction} implies, via \Cref{p:abc-deriv}, that for each $\beta_1$ and $\beta_2$ appearing in $R_{\beta}(f)$, which have order at most $j$, there holds for $i=1,2$ 
\[
\begin{split}
  &\|\pp^{|\beta_i|} \partial^{\beta_i} a^f \|_{C^\alpha_{\rm kin}([\tau,T]\times\R^6)}
  +\|\pp^{-1+|\beta_i|} \partial^{\beta_i} b^f \|_{C^\alpha_{\rm kin}([\tau,T]\times\R^6)}
  \\
  &\ +
  \|\pp^{|\beta_i|} \partial^{\beta_i} c^f \|_{C^\alpha_{\rm kin}([\tau,T]\times\R^6)}
  \lesssim_{K_j} C_\tau.
\end{split}
\]
This estimate applies to all $\partial^{\beta_1}$ and $\partial^{\beta_2}$ appearing in $R_{\beta}(f)$.  We therefore have
\[
	\| \pp^{k - \ell_1} R_{\beta}(f) \|_{C^{\alpha}_{\rm kin}([\tau,T]\times\R^6)}
		\lesssim_{K_j} C_\tau,
\]
for some $\ell_1>0$ depending on $n$ and $j$ and with the implied constant above depending on $K_j$. 

At this point, we may deduce regularity properties of $\partial^{\beta} f$ by proceeding in a similar way to the base case of our induction. The steps are as follows: we apply our first Schauder estimate~\eqref{e.first-schauder} to equation~\eqref{e.partial-f-eqn} and interpolate via \Cref{l:holder-decay}, yielding
\[
  [\pp^{k-\ell_2}\partial^{\beta} f]_{C^{2+\alpha}_{\rm kin}([\tau,T]\times\R^6)}
    \leq C_\tau,
\]
for some $\ell_2>0$ and $\alpha\in (0,1)$. This regularity is passed to the coefficients via \Cref{p:abc-deriv}, which gives
\[
  [\partial^{\beta} a^f]_{C^{1+\alpha}_{\rm kin}([\tau,T]\times\R^6)}
    \leq C_\tau,
\]
where $\alpha$ has decreased again. By interpolation (\Cref{l:interp}), the same estimate holds with $\partial^{\beta}$ replaced by $\partial^{\widetilde\beta}$ for any $\widetilde\beta \leq \beta$. Estimating the other coefficients $b^f$ and $c^f$ similarly, we therefore have a bound for $R_{\beta}(f)$ in $C^{1+\alpha}_{\rm kin}([0,T]\times\R^6)$, which allows us to apply our second Schauder estimate~\eqref{e.second-schauder}, and we conclude, after interpolating as above,
\[
    [\pp^{k-\ell_3}\partial^{\beta} f]_{C^{3+\alpha}_{\rm kin}([\tau,T]\times\R^6)}
        \lesssim C_\tau,
\]
for some $\ell_3>0$. Since $\beta$ was an arbitrary multi-index of order $j$, and the $C^{3+\alpha}_{\rm kin}$-seminorm controls at least one derivative in each variable, we have shown 
\[
[\pp^{k-n(j+1)} \partial^{\beta'} f]_{C^{\alpha(j+1)}_{\rm kin}([\tau,T]\times\R^6)} \leq K_{j+1}C_\tau,
\]
for any multi-index $\beta'$ of order $j+1$, allowing us to close the induction and complete the proof.
\end{proof}

Finally, our main regularity result, \Cref{t:main}, follows immediately from \Cref{t:local-reg} and the decay estimates of \Cref{t:decay}.

\appendix

\section{Properties of H\"older norms}\label{s:appendix}

In this appendix, we collect some useful properties of the H\"older spaces $C^\alpha_\cG$ and $C^\alpha_{\rm kin}$. 
In the first two lemmas, we examine how the two norms are related under the change of variables $\varphi(p)$ defined in~\eqref{e.phi_psi}.

\begin{lemma}\label{l:dist-compare}
    Let $z= (t,x,p) \in \R^7$. Then
    \[
d_{\cG}((t,x,\varphi(p)),(t_0,x_0,\varphi(p_0))) \leq d_{\rm kin}(z,z_0) + d_{\rm kin}(z,z_0)^{\sfrac{2}{3}},
\]
and
\[
d_{\rm kin}(z,z_0) \leq \max\{\pp^3,\ppo^3\}d_{\cG}((t,x,\varphi(p)),(t_0,x_0,\varphi(p_0))) + d_{\cG}((t,x,\varphi(p)),(t_0,x_0,\varphi(p_0)))^{\sfrac{2}{3}},
\]
where $d_{\cG}$ is defined in \eqref{e.dG-def} and $d_{\rm kin}$ is defined in \eqref{e.dk}.
\end{lemma}

\begin{proof}
Recall that
\begin{equation}\label{e.dGphi}
\begin{split}
d_{\mathcal G}((t,x,\varphi(p)), (t_0,x_0,\varphi(p_0))) &= |t-t_0|^{\sfrac{1}{2}} + \left|x-x_0 - (t-t_0)\frac {p_0}{\ppo} \right|^{\sfrac{1}{3}} + \left|\frac p {\pp} - \frac {p_0}{\ppo}\right|.
\end{split}
\end{equation}
Beginning with the middle term, using $|a+b|^{\sfrac13} \leq |a|^{\sfrac13} + |b|^{\sfrac13}$, we have 
\[
\left|x-x_0 +(t-t_0)\frac {p_0}\ppo\right|^{\sfrac{1}{3}} \leq |x-x_0|^{\sfrac{1}{3}} + |t-t_0|^{\sfrac{1}{3}} \leq |x-x_0|^{\sfrac{1}{3}} + d_{\rm kin}(z,z_0)^{\sfrac{2}{3}}, 
\]
and
\[
\begin{split}
|x-x_0|^{\sfrac{1}{3}} %= \left|x-x_0 + (t-t_0)\frac{p_0}\ppo - (t-t_0)\frac{p_0}\ppo\right|^{\sfrac{1}{3}} 
&\leq \left|x-x_0 + (t-t_0)\frac{p_0}\ppo \right|^{\sfrac{1}{3}} + |t-t_0|^{\sfrac{1}{3}} \\
&\leq \left|x-x_0 + (t-t_0)\frac{p_0}\ppo \right|^{\sfrac{1}{3}} + d_{\cG}((t,x,\varphi(p)),(t_0,x_0,\varphi(p_0)))^{\sfrac{2}{3}}. 
\end{split}
\]
For the last term in \eqref{e.dGphi}, we recall that
\[
D\varphi(p) = \frac 1 \pp \left( \Id - \frac{p\otimes p}{\pp^2} \right),
\]
and since $|D\varphi(p)| \leq 1$ for all $p$, one has 
\[
\left|\frac p {\pp} - \frac{p_0} \ppo \right| \leq |p-p_0|.
\]
A direct calculation shows that
\[
    (D\varphi(p))^{-1} = \pp(\Id + p\otimes p).
\]
Recall that $\psi$, defined in~\eqref{e.phi_psi}, is the inverse function of $\phi$. Hence,
\[
D\psi(v) = D\varphi(\psi(v))^{-1} = \frac 1 {\sqrt{1-|v|^2}} \left( \Id + \frac{v_iv_j}{1-|v|^2}\right).
\]
We have $|D\psi(v)| \leq (1-|v|^2)^{-\sfrac32}$, and therefore,
\[
\begin{split}
|p-p_0| = |\psi(v) - \psi(v_0)| &\leq \max\{(1-|v|^2)^{-\sfrac32}, (1-|v_0|^2)^{-\sfrac32}\} |v-v_0|\\
&= \max\{\pp^3, \ppo^3\}\left|\frac p \pp - \frac{p_0}\ppo\right|.
\end{split}
\]
Combining our estimates, we conclude
\[
d_{\cG}((t,x,\varphi(p)),(t_0,x_0,\varphi(p_0)) \leq d_{\rm kin}(z,z_0) + d_{\rm kin}(z,z_0)^{\sfrac{2}{3}},
\]
and
\[
d_{\rm kin}(z,z_0) \leq \max\{\pp^3,\ppo^3\}d_{\cG}((t,x,\varphi(p)),(t_0,x_0,\varphi(p_0)) + d_{\cG}((t,x,\varphi(p)),(t_0,x_0,\varphi(p_0))^{\sfrac{2}{3}},
\]
 as claimed.
 \end{proof}

\begin{lemma}\label{l:holder-compare}
    Let $h(t,x,p)$ be defined on $\Omega$, where 
    \[
    \Omega \subset \{(t,x,p) \in \R^7 : |p|< R\}.
    \]
    Define
    \[
    \widetilde h(t,x,v) = h(t,x,\psi(v)),
    \]
    and let $\widetilde \Omega = \{(t,x,\varphi(p)) : (t,x,p) \in \Omega\}$.
    
    Then, for any $\alpha\in (0,\sfrac 2 3)$,
    \[
    \begin{split}
        [h]_{C^\alpha_{\rm kin}(\Omega)} &\leq C\|\widetilde h\|_{C^{\frac{3\alpha}2}_\cG(\widetilde \Omega)}
    \qquad\text{ and }\qquad
    [\widetilde h]_{C^\alpha_\cG(\widetilde \Omega)} \leq C \langle R\rangle^{\frac{9\alpha}2}\|h\|_{C^{\frac{3\alpha}2}_{\rm kin}(\Omega)},
    \end{split}
    \]
    whenever the right-hand sides are finite, where the constant $C>0$ depends only on $\alpha$.
\end{lemma}
\begin{proof}
Let $z_i = (t_i,x_i,p_i)$. If $d_{\rm kin}(z_1,z_2) < 1$, then \Cref{l:dist-compare} implies
\[
d_\cG((t_1,x_1,\varphi(p_1)),(t_2,x_2,\varphi(p_2))) \leq 2d_{\rm kin}(z_1,z_2)^{\sfrac{2}{3}},
\]
and, thus,
\[
    \frac{|h(z_1) - h(z_2)|} {d_{\rm kin}(z_1,z_2)^\alpha}
    \leq
        2^{\frac{3\alpha}2}\frac{|\widetilde h(t_1,x_1,\varphi(p_1)) - \widetilde h(t_2,x_2,\varphi(p_2))|}{d_\cG((t_1,x_1,\varphi(p_1)),(t_2,x_2,\varphi(p_2)))^{\frac{3\alpha}2}}
    \leq
        2^{\frac{3\alpha}2}[\widetilde h]_{C^{\frac{3\alpha}2}_\cG(\widetilde\Omega)}.
    \]
    If $d_{\rm kin}(z_1,z_2) \geq 1$, then
\[
\frac{|h(z_1) - h(z_2)|} {d_{\rm kin}(z_1,z_2)^\alpha} \leq 2\|h\|_{L^\infty(\Omega)}
    = 2\|\widetilde h\|_{L^\infty(\widetilde \Omega)}.
\]

We now consider the reverse inequality, which proceeds similarly. If $d_\cG((t_1,x_1,\varphi(p_1)),(t_2,x_2,\varphi(p_2))) < 1$, then \Cref{l:dist-compare} implies
\[
d_{\rm kin}(z_1,z_2) \leq 2\langle R+1\rangle^3 d_\cG((t_1,x_1,\varphi(p_1)),(t_2,x_2,\varphi(p_2))),
\]
and, thus,
\[
    \frac{|\widetilde h(t_1,x_1,\varphi(p_1)) - \widetilde h(t_2,x_2,\varphi(p_2))|} {d_\cG((t_1,x_1,\varphi(p_1)),(t_2,x_2,\varphi(p_2)))^\alpha}
    \leq 2^{\frac{3\alpha}2}\langle R+1\rangle^{\frac{9\alpha}2}\frac{|h(z_1) - h(z_2)|}{d_{\rm kin}(z_1,z_2)^{\frac{3\alpha}2}}
    \leq 2^{\frac{3\alpha}2} \langle R+1\rangle^{\frac{9\alpha}2} [h]_{C^{\frac{3\alpha}2}_k(\Omega)}.
\]
If $d_\cG((t_1,x_1,\varphi(p_1)),(t_2,x_2,\varphi(p_2))) \geq 1$, then
\[
    \frac{|\widetilde h(t_1,x_1,\varphi(p_1)) - \widetilde h(t_1,x_2,\varphi(p_2))|} {d_\cG((t_1,x_1,\varphi(p_1)),(t_2,x_2,\varphi(p_2)))^\alpha}
    \leq 2\|\widetilde h\|_{L^\infty(\widetilde\Omega)} 
    = 2\| h\|_{L^\infty(\Omega)}.
\]
This completes the proof.
\end{proof}

Interpolation inequalities, as in the following lemma, are a common tool for proving regularity in H\"older spaces. We omit the proof, which is standard.

\begin{lemma}\label{l:interp}
For any $\alpha\in (0,1)$ and $\eps>0$, the following inequalities hold whenever the right-hand sides are finite:
\begin{align}
\|D_v^2 g\|_{L^\infty(\Omega)} 
&\leq \eps^\alpha [g]_{C^{2+\alpha}_{\mathcal G}(\Omega)} + \frac{C}{\eps^2} \|g\|_{L^\infty(\Omega)},
\\
[\nabla_v g]_{C^\alpha_{\mathcal G}(\Omega)} 
&\leq \eps [g]_{C^{2+\alpha}_{\mathcal G}(\Omega)} + \frac{C}{\eps^{1+\alpha}}\|g\|_{L^\infty(\Omega)},
\\
[\nabla_v g]_{C^\alpha_{\mathcal G}(\Omega)} 
&\leq \eps^{1-\alpha} \|D_p^2g\|_{L^\infty(\Omega)} + \frac{C}{\eps^{1+\alpha}}\|g\|_{L^\infty(\Omega)},
\\
\|\nabla_v g\|_{L^\infty(\Omega)} 
&\leq \eps^{1+\alpha} [g]_{C^{2+\alpha}_{\mathcal G}(\Omega)} + \frac{C}{\eps}\|g\|_{L^\infty(\Omega)},
\\
[g]_{C^\alpha_{\mathcal G}(\Omega)} 
&\leq \eps^2 [g]_{C^{2+\alpha}_{\mathcal G}(\Omega)} + \frac{C}{\eps^\alpha}\|g\|_{L^\infty(\Omega)},\\
[\nabla_v g]_{C^\alpha_{\mathcal G}(\Omega)} 
&\leq \eps^2 [g]_{C^{3+\alpha}_{\mathcal G}(\Omega)} + \frac{C}{\eps^{1+\alpha}}\|g\|_{L^\infty(\Omega)},\\
\|\nabla_v g\|_{L^\infty(\Omega)} 
&\leq \eps^{2+\alpha} [g]_{C^{3+\alpha}_{\mathcal G}(\Omega)} + \frac{C}{\eps}\|g\|_{L^\infty(\Omega)},\\
[D_v^2 g]_{C^\alpha_{\mathcal G}(\Omega)} 
&\leq \eps [g]_{C^{3+\alpha}_{\mathcal G}(\Omega)} + \frac{C}{\eps^{2+\alpha}}\|g\|_{L^\infty(\Omega)},
\end{align}
where $C>0$ is a constant depending only on $\alpha$.

Furthermore, the same estimates hold if the Gallilean H\"older norms $C^\gamma_{\mathcal G}$ are replaced with the  H\"older norms $C^\gamma_{\rm kin}$ (in which case gradients in $v$ should be replaced by gradients in $p$). 
\end{lemma}

The next lemma (which we need to use only in $C^\alpha_{\rm kin}$ spaces) provides a form of interpolation in which we can trade regularity for decay. The proof is the same as \cite[Lemma B.2]{HST2019rough} and is omitted.

\begin{lemma}\label{l:holder-decay}
	For some $\Omega \subset \R^7$, suppose that $\phi:\Omega\to \R$ is such that $\phi\in L^{\infty}_{m_1}(\Omega)$ and $\pp^{m_2}\phi\in C^\alpha_{\rm kin}(\Omega)$, for some $\alpha\in (0,1)$ and $m_1 \geq m_2 \geq 0$.  If $\beta \in (0,\alpha)$ and $\ell \in [m_2, m_1]$ are such that
	\[
	\ell
	\leq m_1 \left(1 - \frac{\beta}{\alpha}\right) + m_2 \frac{\beta}{\alpha},
	\]
	then
	\[
	\|\langle p\rangle^\ell \phi\|_{C^{\beta}_{\rm kin}(\Omega)}
	\lesssim [\pp^{m_2}\phi]_{C^\alpha_{\rm kin}(\Omega)}^\frac{\beta}{\alpha} \|\phi\|_{L^{\infty}_{m_1}(\Omega)}^{1-\frac{\beta}{\alpha}}.%\|\langle v \rangle^{k_2} \phi\|_{\Ckin^\alpha(\R^6)} + \|\phi\|_{L^{\infty,k_1}(\R^6)}.
	\]
\end{lemma}

We also need the following simple adaptation of a standard lemma, whose proof is omitted.

\begin{lemma}\label{l:holder-product}
For any $\alpha\in (0,1)$, $\Omega\subset \R^7$, and any $f,g \in C^\alpha_{\mathcal G}(\Omega)$, there holds
\[
[f g]_{C^\alpha_{\mathcal G}(\Omega)} \leq \|f\|_{L^\infty(\Omega)} [g]_{C^\alpha_{\mathcal G}(\Omega)} + [f]_{C^\alpha_{\mathcal G}(\Omega)}\|g\|_{L^\infty(\Omega)}.
\]
Moreover, the same result holds if $C^\alpha_{\mathcal G}$ is replaced by $C^\alpha_{\rm kin}$. 
\end{lemma}

\bibliographystyle{abbrv}
\bibliography{relativistic_landau}

\end{document}